\documentclass[12pt,a4paper,reqno]{amsart}
\usepackage{graphicx}
\usepackage[scanall]{psfrag}
\usepackage[varg]{txfonts}
\usepackage{times}
\numberwithin{equation}{section}
\numberwithin{figure}{section}

\newtheorem{thm}{Theorem}

\newtheorem{thrm}{Theorem}[section]
\newtheorem{prop}[thrm]{Proposition}

\newtheorem{lemma}[thrm]{Lemma}
\newtheorem{cor}[thm]{Corollary}
\newtheorem{coro}[thrm]{Corollary}

\theoremstyle{definition}
 \newtheorem*{dfn}{Definition}
 
 \newtheorem{expl}[thrm]{Example}

\theoremstyle{remark}
 \newtheorem*{rmk}{Remark}
 \newtheorem{rem}[thrm]{Remark}

\renewcommand{\epsilon}{\varepsilon}
\newcommand{\rgt}[1]{\right#1}
\newcommand{\lft}[1]{\left#1}
\newcommand{\rd}{\partial}
\newcommand{\R}{\mathbb{R}}
\newcommand{\C}{\mathbb{C}}
\newcommand{\Z}{\mathbb{Z}}
\newcommand{\inter}{\operatorname{int}}

\newcommand{\uv}[1]{\frac{\rd}{\rd #1}}
\newcommand{\uvv}[1]{\rd/\rd #1}
\newcommand{\intp}{\lrcorner}
\newcommand{\ostr}[1][2n-1]{\R^{#1}\times S^1}
\newcommand{\skor}{{\perp'}}
\newcommand{\csn}{\textup{CSN}}
\newcommand{\nul}[1]{\operatorname{Null}(#1)}
\newcommand{\omtw}{\omega_{\mathtt{tw}}}
\newcommand{\spn}[1]{\operatorname{Span}\{#1\}}
\newcommand{\rsurg}[1]{\mathbin{\tilde\#_{\{#1\}}}}
\newcommand{\du}[1][\sqrt{\pi}]{\operatorname{DU}(#1)}
\newcommand{\mlt}[1][\pi]{\operatorname{LU}(#1)}
\newcommand{\wgd}[1][\pi]{\operatorname{WGD}(#1)}
\newcommand{\dcyl}{\Delta_{\textup{cyl}}}

\begin{document} 
%
%
\title[Generalizations of twists of contact structures]
{Generalizations of twists of contact structures \\ to higher-dimensions 
via round surgery}
\author{Jiro ADACHI}
\address{Department of Mathematics, Hokkaido University, Sapporo, 
  060-0810, Japan.}
\email{j-adachi@math.sci.hokudai.ac.jp}
\dedicatory{
\mbox{}}
\thanks{This work was supported by JSPS 
  KAKENHI Grant Number~25400077.}

\begin{abstract} 
  Twists of contact structures in dimension $3$ and higher are studied 
in this paper from a viewpoint of contact round surgery. 
 Three kinds of new modifications of contact structures 
which are higher-dimensional generalizations 
of the $3$-dimensional Lutz twists are introduced. 
 One of the operations makes a contact manifold overtwisted in any dimension.
 Another one makes it weakly symplectically non-fillable. 
 And the other one makes it strongly symplectically non-fillable. 
 In other words, they make the overtwisted disc, 
the bordered Legendrian open books, and  the Giroux domains, respectively. 
 The first two modifications can be realized by sequences of 
contact round surgeries. 
 The first operation can be applied anywhere of any contact manifold easily. 
 Further, a version of the first operation 
keeps the homotopy classes of contact structures as almost contact structures, 
and the other version contributes to the Euler classes of contact structures. 
\end{abstract}

\maketitle

\section{Introduction}\label{sec:intro}
  Classification and construction of contact structures and contact manifolds 
have been good and important problems in differential topology. 
 In these few decades, 
$3$-dimensional contact topology has developed drastically. 
 Compared with this, contact topology in higher dimensions 
has not been studied very much. 
 However, in recent few years, there are some remarkable movements 
on contact topology in higher dimensions. 
 Now, it is expected that considering from the unified perspective 
would give both $3$-dimensional and higher-dimensional contact topology 
good influences. 

  In this paper, we discuss certain torsions of contact structures 
in general dimensions. 
 There are three kinds of notions concerning torsion of contact structures: 
tightness, week and strong symplectic fillabilities. 
 The following inclusion relations are known for these notions: 
\begin{equation*}
  \{\text{strongly fillable}\}\subset\{\text{weekly fillable}\}
  \subset\{\text{tight}\}. 
\end{equation*}
 We discuss some higher-dimensional generalizations 
of the Lutz twist and the Giroux torsion. 
 In other words, three kinds of modifications of contact structures 
each of which obstructs each notion above are introduced in this paper. 
 This is the first observation of such operations 
for tightness or overtwistedness in higher dimensions. 
 Throughout this paper, 
observations from the view point of contact round surgery lie 
on the basement. 
 This may be one of the unified perspectives 
to observe torsions of contact structures, 
which leads us to a new direction. 

  The contact round surgery is a notion introduced by the author 
in \cite{art19} (see Section~\ref{sec:ctrdsurg} for definition). 
 The operation is defined for contact manifolds in any odd dimension. 
 The round surgery as a modification of manifolds is introduced 
by Asimov~\cite{asimov} to study non-singular Morse-Smale flows. 
 Not only for this purpose, round handle and round surgery 
have been used for several aspects. 
 Especially, the round surgery seems to get along well with contact structures. 
 There has been some attempts to apply the method to contact topology 
(see~\cite{art18}, \cite{art20}). 
 As a useful tool for the study of contact structures in general dimensions,
some applications of this method are expected. 

  Two important notions concerning torsion of contact structures, 
overtwistedness and the Giroux torsion, 
are introduced for contact structures on $3$-dimensional manifolds. 
 A \emph{contact structure}\/ is a completely non-integrable hyperplane field 
on an odd-dimensional manifold. 
 A contact structure $\xi$ on a $3$-dimensional manifold $M$ 
is said to be \emph{overtwisted}\/ 
if there exists an embedded disk $D\subset M$ which is tangent to $\xi$ 
along the boundary $\rd D$, that is, $T_{x}D=\xi_x$ at any $x\in\rd D$ 
(see Subsection~\ref{sec:ovtw} for precise definition). 
 Such a disk is called an \emph{overtwisted disk}. 
 It is known that, if a contact structure is overtwisted, 
then the contact manifold can not be the boundary 
of a compact symplectic manifold, in a weak sense (see \cite{geitext}). 
 Similarly, a contact structure $\xi$ is said 
to have the \emph{Giroux torsion}\/ at least $n\in\mathbb{N}$ 
if there exists a contact embedding 
$f_n\colon(T^2\times[0,1],\zeta_n)\to(M,\xi)$, 
where $\zeta_n:=\ker\{\cos(2n\pi r)d\theta+\sin(2n\pi r)d\phi\}$ 
with coordinates $(\phi,\theta,r)\in T^2\times[0,1]$. 
 It is known that, 
if a contact structure has the Giroux torsion greater than $0$, 
then the contact manifold can not be the boundary 
of a compact symplectic manifold in a strong sense (see \cite{gay}). 

  Some candidates of generalizations of the notions above has been proposed. 
 Recently, when this paper was being prepared, 
a new notion of overtwisted disc in all dimensions was announced 
by Borman, Eliashberg, and Murphy~\cite{newOT}. 
 It is stronger than other notions. 
 We discuss this new notion in Section~\ref{sec:newot}. 
 First, our discussion is based on the following notions 
introduced by Massot, Niederkr\"uger, and Wendl~\cite{mnw}.  
 As a higher-dimensional generalization of an overtwisted disk, 
a notion of \emph{bordered Legendrian open book}\/ (\textsf{bLob} for short) 
is introduced (see \cite{niedPS}, \cite{mnw}, \cite{gromov}). 
 Roughly speaking, it is an $(n+1)$-dimensional open book 
embedded in a $(2n+1)$-dimensional contact manifold $(M,\xi)$ 
whose pages and boundary are Legendrian 
(see Subsection~\ref{sec:bLob_Gdom} for precise definition). 
 It was proved by them that, if there exists a bordered Legendrian open book, 
then the contact manifold can not be the boundary 
of a compact symplectic manifold in a weak sense 
(see Theorem~\ref{thm:mnw_wkfill} for precise statement). 
 In this sense, a contact manifold which admits a bordered Legendrian open book
is said to be \emph{PS-overtwisted}. 
 In addition, 
as a higher-dimensional generalization of the Giroux $\pi$-torsion domain, 
a notion of \emph{Giroux domain}\/ is also introduced in~\cite{mnw}. 
 Roughly, a Giroux domain is a contactization 
of a certain $2n$-dimensional symplectic manifold 
with contact-type boundary in a $(2n+1)$-dimensional contact manifold 
(see Subsection~\ref{sec:bLob_Gdom} for precise definition). 
 It was proved in \cite{mnw} 
that, if there exists a domain that consists of two Giroux domains 
glued together, 
then the contact manifold can not be the boundary 
of a compact symplectic manifold in a strong sense 
(see Theorem~\ref{thm:mnw_gdom} for precise statement). 

  A modification of a contact structure on a $3$-dimensional manifold 
so that it has overtwisted disks is the so-called Lutz twist (see \cite{lutz}).
 It is a modification 
along a knot transverse to the contact structure 
replacing the standard tubular neighborhood 
$\lft(S^1\times D^2(\sqrt{\epsilon}),\zeta\rgt)$ 
with $\lft(S^1\times D^2(\sqrt{\epsilon+n\pi}),\zeta\rgt)$, 
where $\zeta=\ker\lft\{\cos r^2d\theta+\sin r^2d\phi\rgt\}$, 
$\epsilon>0$ sufficiently small, 
and $D^2(\sqrt{\epsilon})$ is a disk with radius $\sqrt{\epsilon}$ 
(see Subsection~\ref{sec:lutztwist} for precise definition). 
 When $n=1$ (resp.\ $n=2$), it is called the \emph{$\pi$-Lutz twist}\/ 
(resp.\ \emph{$2\pi$-Lutz twist}). 
 The Lutz twist makes an $S^1$-family of overtwisted disks. 
 The meridian disk $\{\phi=\text{const.\ }\}$ contains an overtwisted disk. 
 There exists an important difference between the $\pi$- and $2\pi$-Lutz twists.
 The $\pi$-Lutz twist contributes to the Euler class of a contact structure, 
while the $2\pi$-Lutz twist does not change the homotopy class 
of a contact structure as plane fields. 
 In addition to that, a similar modification along a pre-Lagrangian torus 
is also called the Lutz twist. 
 A pre-Lagrangian torus is an embedded torus in a contact $3$-manifold 
whose characteristic foliation is linear with closed leaves. 
 The \emph{Lutz twist}\/ along a pre-Lagrangian torus is a modification 
replacing $(T^2\times[\delta-\epsilon,\delta+\epsilon],\tilde\zeta)$ 
with $(T^2\times[\delta-\epsilon,\delta+\epsilon+k\pi],\tilde\zeta)$, 
where $\tilde\zeta=\ker\{\cos rd\phi+\sin rd\theta\}$ on $T^2\times\R$, 
$k\in\mathbb{N}$, 
and $T^2\times\{\delta\}$ is pre-Lagrangian. 
 This operation makes the \emph{Giroux $\pi$-torsion domain}\/ 
$(T^2\times[a,a+\pi],\tilde\zeta)$. 

  In this paper, generalizations of the Lutz twists are proposed. 
 We deal with generalization of the both $3$-dimensional Lutz twists 
along a transverse knot and along a pre-Lagrangian torus. 
 The basic ideas are the descriptions of the Lutz twists 
by contact round surgeries in dimension~$3$ (see \cite{art20}, \cite{art18}). 
 There exist other generalizations of the Lutz twists, 
which we mention later in this section. 

  First, we discuss a generalization of the Lutz twist along a transverse knot. 
 The description of the $3$-dimensional Lutz twist along a transverse knot 
by contact round surgeries begins 
with the contact round surgery of index~$1$ 
(see \cite{art20} or Subsubsection~\ref{sec:revrdsurgrep}). 
 We generalize this method to higher dimensions. 
 The first operation 
that we operate for a $(2n+1)$-dimensional contact manifold 
is also the contact round surgery of index~$1$. 
 This implies that the modified region is one of the connected components 
of the attaching region $\rd_-R^{2n+2}_1\cong\rd D^1\times D^{2n}\times S^1$ 
of the $(2n+2)$-dimensional symplectic round handle 
$R^{2n+2}_1\cong D^1\times D^{2n}\times S^1$ of index~$1$. 
 It is regarded as a tubular neighborhood of a certain circle. 
 The generalization of the Lutz twist proposed in this paper 
is operated along a circle embedded into a contact manifold 
which is transverse to the contact structure. 
 The first result is the following. 
%
%
\begin{thm}\label{thm_main}
  Let $(M,\xi)$ be a contact manifold of dimension $(2n+1)$, 
and $\Gamma\subset(M,\xi)$ an embedded transverse circle. 
 Then we can modify $\xi$ in a small tubular neighborhood of\/ $\Gamma$ 
so that the modified contact structure $\tilde\xi$ admits\textup{:} 
\begin{itemize}
\item an $S^1$-family of overtwisted discs, 
\item an $S^1$ family of the bordered Legendrian open books 
each of which has $(n-1)$-dimensional torus $T^{n-1}$ as binding. 
\end{itemize}
  Further, this modification has two versions. 
 One can be done 
so that $\tilde\xi$ is homotopic to the original contact structure $\xi$ 
as almost contact structures. 
 The other can be done 
so that it contributes to the Euler class of the contact structure. 
\end{thm} 
\noindent 
 In this paper, we call these generalizations of the Lutz twist 
obtained in Theorem~\ref{thm_main} 
the \emph{generalized Lutz twist}\/ along a transverse circle. 

 Other higher-dimensional generalizations of the Lutz twist known so far 
preserve the homotopy class of a contact structure 
as almost contact structures. 
 On the other hand, the existence of a generalization of $\pi$-Lutz twist 
was an open question (see \cite{eptw}). 
 Then Theorem~\ref{thm_main} (see also Proposition~\ref{prop:cntrbtoE}) 
gives an answer to one of the questions in \cite{eptw}.

  Next, we discuss higher-dimensional generalizations 
of the $3$-dimensional Lutz twist along a pre-Lagrangian $2$-torus. 
 Two kinds of generalizations are proposed in this paper. 
 One is defined as a modification of a contact structure 
along the so-called $\xi$-round hyper surface $H=K^{2n-1}\times S^1$. 
 The other is defined as a modification of a contact structure 
along a pre-Lagrangian torus $T^{n+1}$ in dimension $2n+1$. 
 A \emph{$\xi$-round hypersurface}\/ introduced in \cite{mnw} 
is roughly a family of contact submanifolds 
(see Subsubsection~\ref{sec:xi-rdhyps}). 
 A \emph{pre-Lagrangian submanifold}\/ is a projection 
of a Lagrangian submanifold in the symplectization, 
which is roughly a family of Legendrian submanifolds 
(see Subsubsection~\ref{sec:preLag}). 
 The both of them can be regarded as higher-dimensional generalizations 
of pre-Lagrangian $2$-torus in a contact $3$-manifold. 

  The result for 
a $\xi$-round hypersurface 
is as follows. 
 Let $\eta_0=\ker\lft.\lft\{\sum_{i=1}^{n}r_i^2d\theta_i\rgt\}\rgt|_{TS^{2n-1}}$ 
be the standard contact structure on a sphere $S^{2n-1}\subset\R^{2n}$, 
where $(r_1,\theta_1,\dots,r_n,\theta_n)$ are the coordinates of $\R^{2n}$.  
%
%
\begin{thm}\label{thm:gdom-xird}
  Let $(M,\xi)$ be a contact manifold of dimension $(2n+1)$, $n>1$, 
and $H=S^{2n-1}\times S^1\subset(M,\xi)$ an embedded $\xi$-round hypersurface 
modeled on the standard contact sphere $(S^{2n-1},\eta_0)$. 
 Then we can modify $\xi$ in a small tubular neighborhood of $H$ 
so that the modified contact structure $\tilde\xi$ admits 
an $S^1$-family of bordered Legendrian open books 
whose bindings are $T^{n-1}$. 
 In other words, the contact manifold $(M,\tilde\xi)$ does not have 
any semi-positive weak symplectic filling if $M$ is closed. 

  If $n=1$, this modification is the $3$-dimensional Lutz twist 
along a pre-Lagrangian torus $H=S^1\times S^1$. 
 It makes not a bordered Legendrian open book, that is, an overtwisted disk, 
but the Giroux $\pi$-torsion domain. 
\end{thm} 
%
%
\begin{rmk}
  Although the modification in Theorem~\ref{thm_main} creates overtwisted discs,
that in Theorem~\ref{thm:gdom-xird} does not directly. 
 We need further global observations to look for overtwisted discs. 
\end{rmk}
\noindent
 We call, in this paper, this generalization 
the \emph{generalized Lutz twist 
along a $\xi$-round hypersurface modeled on the standard contact sphere}. 
 The difference between the two cases, $n>1$ and $n=1$, 
comes from the shape of the (generalized) Lutz tube 
(see Remark~\ref{rem:gltxird-diff}). 
%
%

  The result for 
a pre-Lagrangian torus is as follows. 
 The preceding generalization in Theorem~\ref{thm:gdom-xird}
takes a $\xi$-round hypersurface $H=S^{2n-1}\times S^1$ 
as a generalization of a pre-Lagrangian $2$-dimensional torus 
in a $3$-dimensional contact manifold. 
 The following generalization takes a pre-Lagrangian torus $T^{n+1}$ instead. 
 We can mention the strong fillability by this method. 
%
%
\begin{thm}\label{thm:gdom-pLag}
  Let $(M,\xi)$ be a contact manifold of dimension $(2n+1)$, 
and $T^{n+1}=T^n\times S^1\subset(M,\xi)$ an embedded pre-Lagrangian torus 
with a product foliation by Legendrian torus leaves $T^n\times\{\ast\}$. 
 Then we can modify $\xi$ in a small tubular neighborhood of $T^{n+1}$ 
so that the modified contact structure $\tilde\xi$ admits 
the Giroux domain.   
\end{thm} 
\noindent
 We call, in this paper, this generalization simply 
the \emph{generalized Lutz twist along a pre-Lagrangian torus 
with a product foliation by Legendrian torus leaves}. 

%
%
\begin{rmk}
(1)\   The modification in Theorem~\ref{thm:gdom-pLag} does not creates 
the bordered Legendrian open book directly. 
 We need further global observations to look for it. \\
(2)\ Relations between the Giroux domains and symplectic fillings 
are studied in \cite{mnw}. 
 From their result (see Theorem~\ref{thm:mnw_gdom}) and 
the construction of this modification (in Subsection~\ref{sec:Gdom}), 
a contact structure is modified by the generalized Lutz twist 
in Theorem~\ref{thm:gdom-pLag},
so that the modified contact manifold have 
no semi-positive strong symplectic filling. 
\end{rmk}

  The key ideas for Theorem~\ref{thm_main} are the following. 
 One of the important object is a generalization of the Lutz tube. 
 Recall that a contact round surgery description 
of the $3$-dimensional Lutz twist needs the open Lutz tube: 
\begin{equation*}
  (S^1\times\R^2,\zeta_0),\qquad 
  \zeta_0=\ker\lft\{\cos r^2d\phi+\sin r^2d\theta\rgt\}, 
\end{equation*}
where $(\phi,r,\theta)$ is the cylindrical coordinates of $S^1\times\R^2$ 
(see \cite{art20}, or Subsubsection~\ref{sec:revrdsurgrep}). 
 As a generalization, we take the hyperplane field $\zeta$ on $S^1\times\R^{2n}$
defined as 
\begin{equation*}
  \zeta
  =\ker\lft\{\prod_{i=1}^n(\cos r_i^2)d\phi
  +\sum_{i=1}^n(\sin r_i^2)d\theta_i\rgt\}, 
\end{equation*}
where $(\phi,r_i,\theta_i)$ are coordinates of $S^1\times\R^{2n}$ 
(see Subsection~\ref{sec:gLtube}). 
 It is not a contact structure but a confoliation. 
 Like there exists an overtwisted disk 
in $\{\ast\}\times\R^2\subset(S^1\times\R^2,\zeta_0)$, 
a bordered Legendrian open book with a torus $T^{n-1}$ as binding exists 
in $\{\ast\}\times\R^{2n}\subset\lft(S^1\times\R^{2n},\zeta\rgt)$. 
 The important observation is 
that, in dimension~$3$, there is no difference, outer or inner, 
between two boundary components of the toric annulus 
$S^1\times\lft\{D^2(\sqrt{2\pi})\setminus\inter D^2(\sqrt{\pi})\rgt\}
=T^2\times\lft[\sqrt{\pi},\sqrt{2\pi}\rgt]\subset(S^1\times\R^2,\zeta_0)$. 
 However, in higher-dimensions, it is not true. 
 As we observe in Subsubsection~\ref{sec:revrdsurgrep}, 
from the view point of round surgery, 
the Lutz twist is not the simple replacement of $S^1\times D^2(\sqrt{\epsilon})$
with $S^1\times D^2(\sqrt{\epsilon+\pi})\subset(S^1\times\R^2,\zeta_0)$. 
 It is the replacement of $S^1\times D^2(\sqrt{\epsilon})$ 
with the toric annulus 
$T^2\times\lft[\sqrt{\pi},\sqrt{2\pi}\rgt]$ 
and the blowing down along the ``outer'' end $T^2\times\{\sqrt{2\pi}\}$. 
 In order to generalize this operations to higher-dimensions, 
we regard the toric annulus $T^2\times\lft[\sqrt{\pi},\sqrt{2\pi}\rgt]$ 
as two toric annuli $T^2\times\lft[\sqrt{\pi},\sqrt{3\pi/2}\rgt]$ 
glued along $T^2\times\{\sqrt{3\pi/2}\}$. 
 We generalize these observations to higher-dimensions. 
 We introduce the double $\du:=U(\sqrt{\pi})\cup U(\sqrt{\pi})$ of 
\begin{equation*}
  U(\sqrt{\pi})=\lft\{0\le r_i\le\sqrt{\pi},i=1,\dots,n\rgt\}
  \subset\lft(S^1\times\R^{2n},\zeta\rgt)
\end{equation*}
as a fundamental unit. 
 From the confoliation $\zeta'$ on $\du$ obtained from $\zeta$, 
a contact structure $\tilde\zeta$ on $\du$ is obtained. 
 Removing the standard tubular neighborhood 
of the transverse core 
$S^1\times\{0\}\subset U(\sqrt{\pi})\subset\lft(S^1\times\R^{2n},\zeta\rgt)$ 
from $\du$, 
we obtain a contact manifold diffeomorphic to $S^1\times D^{2n}$. 
 We call it the \emph{model $\pi$-Lutz tube}, 
and let $\lft(\mlt,\tilde\zeta\rgt)$ denote it 
(see Figure~\ref{fig:cnstgltube}). 

  The generalized Lutz twist along an embedded transverse circle 
and that along a $\xi$-round hypersurface modeled on the standard contact sphere
are described by contact round surgeries. 
 Although the descriptions are not explicitly used 
in the definition or proofs of the theorems, 
these are the fundamental ideas of this paper. 
%
%
\begin{thm}\label{thm:hdim} 
  We deal with a $(2n+1)$-dimensional contact manifold $(M,\xi)$.  
\begin{enumerate}
  \item The generalized Lutz twist 
    along a transverse curve $\Gamma\subset(M,\xi)$ 
    is realized by a certain pair of contact round surgeries 
    of index $1$ and $2n$ with the model Lutz tube. 
  \item The generalized Lutz twist
    along a $\xi$-round hypersurface $H=S^{2n-1}\times S^1$ 
    modeled on the standard contact sphere 
    is realized by contact round surgeries 
    of index $2n$ and $1$ with the model Lutz tube. 
\end{enumerate}
\end{thm} 
\noindent
  As a matter of fact, the observations in Theorem~\ref{thm:hdim} 
is a motivation of this paper. 
 In fact, it is proved in \cite{art20} that the $3$-dimensional Lutz twists 
are realized by contact round surgeries. 
 Then it was expected that a similar operations by contact round surgeries 
corresponds to a higher-dimensional generalization of the Lutz twists.  

  From the view point of round surgery, 
it is interesting to regard the double $\lft(\du,\tilde\zeta\rgt)$ 
a certain unit. 
 The generalized Lutz twists dealt in Theorem~\ref{thm:hdim} 
are considered as operations taking in this unit by contact round surgeries 
(see Figures~\ref{fig:l1tw} and~\ref{fig:l2ntw}). 


  Next, we mention some applications of Theorem~\ref{thm_main}, 
although they are not quite a new results. 

 Recall that the generalized Lutz twist is operated 
along an embedded circle transverse to the contact structure. 
 We claim that any circle $\gamma$ embedded in a contact manifold $(M,\xi)$ 
can be approximated by a circle $\bar\gamma$ transverse to $\xi$ 
(see Subsubsection~\ref{sec:findtrvs1}). 
 Therefore, we can apply the generalized Lutz twist anywhere we like. 
 This implies the following. 
%
%
\begin{cor}\label{cor:anywhere}
  If an odd-dimensional manifold $M$ has a contact structure, 
it admits an overtwisted contact structure. 
\end{cor} 
\noindent
 For PS-overtwisted structures, this result 
was proved by Niederkr\"uger and van~Koert~\cite{niedvkoe},
and by Presas~\cite{presas}. 
 It can also be proved by using other generalization of the Lutz twist 
due to Etnyre and Pancholi~\cite{eptw}. 
 However, as the other generalization due to Massot, Niederkr\"uger, 
and Wendl~\cite{mnw} requires stricter conditions, 
it can not be applied everywhere. 
 In~\cite{newOT}, it is proved 
that, in every homotopy class of almost contact structures, 
there exists an overtwisted contact structure.

 An argument similar to Theorem~\ref{thm_main} can be applied 
to some open case. 
 It was claimed by Etnyre and Pancholi~\cite{eptw}, 
and Niederkr\"uger and Presas~\cite{niedpres}  
that there exist at least three distinct contact structures 
on $\mathbb{R}^{2n+1}$, $n\ge1$. 
 One is the standard contact structure 
$\ker\{d\phi+\sum_{i=1}^{n}r_i^2d\theta_i\}$. 
 Another is PS-overtwisted but standard at infinity, that is, 
it is standard outside a compact subset. 
 The other is PS-overtwisted at infinity, that is, for any relatively compact 
open subset, there exists a bordered Legendrian open book outside the subset. 
 By an argument similar to Theorem~\ref{thm_main}, 
a contact structure which is overtwisted at infinity can be constructed 
(see Subsection~\ref{sec:non-cpt}). 
 Then this implies the following. 
%
%
\begin{cor}\label{thm:exoticR2n1}
  There are at least three distinct contact structures 
on $\mathbb{R}^{2n+1}$, $n\ge1$. 
\end{cor} 

  We should mention some other attempts to generalize the Lutz twists 
to higher dimensions. 

  One is due to Etnyre and Pancholi~\cite{eptw}. 
 They constructed a modification of a contact structure 
so that it has a family of bordered Legendrian open books. 
 Their generalization takes an $n$-dimensional submanifold $B\times S^1$ 
of a $(2n+1)$-dimensional contact manifold $(M,\xi)$ 
as a generalization of a transverse knot, 
where $B$ is a closed $(n-1)$-dimensional isotropic submanifold 
with a trivial conformal symplectic normal bundle. 
 Then it modifies the contact structure $\xi$ along $B\times S^1$ 
so that it has an $S^1$-family of bordered Legendrian open books 
whose bindings are $B$. 
 Moreover, the modified contact structure is homotopic to the original $\xi$ 
as almost contact structures. 
 In this sense, it is a generalization of the $2\pi$-Lutz twist.  
 In order to avoid confusion, we call the generalization of the Lutz twist 
introduced in \cite{eptw} the \emph{Etnyre-Pancholi twist}\/ 
(the \emph{EP-twist}\/ for short) in this paper. 

  Another generalization is due 
to Massot, Niederkr\"uger, and Wendl~\cite{mnw}. 
 Their results are inspired by works of Mori~\cite{mori} 
on $5$-dimensional sphere. 
 Their generalization takes a certain $(2n-1)$-dimensional contact submanifold 
$N$ of a $(2n+1)$-dimensional contact manifold $(M,\xi)$ 
as a generalization of a transverse knot. 
 Then it modifies the contact structure $\xi$ along $N$ 
so that it has a family of bordered Legendrian open books 
with an $n$-dimensional parameter space in $N$. 
 The modified contact structure is homotopic to the original $\xi$ 
as almost contact structures. 
 They also generalize the $3$-dimensional Lutz twist 
along a pre-Lagrangian $2$-torus. 
 The generalization takes the so-called $\xi$-round hypersurface 
$H=K^{2n-1}\times S^1$
as a generalization of a pre-Lagrangian $2$-torus in dimension~$3$. 
(see Subsubsection~\ref{sec:xi-rdhyps} for definition). 
 Then it modifies $\xi$ along $H$ so that it has 
the bordered Legendrian open book or the Giroux domains
. 
 The generalization of this type is carefully studied by Kasuya~\cite{kasuya}. 
 In order to avoid confusion, we call the generalization of the Lutz twist 
introduced in \cite{mnw} the \emph{Massot-Niederkr\"uger-Wendl twist}\/ 
(the \emph{MNW-twist}\/ for short) in this paper. 

\smallskip

  This paper is organized as follows. 

  In Section~\ref{sec:prelim}, we review some basic notions in contact topology 
that are needed in the following discussions. 

  Section~\ref{sec:genlztw} is devoted to the generalization 
of the $3$-dimensional Lutz twist along a transverse knot. 
 The second half of Theorem~\ref{thm_main} is proved here. 
 First, we introduce the generalized open Lutz tube 
$\lft(S^1\times\R^{2n},\zeta\rgt)$, 
whose hyperplane field is a conductive confoliation, 
in Subsection~\ref{sec:gLtube}. 
 Then, in Subsection~\ref{sec:defGenLztw}, 
we define the generalized Lutz twist along a transverse circle. 
 The properties of the generalized $\pi$- and $2\pi$-Lutz twists are studied 
in Subsection~\ref{sec:handftwists}. 

  Generalizations of the $3$-dimensional Lutz twist along a pre-Lagrangian torus
is discussed in Section~\ref{sec:gtor}. 
 In Subsection~\ref{sec:circsph}, we define the generalized Lutz twist 
along a $\xi$-round hypersurface, and prove Theorem~\ref{thm:gdom-xird}. 
 In Subsection~\ref{sec:Gdom}, we define the generalized Lutz twist 
along a pre-Lagrangian $(n+1)$-dimensional torus, 
and prove Theorem~\ref{thm:gdom-pLag}. 

  We review the symplectic round handle and the contact round surgery 
in Section~\ref{sec:ctrdsurg}. 
 Especially, contact round surgery of index~$1$ 
is carefully observed, and that of index~$2n$ is defined here. 
 Then, by using these contact round surgeries 
the descriptions of the generalized Lutz twists are discussed 
in Section~\ref{sec:gentw}. 
 Theorem~\ref{thm:hdim} is proved here. 

  In Section~\ref{sec:newot}, we discuss the relation 
between the generalized Lutz twist and the new overtwisted discs 
introduced in~\cite{newOT}. 
 Then, we complete the proof of Theorem~\ref{thm_main}. 

\medskip

\noindent
\textbf{Acknowledgements.}
  The author is grateful to Yakov Eliashberg, John Etnyre, Takeo Noda, 
and Otto van~Koert for valuable discussions and helpful comments. 

\section{Preliminaries}\label{sec:prelim}
  In this section we introduce important notions and properties 
in contact topology 
which are needed in the rest of this paper. 

\subsection{Overtwistedness of contact structures on $3$-manifolds}
\label{sec:ovtw}
  We begin with a basic coarse classification 
of contact structures on $3$-manifolds. 
  It is well-known that contact structures on $3$-dimensional manifolds 
are divided into two contradictory classes, tight and overtwisted. 
 A contact structure $\xi$ on a $3$-dimensional manifold $M$ 
is said to be \emph{overtwisted \/} 
if there exists an embedded disk $D\subset M$ which is tangent to $\xi$ 
along its boundary: $T_{x}D=\xi_x$ at any point $x\in \rd D$. 
 A contact structure $\xi$ is said to be \emph{tight \/} 
if it is not overtwisted. 

  In this paper, we deal with higher dimensional cases as well. 
 A generalization of overtwisted disk to higher dimensions 
is introduced in \cite{mnw} (see Subsection~\ref{sec:bLob} in this paper). 
 Furthermore, a stronger generalization is introduced in~\cite{newOT} 
(see Section~\ref{sec:newot}). 
 In order to recognize the overtwisted disk as a part of the generalizations, 
we should define the original one more precisely. 

  For the precise definition, we need the notion of characteristic foliation 
of a surface in a contact $3$-manifold. 
 Let $F$ be a surface in a contact $3$-manifold $(M,\xi)$. 
 The \emph{characteristic foliation}\/ of $F$ with respect to $\xi$ 
is the $1$-dimensional foliation with singularity 
defined by the vector field $X$ defined as follows. 
 Let $\alpha$ be a contact form defining $\xi$ locally, 
and $\mu_F$ a volume form on $F$. 
 The vector field $X$ is defined by the equation 
$X\intp\mu_F=\alpha\vert_{TF}$. 
 Let $F_\xi$ denote the characteristic foliation. 
 In other words, 
the characteristic foliation $F_\xi$ 
is generated from the line field $\xi\cap TF$. 
 At singular points, $\xi$ and $TF$ coincide. 
 The notion, characteristic foliation, is also defined in higher dimensions. 
 The general definition is given in Subsubsection~\ref{sec:cvhsrf}. 

  Now, we define overtwisted disk. 
 When a contact structure $\xi$ on a $3$-dimensional manifold $M$ 
is overtwisted, 
we may take an embedded disk $D\subset(M,\xi)$ in the definition above 
as a disk whose characteristic foliation $D_\xi$ 
satisfies the following conditions 
(see Figure~\ref{fig:otdisk}): 
(1)~the boundary $\rd D$ is a set of singular points, 
(2)~there exists a unique isolated singular point on $D$, 
(3)~each leaf connects the singular point and the boundary, 
(4)~each leaf is transverse to the boundary. 
%
%
\begin{figure}[htb]
  \centering
  \includegraphics[height=3cm]{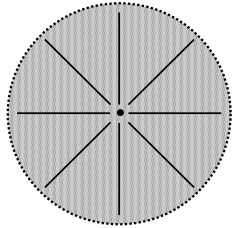}
  \caption{Overtwisted disk}
  \label{fig:otdisk}
\end{figure} 
 Although this definition of overtwisted disk looks stricter than the above, 
we can obtain this overtwisted disk from the embedded disk 
described in the definition above of overtwistedness 
by perturbing and taking a subset. 
 The overtwisted disc introduced in~\cite{newOT} can be taken 
in a tubular neighborhood of the overtwisted disk above 
(see Section~\ref{sec:newot}). 

  Tightness or overtwistedness of contact structures 
has an important relation with fillability by symplectic manifolds. 
 For a closed contact manifold of dimension $2n-1$, 
symplectic fillability is defined as follows. 
%
%
\begin{dfn}
  Let $(M,\xi)$ be a $(2n-1)$-dimensional closed contact manifold. 
Assume that the contact structure $\xi$ is cooriented. 
\begin{itemize}
\item $(M,\xi)$ is said to be \emph{weakly symplectically fillable}\/
  if there exists a $(2n)$-dimensional compact symplectic manifold $(W,\omega)$ 
  which satisfies (i)~$\rd W=M$ as oriented manifolds, 
  and (ii)~$\omega^{n-1}\vert_\xi>0$. 
\item $(M,\xi)$ is said to be \emph{strongly symplectically fillable}\/
  if there exist a compact symplectic manifold $(W,\omega)$ 
  and a Liouville vector field $X$ (i.e. $L_X\omega=\omega$) 
  defined near $\rd W$ pointing outward transversely to $\rd W$ 
  that satisfy (i)~$\rd W=M$, and (ii)~$\xi=\ker(X\intp\omega)\vert_{TM}$. 
\end{itemize}
\end{dfn}
\noindent
 It is not hard 
to check that a strongly symplectically fillable contact manifold 
is weakly symplectically fillable. 
 For $3$-dimensional contact manifolds, 
there exists an important property 
proved by Eliashberg~\cite{elifill} and Gromov~\cite{gromov}. 
%
%
\begin{thrm}[Eliashberg, Gromov]\label{thm:eligro}
  Let $(M,\xi)$ be a contact $3$-manifold 
which is weakly symplectically fillable. 
 Then the contact structures $\xi$ is tight. 
\end{thrm}
\noindent
 This implies 
that if there exists an overtwisted disk in a contact $3$-manifold, 
it never has any weak symplectic filling. 

\subsection{Lutz twist and Giroux torsion}\label{sec:lutztwist}
  We review Lutz twist and Giroux torsion in this subsection. 
 To generalize these notions to higher dimensions 
is one of main purposes of this paper. 
 Then to review the original definition 
for contact structures on $3$-dimensional manifolds is important. 
 Lutz twist is a modification of a contact structure 
on a $3$-dimensional manifold. 
 It is introduced by Lutz~\cite{lutz} to study contact structures 
which are not equivalent to the standard one. 
 The Lutz twist is operated along a transverse knot. 
 A similar twist along a pre-Lagrangian torus is also called a Lutz twist. 
 The second one makes the notion of Giroux torsion~\cite{giroux99}. 

  First, we introduce the Lutz twist along a transverse knot. 
 Let $\Gamma$ be a transverse knot in a contact $3$-manifold $(M,\xi)$. 
 Then it is well known (see \cite{geitext} for example) 
that it has the standard tubular neighborhood  $U\subset(M,\xi)$ 
which is contactomorphic to a tubular neighborhood of 
$S^1\times\{0\}\subset(S^1\times\R^2,\xi_0)$, 
where $\xi_0:=\ker\lft\{d\phi+r^2d\theta\rgt\}$ and $(\phi,r,\theta)$ 
are the cylindrical coordinates of $S^1\times\R^2$. 
 We may regard $U$ contactomorphic to $(S^1\times D^2(\delta),\xi_0)$ 
for some sufficiently small $\delta>0$, 
where $D^2(\delta)$ is an unit disk with radius $\delta$. 
 In order to define the Lutz twist, we work on $S^1\times D^2(\delta)$. 
 Take a $1$-form in the form $f(r)d\phi+g(r)d\theta$, 
for some functions $f(r)$, $g(r)$. 
 It is a contact form if the curve $(g(r),f(r))$ in $(g,h)$-plane 
rotates around the origin with respect to $r$ (see Figure~\ref{fig:LzTw}-(I)). 
%
%
\begin{figure}[htb]
  {\scriptsize \includegraphics[height=4.8cm]{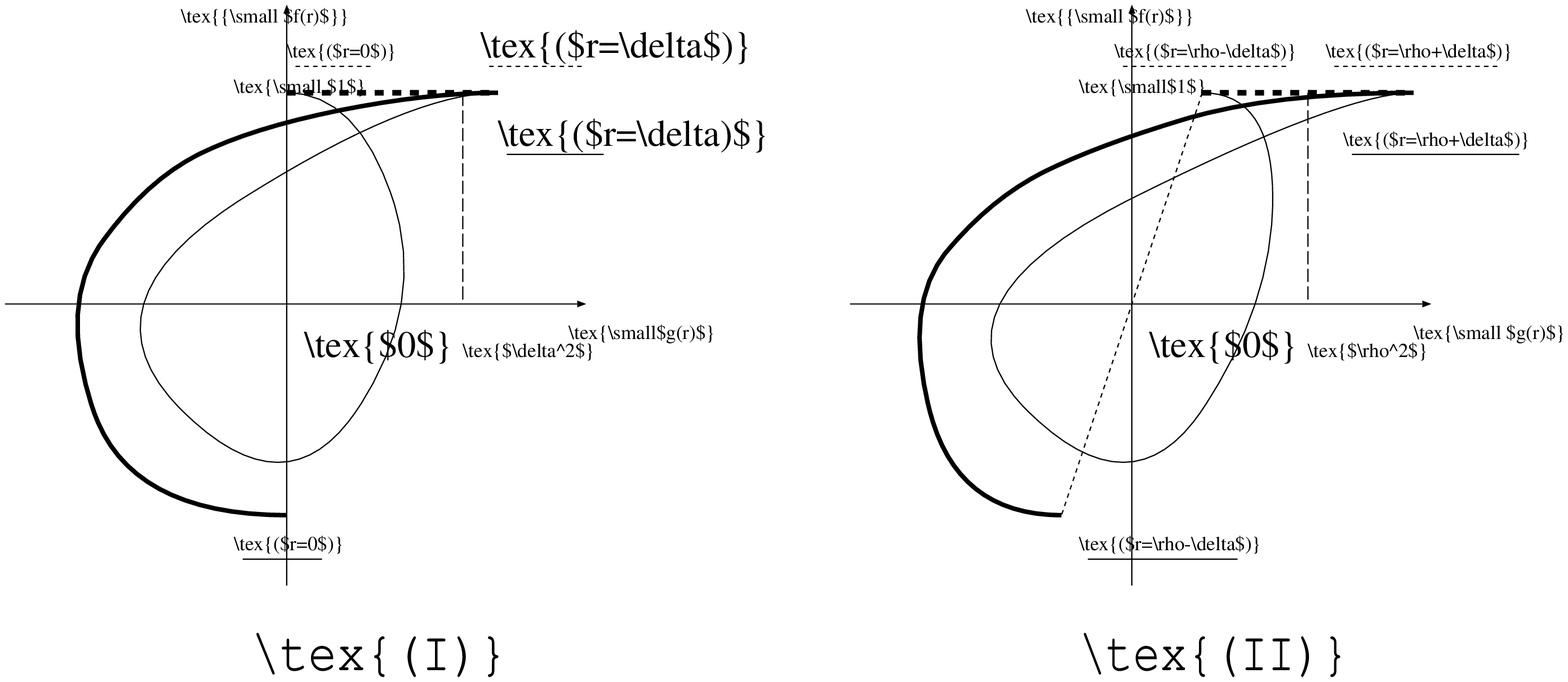}}
  \caption{Lutz twisting}
  \label{fig:LzTw}
\end{figure} 
 The standard contact structure on $S^1\times\R^2$ is represented 
by the dotted horizontal segment in Figure~\ref{fig:LzTw}-(I). 
 Then we have two functions $f(r)$, $g(r)$ 
so that the curve $(g(r),f(r))$ 
is a solid thick curve in Figure~\ref{fig:LzTw}-(I). 
 In other words, we have a contact structure 
$\xi':=\ker\lft\{f(r)d\phi+g(r)d\theta\rgt\}$. 
 Replacing the given contact structure $\xi$ on $U$ 
with the contact structure corresponding to $\xi'$, 
we obtain a new contact structure $\tilde\xi$ on $M$. 
 This operation, 
modifying the given contact structure $\xi$ along a transverse knot $\Gamma$, 
is called the $\pi$-\emph{Lutz twist}\/ along $\Gamma$. 
 By taking a thin curve in Figure~\ref{fig:LzTw}-(I) 
instead of the thick curve above, 
we can define the $2\pi$-\emph{Lutz twist}\/ along $\Gamma$. 

  A Lutz twist along a transverse curve makes a contact structure 
be overtwisted. 
 In fact, by definition, we have a tubular neighborhood $U$ 
of the transverse curve which is contactomorphic to 
$\lft(S^1\times D(\delta), \xi'=\ker\{f(r)d\phi+g(r)d\theta\}\rgt)$, 
where the functions $f(r)$, $g(r)$ are given in Figure~\ref{fig:LzTw}-(I). 
 For the thick curve in Figure~\ref{fig:LzTw}-(I) (or a $\pi$-Lutz twist), 
there exists $\epsilon$, $0<\epsilon<\delta$, 
where the curve intersects with the vertical axis 
(i.e.\ $g(\epsilon)=0,\ f(\epsilon)>0$). 
 Then the disk 
$D:=\{r\le\epsilon,\ \theta=\bar\theta\}\subset(S^1\times D(\delta),\xi')$ 
is an overtwisted disk for any constant $\bar\theta\in S^1$. 
 For the thin curve (or a $2\pi$-Lutz twist), 
there exist two points $\epsilon_1,\ \epsilon_2$, 
$0<\epsilon_1<\epsilon_2<\delta$, 
where the curve intersects with the vertical axis. 
 The disk $\tilde D:=\{r\le\epsilon_1,\theta=\bar\theta\}$ 
is an overtwisted disk. 
 According to Theorem~\ref{thm:eligro}, we may say 
that the Lutz twist makes a contact manifold be symplectically non-fillable. 

  Next, we review the notion of the Giroux torsion. 
 We begin with the Lutz twist along a pre-Lagrangian torus. 
 Let $T$ be a pre-Lagrangian torus in a contact $3$-manifold $(M,\xi)$. 
 A \emph{pre-Lagrangian torus}\/ is an embedded $2$-dimensional torus 
with a linear characteristic foliation $T_\xi$ with closed leaves. 
 Then it is well known (see \cite{geitext}, \cite{giroux99}) 
that it has a tubular neighborhood which is contactomorphic 
to a tubular neighborhood of $\{r=\rho\}\subset(S^1\times\R^2,\xi_0)$, 
where the torus $\{r=\rho\}$ has the same characteristic foliation as $T$. 
 We may regard the tubular neighborhood as 
$\lft(S^1\times \{D(\rho+\delta)\setminus\inter D(\rho-\delta)\},\xi_0\rgt)$. 
 The $\pi$-\emph{Lutz twist along}\/ $T\subset(M,\xi)$ is defined 
as the operation replacing the tubular neighborhood of $T$ 
with $S^1\times \{D(\rho+\delta)\setminus\inter D(\rho-\delta)\}$ 
with the contact structure correspondent to the thick curve 
in Figure~\ref{fig:LzTw}-{(II)}. 
 By using the thin curve in Figure~\ref{fig:LzTw}-{(II)} 
instead of the thick one, 
we can define the $2\pi$-\emph{Lutz twist along}\/ $T$. 

  Then we introduce the notion of the Giroux torsion. 
 A contact manifold $(M,\xi)$ is said 
to have the \emph{Giroux torsion at least} $n\in\mathbb{N}$ 
if there exists a contact embedding 
$f_n\colon(T^2\times I, \zeta_n)\to(M,\xi)$, 
where $\zeta_n$ is a contact structure on $T^2\times I$ with coordinates 
$(\phi,r,\theta)\in S^1\times I\times S^1\subset S^1\times\mathbb{R}^2$ 
defined as 
\begin{equation*}
  \zeta_n:=\ker\{\cos(2n\pi r)d\theta+\sin(2n\pi r)d\phi\}.   
\end{equation*}
 The supremum of these numbers for all possible such embeddings 
to the contact manifold $(M,\xi)$ 
is called the \emph{Giroux torsion}\/ of $(M,\xi)$. 
 If there exists no such embedding, 
$(M,\xi)$ is said to have the Giroux torsion $0$. 
 The definition can be extended for half-integers $n=m/2$, $m\in\mathbb{N}$. 
 Especially, for $n=1/2$, 
we call the image of the contact embedding 
$f_{1/2}\colon(T^2\times I, \zeta_{1/2})\to(M,\xi)$ 
the \emph{Giroux $\pi$-torsion domain}. 

  It is clear that the Giroux torsion and the Lutz twist along a torus 
are closely related. 
 The contact form defining the contact structure $\zeta_n$ in the definition of 
the Giroux torsion corresponds to a curve in $\lft(g(r),f(r)\rgt)$-plane, 
like Figure~\ref{fig:LzTw}-{(II)}, 
rotating around the origin $n$ times. 
 This implies that the $2\pi$-Lutz twist makes a contact structure 
have the Giroux torsion at least $1$. 

  At the end of this subsection, 
we review the important properties of the Giroux torsion. 
 The existence of the Giroux torsion domain does not imply overtwistedness. 
 However, it is proved by Gay~\cite{gay} 
that if the Giroux torsion is greater than or equal to $1$, 
the contact manifold is not strongly symplectically fillable. 

\subsection{Bordered Legendrian open book and Giroux domain}
\label{sec:bLob_Gdom}
  We introduce two notions, bordered Legendrian open book and Giroux domain 
in this subsection. 
 They are defined in \cite{mnw} for higher-dimensional contact manifolds 
as generalizations of certain important notions 
in $3$-dimensional contact topology. 

\subsubsection{Bordered Legendrian open book}\label{sec:bLob}
 A bordered Legendrian open book is introduced in \cite{mnw} 
as a higher-dimensional generalization 
of overtwisted disk in dimension~$3$. 
 It is defined as follows. 
 Let $(M,\xi)$ be a contact $(2n+1)$-dimensional manifold. 
%
%
\begin{dfn}
  A compact $(n+1)$-dimensional submanifold  $N\subset(M,\xi)$ with boundary 
is called a \emph{bordered Legendrian open book} (\textsf{bLob} for short) 
if there exists a pair $(B,\theta)$ 
of $(n-1)$-dimensional closed submanifold $B\subset\inter N$ 
with trivial normal bundle 
and a mapping $\theta\colon N\setminus B\to S^1$ 
which satisfies the following conditions: 
\begin{itemize}
\item $\theta\colon N\setminus B\to S^1$ is a fibration 
  \begin{itemize}
  \item that coincide, 
    in a tubular neighborhood $B\times D^2$ of $B=B\times\{0\}$, 
    with $(b,r,\phi)\mapsto \phi$, 
    where $(r,\phi)$ is the polar coordinates of $D^2$ and $b\in B$ is a point, 
  \item whose fibers are transverse to $\partial N$ in $N$, 
  \end{itemize}
\item fibers $\theta^{-1}(s)\subset(M,\xi)$ of 
  $\theta\colon N\setminus B\to S^1$, $s\in S^1$, are Legendrian, 
\item the boundary $\partial N\subset(M,\xi)$ is Legendrian. 
\end{itemize}
\end{dfn}

  In dimension $3$ (i.e. $n=1$), a bordered Legendrian open book 
is an overtwisted disk. 
 In fact, a $0$-dimensional closed manifold $B$ is a point. 
 Then a $2$-dimensional manifold $N$ is a $2$-dimensional disk $D$ 
with the Legendrian boundary, 
and the fibers of $\phi\colon D\setminus\{pt\}\to S^1$ 
are leaves of the characteristic foliation $D_\xi$.  

  An important property of bordered Legendrian open book 
is proved in \cite{mnw} (see also \cite{niedPS},\cite{gromov}). 
 It is a reason why a bordered Legendrian open book can be considered 
as a generalization of an overtwisted disk. 

%
%
\begin{thrm}[Massot, Niederkr\"uger, Wendl]\label{thm:mnw_wkfill}
  If a closed contact manifold has a bordered Legendrian open book $N$, 
then it can not have any semi-positive symplectic filling $(W,\omega)$ 
which satisfies that $\omega\vert_{TN}$ is exact. 
\end{thrm} 
\noindent
 A $2n$-dimensional symplectic manifold $(W,\omega)$ 
is said to be \emph{semi-positive}
if there exists no element $A\in\pi_2(W)$ 
which satisfies $\langle A,[\omega]\rangle>0$ 
and $3-n\le\langle A,c_1(W)\rangle<0$, 
where $c_1(W)\in H^2(W:\R)$ is the first Chern class 
and $A$ is regarded as $S\in H_2(W;\R)$ by the Hurewicz homomorphism. 
 Note that any symplectic manifold of dimension less than or equal to $6$ 
is semi-positive. 
 According to \cite{mnw}, the condition ``semi-positive'' should be removable 
in the future. 
%
%
\begin{rmk}
  In \cite{mnw}, Massot, Niederkr\"uger, and Wendl proved much stronger results 
than the above. 
 They introduced a new notion, weak symplectic fillability, 
for higher-dimensional cases. 
 They obtained the above result for weak symplectic filling. 
\end{rmk}

  By comparing this theorem with Theorem~\ref{thm:eligro}, 
a contact manifold which contains a bordered Legendrian open book 
is said to be \emph{PS-overtwisted}. 
 ``PS'' comes from ``plastikstufe'', the name of the former version 
of bordered Legendrian open book (see \cite{niedPS}). 

\subsubsection{Giroux domain}\label{sec:defGdom}
  A Giroux domain is introduced in \cite{mnw} 
as a generalization of the Giroux $\pi$-torsion domain 
(see Subsection~\ref{sec:lutztwist}). 
 According to \cite{mnw}, this notion is due to Giroux. 

  The definition of the Giroux domain is based on another important notion, 
the ideal Liouville domain. 
%
%
\begin{dfn}[Giroux]
  Let $\Sigma$ be a compact manifold with boundary, 
$\omega$ an exact symplectic structure 
on the interior $\inter \Sigma$ of $\Sigma$, 
and $\xi$ a contact structure on the boundary $\rd\Sigma$. 
 The triple $\lft(\Sigma,\omega,\xi\rgt)$ 
is called an \emph{ideal Liouville domain}\/ 
if there exists a Liouville $1$-form $\beta$ for $\omega$ (i.e. $d\beta=\omega$)
defined on $\inter\Sigma$ 
which satisfies the following condition: 
for any smooth function $f\colon\Sigma\to[0,\infty)$ which has $\rd\Sigma$ 
as a regular level set $f^{-1}(0)=\rd\Sigma$, 
the $1$-form $f\beta$ on $\inter\Sigma$ extends to $\tilde\beta$ on $\Sigma$ 
smoothly so that $\ker\tilde\beta\vert_{T(\rd\Sigma)}=\xi$. 
\end{dfn}

  Now, the Giroux domain is defined 
as a contactization of the ideal Liouville domain. 
 Let $(\Sigma,\omega,\xi)$ be an ideal Liouville domain, 
and $\beta$ a Liouville $1$-form for $\omega$ appeared in the definition. 
 For a function $f\colon\Sigma\to[0,\infty)$ 
with a regular level set $f^{-1}(0)=\rd\Sigma$, 
we have a contact form $fdt+\tilde\beta$ on $\Sigma\times\R$, 
where $t\in\R$ is a coordinate and $\tilde\beta$ is an extension of $f\beta$. 
 Since the contact form is independent of $t\in\R$, 
we may regard it as a contact form on $\Sigma\times S^1$. 
 The contact manifold $\lft(\Sigma\times S^1,\ker(fdt+\tilde\beta)\rgt)$ 
is called a \emph{Giroux domain}\/ associated to the Liouville domain 
$(\Sigma,\omega,\xi)$. 

  The most basic example is the Giroux $\pi$-torsion domain. 
%
%
\begin{expl}\label{expl:annumodel}
 Let $\Sigma=S^1\times[0,\pi]$ be an annulus with coordinates $(\theta,s)$, 
$\omega=(1/\sin^2s)d\theta\wedge ds$ a symplectic structure 
defined on $\inter \Sigma=S^1\times(0,1)$, 
and $\xi=\ker d\theta$ a contact structure on $\rd\Sigma=S^1\times\{0,\pi\}$. 
 Then $(\Sigma,\omega,\xi)$ is an ideal Liouville domain 
with a Liouville form $\beta=(\cot s)\; d\theta$ (see \cite{mnw}). 
 For a function $f\colon\Sigma\to\R$, $(\theta,s)\mapsto\sin s$, 
the $1$-form $f\beta=(\cos s)\; d\theta$ can be defined on $\Sigma$, 
and $f\beta\vert_{T(\rd\Sigma)}=\pm d\theta$. 
 Therefore, the Giroux domain associated to $(\Sigma,\omega,\xi)$ is 
\begin{equation*}
  \lft(\Sigma\times S^1,\ker(fdt+f\beta)\rgt)
  =\lft(S^1\times[0,\pi]\times S^1,\ker\{(\sin s)dt+(\cos s)d\theta\}\rgt),
\end{equation*}
which is nothing but the Giroux $\pi$-torsion domain. 
\end{expl}

  The Giroux domain is introduced in \cite{mnw} to discuss Giroux torsion 
in higher dimensions. 
 The following property is proved in \cite{mnw}. 
%
%
\begin{thrm}[Massot, Niederkr\"uger, Wendl]\label{thm:mnw_gdom}
  Let $(M,\xi)$ be a closed contact manifold of dimension $2n+1$. 
 Suppose that $(M,\xi)$ contains subdomain $N$ with boundary 
consisting of two Giroux domains $\Sigma^+\times S^1$ and $\Sigma^-\times S^1$ 
of dimension $2n+1$. 
 Further, assume that $\rd(\Sigma^-\times S^1)=\rd\Sigma^-\times S^1$ 
has a connected component actually glued to a connected component 
of $\rd(\Sigma^+\times S^1)$ 
and a connected component which never intersect with $\Sigma^+\times S^1$. 
 Then $(M,\xi)$ has no semi-positive strong symplectic filling. 
\end{thrm}
%
%
\begin{rmk}
  In \cite{mnw}, Massot, Niederkr\"uger, and Wendl proved much stronger results 
than the above. 
 They discussed the relation between the Giroux domain 
and symplectic fillabilities for higher-dimensional cases. 
\end{rmk}

\subsection{Contact structures on non-compact manifolds}
\label{sec:atinfty}
  In this subsection, we deal with contact structures on non-compact manifolds. 
 We introduce a new notion, generalizing a notion for contact structures 
on non-compact $3$-manifolds. 

  Contact structures on $\R^3$ were classified by Eliashberg in \cite{eliashr3}.
 The notions ``overtwisted at infinity'' and ``tight at infinity'' 
were introduced in the paper. 
 A contact structure $\xi$ on a non-compact $3$-dimensional manifold $M$ 
is said to be \emph{overtwisted at infinity}\/ 
if, for any relatively compact subset $U\subset M$, 
there exists an overtwisted disk in $(M\setminus U,\xi|_{M\setminus U})$. 

  Now that the notion, overtwistedness, is generalized 
to higher dimensions in~\cite{newOT}, 
the notion, overtwistedness at infinity, is defined in higher dimensions 
in the same way. 

  Similarly we can define such a notion for the PS-overtwistedness. 
%
%
\begin{dfn}
  A contact structure $\xi$ on a non-compact manifold $M$ 
is said to be \emph{PS-overtwisted at infinity}\/ 
if, for any relatively compact subset $U\subset M$, 
there exists an overtwisted disk in $(M\setminus U,\xi|_{M\setminus U})$. 
\end{dfn} 

  Examples of contact structures overtwisted at infinity 
on $\R^{2n+1}$ are constructed in \cite{niedpres} and \cite{eptw}. 
 A more natural example is given in Subsection~\ref{sec:non-cpt} 
by the method in this paper. 
 By the method in this paper, we obtain a contact structure on $\R^{2n+1}$ 
overtwisted at infinity. 

\subsection{Neighborhood theorems for Submanifolds 
  in contact manifolds}\label{sec:sbminct}
  In this subsection, we review some results concerning certain submanifolds 
in contact manifolds. 
 Neighborhood theorems for some specific submanifolds are introduced. 
 First, we discuss isotropic submanifolds. 
 And then we deal with contact submanifolds. 
 $\xi$-round hypersurfaces is an important notion introduced in \cite{mnw}. 
 Then we introduce pre-Lagrangian submanifolds. 
 We review convex hypersurfaces as well. 

\subsubsection{Isotropic submanifold}\label{sec:isot}
  An isotropic submanifold in a contact manifold is defined as follows. 
 Let $(M,\xi)$ be a $(2n+1)$-dimensional contact manifold, 
and $\alpha$ a defining $1$-form of $\xi$: $\xi=\ker\alpha$. 
 A submanifold $N\subset(M,\xi)$ is said to be \emph{isotropic\/} 
if the pullback $i^\ast\alpha$ of $\alpha$ 
by the inclusion mapping $i\colon N\hookrightarrow M$ vanishes on $N$. 
 Note that the maximum dimension of isotropic submanifolds of $(M,\xi)$ 
is $n$ because of the non-integrability of $\xi$. 
 When the dimension of an isotropic submanifold $N\subset(M,\xi)$ is $n$, 
it is said to be \emph{Legendrian}. 

  The framing, namely the trivialization of the normal bundle, 
of an isotropic submanifold in a contact manifold is described as follows. 
 Let $N$ be an isotropic submanifold 
of a contact manifold $(M,\xi)$ with a contact form $\alpha$. 
 Then $d\alpha$ defines a symplectic structure on each fiber, or hyperplane, 
of $\xi$. 
 With respect to the symplectic structure $d\alpha$, 
two  vectors $u,\ v\in\xi_x$ is said to be \emph{skew orthogonal}\/ 
if $d\alpha_x(u,v)=0$. 
 This relation is written as $u\skor v$. 
 Set 
\begin{equation*}
  (TN)^{\skor}:=\bigcup_{x\in N}\lft\{u\in\xi_x\mid 
  d\alpha_x(u,v)=0,\ \text{for any}\ v\in T_xN\rgt\}. 
\end{equation*}
 It is a subbundle of $\xi\vert_N\subset TM\vert_N$. 
 Then the definition of isotropic submanifolds implies $TN\subset(TN)^{\skor}$. 
 The quotient bundle $(TN)^{\skor}/TN$ is a symplectic bundle, 
and is called the \emph{conformal symplectic normal bundle}\/ 
of the isotropic submanifold $N\subset(M,\xi=\ker\alpha)$. 
 Let $\csn(N,M)$ denote this bundle. 
 It follows that the ordinary normal bundle $\nu(N,M)=(TM|_N)/TN$ splits as 
\begin{equation}\label{eq:splitting}
  \nu(N,M)\cong (TM|_N)/(\xi|_N)\oplus(\xi|_N)/(TN)^{\skor}\oplus \csn(N,M).   
\end{equation}
 The following proposition implies that the symplectic normal bundle determines 
a local form of an isotropic submanifold (see \cite{geitext}). 
%
%
\begin{prop}\label{prop:isotnbd}
  Let $N_i$ be a closed isotropic submanifold 
in a contact manifold $(M_i,\xi_i)$ with a contact form $\alpha_i$ 
for each $i=0,1$.  
 Suppose that $\varphi\colon N_0\to N_1$ is a diffeomorphism 
covered by a bundle isomorphism of their symplectic normal bundles 
$\csn(N_i,M_i)$. 
 Then $\varphi$ extends to a strict contactomorphism of their neighborhoods. 
\end{prop} 

\subsubsection{Contact submanifold}\label{sec:ctsubmfd}
   Next, we discuss contact submanifolds. 
 Let $(M,\xi=\ker\alpha)$ be a contact manifold, 
and $N\subset M$ a submanifold. 
 $N$ is called a \emph{contact submanifold}\/ of $(M,\xi)$ 
if $\xi\vert_{TN}=\ker(i^\ast\alpha)$ is a contact structure on $N$, 
where $i\colon N\hookrightarrow M$ is the inclusion mapping. 
 For a contact submanifold $N\subset(M,\xi)$, 
another conformal symplectic normal bundle is defined as follows. 
 As a vector bundle whose fibers are skew orthogonal 
to the contact structure $\xi\vert_{TN}$ in $\xi$, we have 
\begin{equation*}
  (\xi|_{TN})^{\skor}:=\bigcup_{x\in N}\lft\{u\in\xi_x\mid 
  d\alpha_x(u,v)=0,\ \text{for any}\ v\in (\xi|_{TN})_x\rgt\}. 
\end{equation*}
 The following proposition implies 
that a contact structure on a contact submanifold 
and the symplectic normal bundle determine 
a local form of the contact submanifold (see \cite{geitext}). 
%
%
\begin{prop}\label{prop:ctnbd}
  Let $N_i$ be a closed contact submanifold 
in a contact manifold $(M_i,\xi_i)$ with a contact form $\alpha_i$ 
for each $i=0,1$.  
 Suppose that 
$\varphi\colon (N_0,\ker\alpha_0|_{TN_0})\to (N_1,\ker\alpha_1|_{TN_1})$ 
is a contactomorphism 
covered by a bundle isomorphism of their symplectic normal bundles 
$\lft(\xi_i|_{TN_i}\rgt)^\skor$. 
 Then $\varphi$ extends to a strict contactomorphism of their neighborhoods. 
\end{prop} 
\noindent
 A transverse curve in a contact manifold 
is a contact submanifold of dimension $1$. 
 Then, as a corollary of Proposition~\ref{prop:ctnbd}, 
we have the following neighborhood theorem for transverse curves. 

%
%
\begin{coro}\label{cor:s1dar}
  Let $\Gamma$ be a transverse circle 
in a $(2n+1)$-dimensional contact manifold $(M,\xi)$. 
 Then there exists a tubular neighborhood of\/ $\Gamma\subset(M,\xi)$ 
which is contactomorphic to some tubular neighborhood of a transverse circle 
$S^1\times\{0\}$ in the contact open manifold 
$(S^1\times\R^{2n},\xi_0)$ with  
\begin{equation*}
  \xi_0=\ker\lft(d\phi+\sum_{i=1}^n(x_idy_i-y_idx_i)\rgt)
  =\ker\lft(d\phi+\sum_{i=1}^nr_i^2d\theta_i\rgt), 
\end{equation*}
where $\phi$ is a coordinate of $S^1$, 
$(x_i,y_i)$ are coordinates of $\R^2$,  
and $(r_i,\theta_i)$ are polar coordinates of\/ $\R^2$, $i=1,\dots,n$. 
\end{coro} 

\subsubsection{$\xi$-round hypersurface}\label{sec:xi-rdhyps}
  The notion of $\xi$-round hypersurface is introduced in \cite{mnw} 
as a higher-dimensional generalization 
of pre-Lagrangian torus in contact $3$-manifold. 
 Let $H$ be a hypersurface in a $(2n+1)$-dimensional contact manifold $(M,\xi)$.
 Suppose that the hyper surface $H$ is diffeomorphic to $N\times S^1$ 
where $N$ is a closed $(2n-1)$-dimensional manifold 
with a contact structure $\eta$. 
 The hyper surface $H\cong N\times S^1$ is called 
a \emph{$\xi$-round hypersurface modeled on a contact manifold $(N,\eta)$}\/ 
if it is transverse to $\xi$ and satisfies 
\begin{equation*}
  \xi|_{T(N\times\{s\})}=\eta, \qquad 
  \xi_{(p,s)}\cap T_{(p,s)}H=\eta_p\oplus T_sS^1, 
\end{equation*}
at any point $(p,s)\in N\times S^1=H$, where we identify $H$ and $N\times S^1$. 
 Note that, when $(M,\xi)$ is $3$-dimensional, a $\xi$-round hypersurface 
is a pre-Lagrangian torus with closed leaves of the characteristic foliation. 

%
%
\begin{expl}\label{expl:xirdmostctsph}
  We observe that the boundary of the standard tubular neighborhood 
of a transverse circle in a contact manifold 
is a $\xi$-round hypersurface modeled on the standard contact sphere. 
 From Corollary~\ref{cor:s1dar}, 
there exists a tubular neighborhood of a transverse circle 
in a $(2n+1)$-dimensional contact manifold 
which is contactomorphic to a tubular neighborhood of $S^1\times\{0\}$ 
in the standard contact manifold $(S^1\times\R^{2n},\xi_0)$. 
 The standard contact structure $\xi_0$ is given as 
\begin{equation*}
  \xi_0=\ker\lft\{d\phi+\sum_{i=1}^n(x_idy_i-y_idx_i)\rgt\}
  =\ker\lft\{d\phi+\sum_{i=1}^nr_i^2d\theta_i\rgt\}, 
\end{equation*}
where $\phi$ is a coordinate of $S^1$, 
$(x_i,y_i)$ are coordinates of $\R^2$,  
and $(r_i,\theta_i)$ are the polar coordinates of $\R^2$, $i=1,\dots,n$. 
 Then we observe the boundary of the tubular neighborhood 
$S^1\times D^{2n}(\rho)\subset S^1\times\R^{2n}$ 
of the transverse core $S^1\times\{0\}$, 
where $D^{2n}(\rho)\subset\R^{2n}$ is the $(2n)$-dimensional disk 
with radius $\rho$. 
 The boundary $S^1\times\rd D^{2n}(\rho)\cong S^1\times S^{2n-1}$ is transverse 
to $\xi_0$ since $\xi_0$ has $\uvv{r_i}$ elements. 
 The contact structure $\xi_0|_{T(\{s\}\times S^{2n-1})}$ 
induced from $\xi_0$ on a sphere $\{s\}\times S^{2n-1}$, $s\in S^1$, is 
the standard contact structure 
$\eta_0:=\ker\lft(\sum_{i=1}^n(x_idy_i-y_idx_i)\rgt)$ 
on a $(2n-1)$-dimensional sphere. 
 Then the boundary $S^1\times\rd D^{2n}(\rho)\subset(S^1\times\R^{2n},\xi_0)$ 
is a $\xi_0$-round hypersurface 
modeled on the standard contact sphere $(S^{2n-1},\eta_0)$. 
\end{expl}

  For $\xi$-round hypersurfaces, 
it is proved in \cite{mnw} that the neighborhood is determined 
by contact submanifold $(N,\eta)$. 
%
%
\begin{prop}\label{prop:xi-rdhyps}
  Let $H=N\times S^1$ be a $\xi$-round hypersurface 
in a contact manifold $(M,\xi)$. 
 Then there exists a small tubular neighborhood of $H\subset(M,\xi)$ 
which is contactomorphic to 
$\lft(H\times(-\epsilon,\epsilon),\ker(\alpha_N+sd\phi)\rgt)$, 
where $\alpha_N$ is a contact form on $N$ for the contact structure $\xi|_{TN}$,
and $s\in(-\epsilon,\epsilon)$, $\phi\in S^1$ are coordinates. 
\end{prop}

\subsubsection{Pre-Lagrangian submanifold}\label{sec:preLag}
  Although we have already defined pre-Lagrangian 
torus in a contact $3$-manifold, 
pre-Lagrangian submanifold in a contact manifold is a more general notion. 
 Let $L$ be an $(n+1)$-dimensional submanifold 
in a $(2n+1)$-dimensional contact manifold $(M,\xi)$. 
 The submanifold is said to be \emph{pre-Lagrangian}\/ 
if it satisfies the following condition: 
\begin{itemize}
\item $L$ is transverse to the contact structure $\xi$, 
\item the distribution $\xi\cap TL$ induced from $\xi$ on $L$ is integral 
  and is defined by a closed $1$-form on $L$. 
\end{itemize}
 In other words, $L$ is foliated by Legendrian leaves. 
 The term ``pre-Lagrangian'' comes from the fact 
that, for a pre-Lagrangian submanifold in a contact manifold, 
there exists a Lagrangian submanifold in the symplectization 
of the contact manifold that projects to the pre-Lagrangian submanifold 
(see \cite{elihofsal}). 

  A pre-Lagrangian submanifold has a characteristic tubular neighborhood. 
 Let $L$ be a pre-Lagrangian submanifold 
in a contact $(2n+1)$-dimensional manifold $(M,\xi)$. 
 Then the Legendrian submanifold $N^n\subset L\subset(M,\xi)$ 
obtained as an integral manifold of $\xi\cap TL$ 
has the standard tubular neighborhood 
which is contactomorphic to the jet space $J^1(N,\R)\to N$ 
with the canonical distribution.
 In this tubular neighborhood, a pre-Lagrangian submanifold $L$ 
is identified with the so called ``$0$-wall'' $N\times\R\subset J^1(N,\R)$, 
that is the locus consists of $1$-jets of functions on $N$ 
with zero differential (see \cite{elihofsal}). 
 Therefore, if a pre-Lagrangian submanifold $L$ has 
a product foliation $N\times\R$ or $N\times S^1$ with Legendrian leaves $N$, 
then $L$ has a unique tubular neighborhood in $(M,\xi)$. 

%
%
\begin{prop}\label{prop:nbdpL}
  Let $L$ be a pre-Lagrangian submanifold 
in a $(2n+1)$-dimensional contact manifold $(M,\xi)$. 
 If $L$ has a product Legendrian foliation, 
there exists a unique, up to contactomorphism, 
tubular neighborhood in $(M,\xi)$.
\end{prop}

\subsubsection{Convex hypersurface}\label{sec:cvhsrf}
 Theory of convex hypersurface is introduced by Giroux~\cite{giroux99} 
and has been a valuable tool in the $3$-dimensional contact topology. 
 It is also valid in higher dimensions. 
 We review some basic things and neighborhood theorems 
that are needed for contact round surgery 
in Section~\ref{sec:ctrdsurg} and~\ref{sec:gentw}. 

  First, we review the notion of characteristic foliation 
before introducing convexity. 
 It has already been defined for a surface in a contact $3$-manifold 
in Subsection~\ref{sec:ovtw}. 
 For general dimensions, it is defined as follows. 
 Let $\Sigma$ be a hypersurface 
in a $(2n+1)$-dimensional contact manifold $(M,\xi)$. 
 There exists a conformal symplectic structure on $\xi$ 
since $\xi$ is completely non-integrable. 
 Taking the skew-orthogonal of the intersection 
$\xi_x\cap T_x\Sigma\subset\xi_x$ 
at each point $x\in\Sigma$, 
we have a $1$-dimensional distribution with singularities on $\Sigma$. 
 Further, by integrating the distribution, 
we have a singular $1$-dimensional foliation $\Sigma_\xi$ on $\Sigma$. 
 Singular points are points where $\xi_x=T_x\Sigma$. 
 We call this singular foliation $\Sigma_\xi$ 
the \emph{characteristic foliation}\/ of $\Sigma$ with respect to $\xi$. 
 Locally the characteristic foliation $\Sigma_\xi$ is defined 
by the vector field $X$ that satisfies 
$X\lrcorner\mu_\Sigma=\alpha\wedge(d\alpha)^{n-1}|_{T\Sigma}$, 
where $\mu_{\Sigma}$ is a volume form on $\Sigma$, 
and $\alpha$ is a contact form for $\xi$. 
 One of the important properties is that a characteristic foliation determines 
a germ of contact structures along a hypersurface (see~\cite{giroux91}). 
%
%
\begin{prop}[Giroux]\label{prop:charfol}
  Let $\Sigma$ be a closed hypersurface in a contact manifold $(M,\xi)$. 
 Then two germs of contact structures 
which have the same characteristic foliation are isomorphic. 
\end{prop}

  The convexity of a hypersurface in a contact $(2n+1)$-dimensional manifold 
is defined as follows. 
 It is defined by using a certain vector field. 
 A vector field $X$ on a $(2n+1)$-dimensional contact manifold $(M,\xi)$ 
is said to be \emph{contact\/} 
if its flow $\phi_t=\operatorname{Exp}(tX)$ 
preserves the contact structure $\xi$: $\lft({\phi_t}\rgt)_\ast\xi=\xi$. 
 A hypersurface $\Sigma$ 
embedded in a $(2n+1)$-dimensional contact manifold $(M,\xi)$ 
is said to be \emph{convex\/} 
if there exists a contact vector field on a neighborhood of $\Sigma\subset M$ 
which is transverse to $\Sigma$. 
 A convex hypersurface has a vertically invariant tubular neighborhood. 

  One of the most important properties is the flexibility 
of characteristic foliations. 
 In order to describe the property, we need the following notions. 
 Let $\Sigma$ be a convex hypersurface 
in a $(2n+1)$-dimensional contact manifold $(M,\xi)$, 
and $X$ a contact vector field defined near $\Sigma$ 
which is transverse to $\Sigma$. 
 The \emph{dividing set}\/ $\Gamma_\Sigma$ of $\Sigma$ is defined as 
\begin{equation*}
  \Gamma_\Sigma:=\{x\in\Sigma\mid\xi_x\ni X_x\}. 
\end{equation*}
 In other words, it is a set of points where the contact hyperplane 
gets ``vertical'' to $\Sigma$. 
 A dividing set $\Gamma_\Sigma$ is a hypersurface in $\Sigma$ that is 
transverse to leaves of a characteristic foliation $\Sigma_\xi$. 
 On the other hand, 
a $1$-dimensional foliation $\mathcal{F}$ on $\Sigma$ with singularity
is said to be \emph{divided\/} by $\Gamma_\Sigma$ if
\begin{itemize} \setlength{\itemsep}{0cm}\setlength{\parskip}{0cm} 
  \item $\Sigma\setminus\Gamma_\Sigma$ is divided into two kinds of regions: 
  $\Sigma\setminus\Gamma_\Sigma=U_+\sqcup U_-$, 
\item any leaf of $\mathcal{F}$ is transverse to $\Gamma_\Sigma$, 
\item there exists a vector field $v$ which is tangent to $\mathcal{F}$ 
  and a volume form $\omega$ on $\Sigma$ which satisfy: 
  \begin{enumerate} \setlength{\itemsep}{0cm}\setlength{\parskip}{0cm} 
    \renewcommand{\labelenumi}{\textup{(\theenumi)}}
  \item $\operatorname{div}_\omega v>0$ on $U_+$, 
    $\operatorname{div}_\omega v<0$ on $U_-$, 
  \item the vector field $v$ looks outward of $U_+$ at $\Gamma_\Sigma$. 
  \end{enumerate}
\end{itemize}
 A characteristic foliation is clearly divided by the dividing set. 
 It is also proved that if a characteristic foliation is divided by some set 
then the closed hypersurface is convex (see~\cite{giroux99}). 
 The following theorem implies that dividing set has certain essential parts
of characteristic foliation. 

%
%
\begin{thrm}[Giroux, \cite{giroux99}]\label{thm:convFlex}
  Assume that $\Sigma$ is a convex hypersurface 
in a $(2n+1)$-dimensional contact manifold $(M,\xi)$ 
with a contact vector field $X$ transverse to $\Sigma$, 
and that $\Gamma_\Sigma\subset\Sigma$ is a dividing set. 
 Let $\mathcal{F}$ be a $1$-dimensional foliation with singularity on $\Sigma$ 
divided by $\Gamma_\Sigma$. 
 Then there exists a family $\phi_t\colon\Sigma\to M$, $t\in[0,1]$, 
of embeddings which satisfies\textup{:} 
\begin{itemize} \setlength{\itemsep}{0cm}\setlength{\parskip}{0cm} 
\item $\phi_0=\textup{id}_\Sigma$, $\phi_t|_{\Gamma_\Sigma}
  =\textup{id}_{\Gamma_\Sigma}$, 
  for any $t\in[0,1]$,
\item $\phi_t(\Sigma)$ is transverse to $X$ for any $t\in[0,1]$,
\item $(\phi_1(\Sigma))_\xi=(\phi_1)_\ast\mathcal{F}$. 
\end{itemize}
\end{thrm}
\noindent
 This theorem implies that, in order to glue two contact manifolds, 
it is sufficient 
to check dividing sets on the boundaries and their orientations. 
 In fact, a characteristic foliation on a surface 
embedded in a contact manifold determines a germ of contact structure 
(see Proposition~\ref{prop:charfol}). 
 Then Theorem~\ref{thm:convFlex} implies a dividing set on the surface 
determines a germ, or the invariant tubular neighborhood, 
of contact structure.

\subsection{Confoliation}\label{sec:confoli}
  We review confoliations in this subsection. 
 Confoliation is a certain generalization of contact structure. 
 The notion includes both contact structure and foliation. 
 In the $3$-dimensional case, it is studied in \cite{elith}. 
 In this paper, we need results in higher dimensional case, as well. 
 We refer \cite{altwu} for that. 

  First of all, confoliation is defined as follows. 
 Let $M$ be an oriented $(2n+1)$-dimensional manifold with a Riemannian metric.
 A hyperplane field $\xi$ is determined, at least locally, 
as the kernel $\xi=\ker\alpha$ of a $1$-form $\alpha$ vanishing nowhere. 
%
%
\begin{dfn}
  A hyperplane field $\xi=\ker\alpha$ is called a (positive)
  \emph{confoliation} if it satisfies the following inequality:
  \begin{equation}\label{eq:confoli}
    \ast\{\alpha\wedge(d\alpha)^n\}\ge0, 
  \end{equation}
where $\ast$ is the Hodge star operator. 
\end{dfn} 
\noindent
 If $\xi$ is a contact structure, then $\ast\{\alpha\wedge(d\alpha)^n\}>0$.  

  One of the important properties is 
that, under some condition, a confoliation is deformed to a contact structure. 
 This property is studied in \cite{alt}, \cite{elith} for dimension $3$, 
and in \cite{altwu} for general dimensions. 

  In order to describe the condition, 
we need the following notion of the ``conductivity'' of contactness. 
 Let $M$ be an oriented $(2n+1)$-dimensional manifold with a Riemannian metric, 
and $\xi=\ker\alpha$ a positive confoliation on $M$. 
 Since $\xi=\ker\alpha$ is a confoliation, 
\begin{equation*}
  \tau(\alpha):=\ast\{\alpha\wedge(d\alpha)^{n-1}\}
\end{equation*}
is a differential $2$-form, by the defining inequality~\eqref{eq:confoli}. 
 Set $\nul{\tau(\alpha)}:=\{X\in TM\mid X\intp\tau(\alpha)=0\}$. 
 Moreover, let $H(\xi)\subset M$ denote the \emph{hot zone}, that is, 
the subset where $\xi$ is a contact structure: 
\begin{equation*}
  H(\xi):=\{x\in M\mid \ast\{\alpha\wedge(d\alpha)^n\}>0\}
  =\{x\in M\mid \alpha\wedge(d\alpha)^n>0\}. 
\end{equation*}
%
%
\begin{dfn}
  A confoliation $\xi$ is said to be \emph{conductive}\/ 
if any point in $M$ is connected to the contact part $H(\xi)\subset M$ 
by a path everywhere tangent to $\nul{\tau(\alpha)}^\perp$. 
\end{dfn} 

  Then the following result on deforming a confoliation to a contact structure 
was proved by Altschuler and Wu~\cite{altwu}. 
%
%
\begin{thrm}[Altschuler-Wu]\label{thm:approx}
  Any conductive confoliation on a compact orientable manifold 
is $C^\infty$-close to a contact structure. 
\end{thrm} 

\section{Generalization of Lutz twist}
\label{sec:genlztw}
  A generalization of the Lutz twist for contact structures 
on higher dimensional manifolds is introduced in this section. 
 This operation is defined in Subsection~\ref{sec:defGenLztw} 
as a modification of a contact structure along a certain embedded circle. 
 Then, in Subsection~\ref{sec:handftwists}, we discuss important properties 
of the half and full generalized Lutz twists. 
 These arguments amount to the proof of Theorem~\ref{thm_main} 
except overtwistedness. 
 The overtwistedness is discussed in Section~\ref{sec:newot}. 
 The key of this new notion is the generalized Lutz tube, 
which is introduced in Subsection~\ref{sec:gLtube}. 
 Further, in Subsection~\ref{sec:non-cpt}, we define a generalized Lutz twist 
along a line instead of circle. 

  We should remark that the definitions of these modifications are deduced 
from the symplectic round handle (see Section~\ref{sec:gentw}), 
although the definitions can be done without using it. 

\subsection{Generalization of the Lutz tube}\label{sec:gLtube}
  First of all, we introduce a new notion of generalized Lutz tube 
for higher-dimensional contact manifolds. 
 Recall that the original Lutz twist for $3$-dimensional contact manifolds 
is defined as an operation 
replacing the standard tubular neighborhood of a transverse knot 
with the so-called Lutz tube. 
 We will define a higher-dimensional generalization of the twist 
as a modification along a transverse circle in a contact manifold 
in a similar way. 
 Then, for the definition, 
we need a certain contact structure on the tubular neighborhood 
of a transverse circle 
as a generalization of the Lutz tube. 

  The generalized Lutz tube should be defined as a tubular neighborhood 
$S^1\times\R^{2n}$ of a circle with a certain contact structure. 
 However, not constructing such a model directly, 
we construct a prototype first. 
 The underlying manifold is $S^1\times\R^{2n}$, 
that is an open tubular neighborhood of $S^1\times\{0\}$. 
 Let $(\phi,r_1,\theta_1,\dots,r_n,\theta_n)$ be coordinates 
of $S^1\times\R^{2n}$. 
 Setting 
\begin{align}\label{eq:genlutzf}
  \omtw:=&\lft(\prod_{i=1}^{n}\cos r_i^2\rgt)d\phi
  +\sum_{i=1}^{n}\sin r_i^2d\theta_i \notag \\
  =&\{(\cos r_1^2)\cdots(\cos r_n^2)\}d\phi
  +(\sin r_1^2)d\theta_1+\dots+(\sin r_n^2)d\theta_n, 
\end{align}
we have a hyperplane field $\zeta:=\ker\omega_{\mathtt{tw}}$ 
on $S^1\times\mathbb{R}^{2n}$. 
 This $\zeta$ is not a contact structure but a confoliation. 
 In fact, 
\begin{align*}
  &\omtw\wedge(d\omtw)^n \notag \\
  =&2^{n}n!\lft\{(\cos^2r_1^2)\cdots(\cos^2r_n^2)
  +\sum_{i=1}^n(\cos^2r_1^2)\cdots(\cos^2r_{i-1}^2)(\sin^2r_i^2)(\cos^2r_{i+1}^2)
  \cdots(\cos^2r_n^2)\rgt\} \notag \\
  &\hspace{5cm} \hfill r_1r_2\cdots r_{n}d\phi\wedge dr_1\wedge d\theta_1\wedge\cdots
  \wedge dr_n\wedge d\theta_n. 
\end{align*}
 In other words, 
\begin{align}  \label{eq:protoctf}
  &\ast \lft\{\omtw\wedge(d\omtw)^n\rgt\} \notag \\
  =&2^{n}n!\lft\{(\cos^2r_1^2)\cdots(\cos^2r_n^2)
  +\sum_{i=1}^n(\cos^2r_1^2)\cdots(\cos^2r_{i-1}^2)(\sin^2r_i^2)(\cos^2r_{i+1}^2)
  \cdots(\cos^2r_n^2)\rgt\}\ge0. 
\end{align}
 Then, by the defining inequality~\eqref{eq:confoli}, 
the hyperplane field $\zeta=\ker\omtw$ is a confoliation. 

%
%
\begin{rem}\label{rem:3dimomtw}
  In the case of dimension $3$, the hyperplane field $\zeta=\ker\omtw$ 
is a contact structure. 
 $(S^1\times\R^2,\zeta)$ is nothing but the original Lutz tube. 
 In fact, $\omtw=(\cos r_1^2)d\phi+(\sin r_1^2)d\theta_1$. 
\end{rem} 

  A generalization of the Lutz tube for higher dimensions is obtained
from the prototype $\lft(S^1\times\R^{2n},\zeta\rgt)$ above 
by deforming the confoliation $\zeta$ into a contact structure 
after some procedures (See in Subsection~\ref{sec:defGenLztw}). 
 In order to apply the approximation result, Theorem~\ref{thm:approx}, 
we need the following. 
%
%
\begin{prop}\label{prop:conductive}
  The confoliation $\zeta=\ker\omtw$ on $S^1\times\R^{2n}$ is conductive. 
\end{prop} 

\begin{proof}
  According to the definition of conductivity 
(see Subsection~\ref{sec:confoli}), 
we show that any point in the non-contact locus 
$\Sigma(\zeta)=(S^1\times\R^{2n})\setminus H(\zeta)$ of the confoliation $\zeta$
is connected by a path tangent to $\nul{\tau(\omtw)}^\perp$ 
to a point in the hot zone $H(\zeta)$, where $\zeta$ is contact. 

  The non-contact locus $\Sigma(\zeta)$ of $\zeta=\ker\omtw$ is obtained 
from Inequality~\eqref{eq:protoctf} as follows: 
\begin{align}\label{eq:nonctloc}
  \Sigma(\zeta)
  &=\lft\{(\phi,r_1,\theta_1,\dots,r_n,\theta_n)\in S^1\times\R^{2n}\;\lft|\;
  \omtw\wedge(d\omtw)^n=0\rgt.\rgt\} \notag \\
  &=\bigcup_{\substack{i,j=1\\i\ne j}}^n\lft\{\cos r_i^2=0,\ \cos r_j^2=0\rgt\} 
   =\bigcup_{\substack{i,j=1\\i\ne j}}^n\lft\{(\cos r_i^2)^2+(\cos r_j^2)^2=0\rgt\} 
   \notag \\
  &=\bigcup_{\substack{i,j=1\\i\ne j}}^n
  \lft\{\lft.r_i=\sqrt{\lft(\frac{1}{2}+l\rgt)\pi},\ 
  r_j=\sqrt{\lft(\frac{1}{2}+m\rgt)\pi}\ \rgt|\ l,m\in\mathbb{N}\rgt\}. 
\end{align}
 Then the non-contact locus $\Sigma(\zeta)$ has a stratification 
each of whose strata is of the form $\{r_{i_1}=c_{i_1},r_{i_2}=c_{i_2},\dots\}$, 
where $i_1,i_2,\dots\in\{1,2,\dots,n\}$, and $c_{i_k}$ are constant. 
 The codimension of the highest dimensional stratum is two. 

  Next, we observe the constraint 
imposed on paths in $(S^1\times\R^{2n},\zeta)$. 
 The $2$-form $\tau$, 
in the definition of conductivity in Subsection~\ref{sec:confoli}, 
for $\omtw$ is calculated as follows: 
\begin{align}
  \omtw\wedge(d\omtw)^{n-1}
  =&2^{n-1}(n-1)!\lft\{\sum_{i\ne j}(\cos^1r_j^2)(\cos^2r_1^2)\cdots
  \widehat{(\cos^2r_i^2)}\cdots\widehat{(\cos^2r_j^2)}
  \cdots(\cos^2r_n^2)\rgt. \notag \\
  &\quad\qquad\qquad\lft.\lft.d\phi\wedge(r_1dr_1\wedge d\theta_1)\wedge
  \dots\wedge\widehat{(r_{j}dr_j\wedge d\theta_j)}\wedge
  \dots\wedge(r_{n}dr_n\wedge d\theta_n)\rgt.\rgt. \notag \\
  &\lft.\lft.+\sum_{j=1}^n(\sin r_j^2)(\cos r_1^2)\cdots\widehat{(\cos r_j^2)}
  \cdots(\cos r_n^2)\rgt.\rgt. \notag \\
  &\quad\qquad\qquad\lft.\lft.d\theta_j\wedge(r_1dr_1\wedge d\theta_1)
  \wedge\cdots\wedge\widehat{(r_{j}dr_j\wedge d\theta_j)}\wedge
  \cdots\wedge(r_{n}dr_n\wedge d\theta_n)\rgt.\rgt. \notag \\
  &\lft.\lft.+\sum_{i\ne j}(\sin r_i^2)(\sin r_j^2)(\cos r_j^2)
  (\cos^2 r_1^2)\cdots\widehat{(\cos^2 r_i^2)}\cdots\widehat{(\cos^2 r_j^2)}
  \cdots(\cos^2 r_n^2)\rgt.\rgt. \notag \\
  &\quad\qquad\qquad\lft.\lft.d\phi\wedge(r_{i}dr_i)\wedge d\theta_j
  \wedge(r_1dr_1\wedge d\theta_1)\wedge
  \dots\wedge\widehat{(r_{i}dr_i\wedge d\theta_i)}\wedge\cdots \rgt.\rgt. 
    \notag\\
  &\hspace{5cm}
  \wedge\widehat{(r_{j}dr_j\wedge d\theta_j)}\wedge
  \dots\wedge(r_{n}dr_n\wedge d\theta_n)\Biggr\}, 
\end{align}
then 
\begin{align}
  \tau(\omtw)=&\ast\lft\{\omtw\wedge(d\omtw)^{n-1}\rgt\} \notag \\
  =&2^{n-1}(n-1)!\lft\{\sum_{i\ne j}(\cos r_j^2)(\cos^2r_1^2)\cdots
  \widehat{(\cos^2r_i^2)}\cdots\widehat{(\cos^2r_j^2)}\cdots(\cos^2r_n^2)
  \ (r_{j}dr_j)\wedge d\theta_j\rgt. \notag \\
  &+\sum_{j=1}^n(\sin r_j^2)(\cos r_1^2)\cdots\widehat{(\cos r_j^2)} 
  \cdots(\cos r_n^2)\ d\phi\wedge(r_{j}dr_j) \notag \\
  &\lft.+\sum_{i\ne j}(\sin r_i^2)(\sin r_j^2)(\cos r_j^2)
  (\cos^2 r_1^2)\cdots\widehat{(\cos^2 r_i^2)}\cdots\widehat{(\cos^2 r_j^2)}
  \cdots(\cos^2 r_n^2)\ d\theta_i\wedge(r_{j}dr_j)\rgt\} \notag \\
  =&2^{n-1}(n-1)!\sum_{i\ne j}(\cos r_1^2)\cdots
  \widehat{(\cos r_i^2)}\cdots\widehat{(\cos r_j^2)}\cdots(\cos r_n^2) \notag \\
  &\hspace{2.5cm}\lft[(\cos r_1^2)\cdots\widehat{(\cos r_i^2)}\cdots(\cos r_n^2)
  \lft\{(\sin r_i^2)(\sin r_j^2)d\theta_i-d\theta_j\rgt\}\rgt. \notag \\
  &\hspace{7.5cm}\lft.+(\cos r_i^2)(\sin r_j^2)d\phi\rgt]
  \wedge(r_{j}dr_j), 
\end{align}
where the symbol ``$\widehat{\mbox{\qquad}}$'' implies 
to eliminate the indicated part. 
 On the non-contact locus $\Sigma(\zeta)$, 
the form $\omtw\wedge(d\omtw)^{n-1}$ vanishes from the formula above. 
 Then $\nul{\tau}^\perp=\{0\}$ there. 
 In the following, we observe $\nul\tau^\perp$ on the hot zone 
$H(\zeta)=(S^1\times\R^{2n})\setminus\Sigma(\zeta)$. 
 From the formula above, $\nul{\tau(\omtw)}$ is 
in $\spn{\uvv\phi,\uvv\theta_1,\dots,\uvv\theta_n}$ 
spanned by vector fields $\uvv\phi,\uvv\theta_1,\dots,\uvv\theta_n$. 
 It follows that vector fields $\uvv{r_1},\dots,\uvv{r_n}$ 
are in $\nul{\tau(\omtw)}$. 
 Therefore, a path always tangent to these directions 
satisfies the required condition. 

  Now, we confirm that any point in the non-contact locus $\Sigma(\zeta)$ 
is connected to a point in the hot zone 
$H(\zeta)=(S^1\times\R^{2n})\setminus\Sigma(\zeta)$. 
 Recall that $\Sigma(\zeta)$ has the stratification 
from Formula~\eqref{eq:nonctloc}. 
 Each stratum is of the form $\{r_{i_1}=c_{i_1},r_{i_2}=c_{i_2},\dots\}$. 
 Any such strata is not tangent to the vector field 
$X=r_1(\uvv{r_1})+r_2(\uvv{r_2})+\dots+r_n(\uvv{r_n})$. 
 The vector field $X$ is tangent to $\nul{\tau(\omtw)}^\perp$ 
at any point of $H(\zeta)$ from the argument in the paragraph above. 
 Then for any $x\in\Sigma(\zeta)$, we have a curve $\gamma_x(t)$, $t=[0,1]$, 
which satisfies $\dot\gamma_x(t)\in\nul{\tau(\omtw)}^\perp$, 
$\gamma_x(0)=x$, $\dot\gamma_x(0)=0$, $\gamma_x(1)\in H(\zeta)$. 

  Thus, we complete the proof. 
\end{proof}

  In order to regard the obtained manifold $\lft(S^1\times\R^{2n},\zeta\rgt)$ 
with the conductive confoliation $\zeta=\ker\omtw$ 
as a generalization of the Lutz tube, 
it should have some properties. 
 The reason why we take the Lutz tube in dimension~$3$ is 
that its contact structure is overtwisted (see Subsection~\ref{sec:lutztwist}).
 Then a generalization is supposed to be PS-overtwisted 
or overtwisted in the sense of~\cite{newOT}. 
 Actually, $\lft(S^1\times\R^{2n}, \zeta\rgt)$ has the following property. 
%
%
\begin{prop}\label{prop:exblob}
  The $(2n+1)$-dimensional manifold 
$\lft(S^1\times\R^{2n},\zeta\rgt)$ with a confoliation 
contains, in the contact part, a bordered Legendrian open book with 
an $(n-1)$-dimensional torus $T^{n-1}$ as the binding. 
\end{prop} 
\begin{proof}
  In the $3$-dimensional case (i.e.\ $n=1$), 
$\zeta=\ker\omtw$ is already a contact structure, 
and $(S^1\times\R^2,\zeta)$ is the Lutz tube (see Remark~\ref{rem:3dimomtw}). 
 In the Lutz tube, there exists a bordered Legendrian open book, 
an overtwisted disk in this case (see Subsection~\ref{sec:lutztwist}). 
 Then, in the rest of this proof, we deal with the case when $n>1$. 

  First, we find a bordered Legendrian open book 
in the prototype $\lft(S^1\times\R^{2n},\zeta=\ker\omtw\rgt)$, 
although it is not contact. 
 The object to be observed is the submanifold 
\begin{equation}\label{eq:blobfam}
  P:=\lft\{(\phi,r_1,\theta_1,\dots,r_n,\theta_n)\in S^1\times\R^{2n}\;\lft|\;
  r_1\le\sqrt{\pi},\ r_2=\dots=r_n=\sqrt{\pi}\rgt.\rgt\} 
\end{equation}
(see Figure~\ref{fig:tobeblob}). 
%
%
\begin{figure}[htb]
  \centering
  {\small \includegraphics[height=3.3cm]{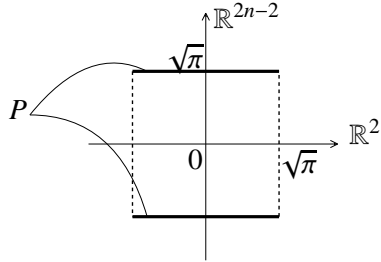}}
  \caption{Object to be a (family of) \textsf{bLob} $\pmod{\times S^1}$.}
  \label{fig:tobeblob}
\end{figure} 
 It is diffeomorphic to 
$S^1\times D^2\times S^1\times\dots\times S^1=S^1\times D^2\times T^{n-1}$. 
 We show that $P$ is an $S^1$-family of bordered Legendrian open books 
with the torus $T^{n-1}$ as bindings. 
 We should remark that the submanifold 
$P\subset\lft(S^1\times\R^{2n},\zeta\rgt)$ 
lies in the contact part $H(\zeta)\subset\lft(S^1\times\R^{2n},\zeta\rgt)$ 
of the confoliation $\zeta=\ker\omtw$. 
 In fact, recall that the con-contact locus is 
\begin{equation*}
  \Sigma(\zeta)
  =\bigcup_{\substack{i,j=1\\i\ne j}}^n\lft\{(\phi,r_1,\theta_1,\dots,r_n,\theta_n)
  \in S^1\times\R^{2n}\;
  \lft|\; r_i=\sqrt{\lft(\frac12+l\rgt)\pi},\ 
  r_j=\sqrt{\lft(\frac12+m\rgt)\pi},\ 
  l,m\in\mathbb{N}\rgt.\rgt\}
\end{equation*}
from Formula~\eqref{eq:nonctloc}. 
 By comparing this with Formula~\eqref{eq:blobfam}, 
there is no intersection between $P$ and $\Sigma(\zeta)$. 
 In other words, $P$ is in the contact part 
$H(\zeta)=(S^1\times\R^{2n})\setminus\Sigma(\zeta)$. 
 Since $P$ is in $H(\zeta)$, the notion ``Legendrian'' makes sense 
for submanifolds in $P$. 

  In order to find bordered Legendrian open books, 
we regard the submanifold $P$ as follows. 
 Set 
\begin{equation*}
  N:=\lft\{(\phi,r_1,\theta_1,\dots,r_n,\theta_n)
  \in S^1\times\R^{2n}\;
  \lft|\; \phi=\phi_0,r_1\le\sqrt{\pi},r_2=\dots=r_n=\sqrt{\pi}\rgt.\rgt\}, 
\end{equation*}
for some constant $\phi_0\in S^1$, which is diffeomorphic to $D^2\times T^{n-1}$.
 Then $P$ is regarded as an $S^1$-family of $N$. 
 Like the Lutz tube has an $S^1$-family of the overtwisted disks, 
we show that this $N$ is a bordered Legendrian open book. 
 The binding for open book structure is 
\begin{equation*}
  B:=\lft\{(\phi_0,r_1,\theta_1,\dots,r_n,\theta_n)\in S^1\times\R^{2n}\;
  \lft|\; r_1=0,r_2=\dots=r_n=\sqrt{\pi}\rgt.\rgt\}, 
\end{equation*}
which is a closed manifold, a torus $T^{n-1}$, in $\inter N$. 
 In fact, we have a fibration $p\colon N\setminus B\to S^1$ defined as 
$p\colon\lft(\phi_0,r_1,\theta_1,\sqrt{\pi},\theta_2,\dots,
\sqrt{\pi},\theta_n\rgt)\mapsto \theta_1$. 
 The fiber, or page, of $p\colon N\setminus B\to S^1$ on $\bar\theta\in S^1$ is 
\begin{align*}
  F_{\bar\theta}:=&\lft\{\lft.\lft(\phi_0,r_1,\bar{\theta},\sqrt{\pi},\theta_2,
  \dots,\sqrt{\pi},\theta_n\rgt)\in S^1\times\R^{2n}\;\rgt|\; 
  0<r_1\le\sqrt{\pi}, \theta_2\in S^1,\dots,\theta_n\in S^1\rgt\} \\
  \cong&\lft(\lft.0,\sqrt{\pi}\rgt.\rgt]\times T^{n-1}, 
\end{align*}
which is transverse in $N$ to the boundary 
\begin{equation*}
  \rd N=\lft\{\lft.\lft(\phi_0,\sqrt{\pi},\theta_1,\sqrt{\pi},\theta_2,\dots,
  \sqrt{\pi},\theta_n\rgt)\in S^1\times\R^{2n}\;\rgt|\;
  \theta_1\in S^1,\dots,\theta_n\in S^1\rgt\}\cong T^n. 
\end{equation*}
 It is because 
$TF_{\bar\theta}+T(\rd N)=\spn{\uvv{r_1},\uvv{\theta_2},\dots,\uvv{\theta_n}}
+\spn{\uvv{\theta_1},\uvv{\theta_2},\dots,\uvv{\theta_n}}=TN$. 

  Each fiber $F_{\bar{\theta}}$ and the boundary $\rd N$ are Legendrian. 
 It is explained by observing the pull-backs of $\omtw$ by the inclusions 
$i_{\bar\theta}\colon F_{\bar\theta}\hookrightarrow S^1\times\R^{2n}$ and 
$i_{\rd}\colon \rd N\hookrightarrow S^1\times\R^{2n}$. 
 We have 
\begin{equation*}
  i_{\bar\theta}^\ast\omtw
  =(\cos r_1^2)d(i_{\bar\theta}^{\ast}\phi)
  +(\sin r_1^2)d(i_{\bar\theta}^{\ast}\theta_1)=0, \qquad 
  i_{\rd}^\ast\omtw=d(i_{\rd}^{\ast}\phi)=0. 
\end{equation*}
 These equations imply that they both are Legendrian. 
 Thus, we have confirmed that $N$ is a bordered Legendrian open book. 
\end{proof}
\noindent
 We discuss the relation between $\lft(S^1\times\R^{2n},\zeta\rgt)$ 
and the higher-dimensional overtwisted disc in the sense of~\cite{newOT} 
in Section~\ref{sec:newot}. 

  Now, we have a reason why we regard what we introduced above 
as a generalization of the Lutz tube. 
%
%
\begin{dfn}
  We call the manifold $\lft(S^1\times\R^{2n},\zeta\rgt)$ 
with the confoliation $\zeta=\ker\omtw$ 
the \emph{generalized open Lutz tube}. 
\end{dfn} 

\subsection{Generalization of the Lutz twist}\label{sec:defGenLztw}
  In this subsection, we define a generalization of the Lutz twist 
along a transverse circle 
by using the generalized open Lutz tube defined in the previous subsection. 
 The definition is not a simple replacement 
with a part of the generalized open Lutz tube. 
 Before the definition, we observe the reason 
from the view point of contact round surgery. 
 First, we review the contact round surgery realization 
of the $3$-dimensional Lutz twist. 
 In order to apply the method to higher dimensions, 
we prepare certain contact manifolds to be replaced. 
 Then we define generalizations of the half and full Lutz twists. 

 \subsubsection{Review of a round surgery representation 
of the Lutz twist. }\label{sec:revrdsurgrep}
  First of all, we roughly review the outline of 
a round surgery realization of the $3$-dimensional Lutz twist 
(see \cite{art20}). 
 Although the definition of the contact round surgery will be given later 
in Section~\ref{sec:ctrdsurg}, 
the basic idea is important to define a generalization. 

  We try to execute the Lutz twist along a transverse knot $\Gamma$ 
in a contact $3$-manifold $(M,\xi)$. 
 First, we operate a contact round surgery of index~$1$ along $\Gamma$ 
and the core transverse circle $\Gamma_0=S^1\times\{0\}$ 
in the Lutz tube 
$\lft(S^1\times\R^2,\ker\{(\cos r^2)d\phi+(\sin r^2)d\theta\}\rgt)$. 
 That is a fiberwise connected sum of fibrations on $S^1$. 
 Then we operate a contact round surgery of index~$2$ 
along a certain $2$-dimensional torus 
so that the overtwisted disks are taken in. 
 The part taken in in this procedure is a solid torus, 
or a tubular neighborhood of a circle. 

  The important thing is that the core of this part is added 
by the second surgery. 
 And that the neighborhood of the core $\Gamma_0$ in the Lutz tube 
corresponds to the neighborhood of the boundary of the solid torus taken in 
(see Figure~\ref{fig:3dimsurg}, thick circles). 
%
%
\begin{figure}[htb]
  \centering
  {\small \includegraphics[width=13.8cm]{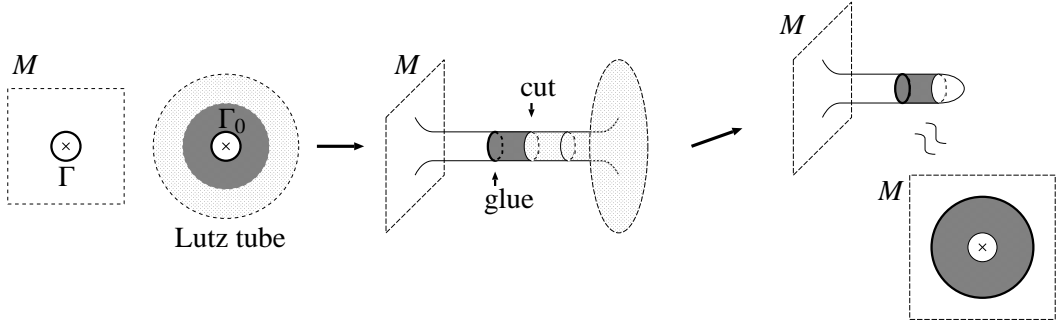}}
  \caption{Lutz twist by round surgeries $\pmod{\times S^1}$.}
  \label{fig:3dimsurg}
\end{figure} 
 From this point of view, the $3$-dimensional Lutz twist 
along a transverse knot 
is not a simple replacement. 

\subsubsection{Double of the generalized Lutz tube and Model Lutz tube. }
\label{sec:dbl-mlt}
  In order to realize the idea above in higher dimensions, 
we introduce a new notion, the double of the generalized Lutz tube. 
 This corresponds to the second round surgery of index $2n$, 
or index $2$ for the $3$-dimensional case. 
 Then we construct a tubular neighborhood of a circle with a contact structure 
to be replaced, 
by removing some tubular neighborhood of the transverse core $\Gamma_0$ 
in one of the generalized Lutz tube. 
 The obtained tubular neighborhood 
is not a part of the generalized open Lutz tube in higher dimensions. 
 For the construction bellow, we regard a disk $D^{2n}\subset\R^{2n}$ 
as a product $D^2\times\dots\times D^2\subset\C^n=\R^{2n}$. 
 Set 
\begin{equation*}
  U(r):=\lft\{\lft.(\phi,r_1,\theta_1,\dots,r_n,\theta_n)\in S^1\times\R^2
  \;\rgt|\; 0\le r_i\le r,\ i=1,2,\dots,n\rgt\}\cong S^1\times D^{2n}. 
\end{equation*}

  First, we make the double of the generalized Lutz tube. 
 Let $\lft(S^1\times\R^{2n},\zeta\rgt)$ be the generalized open Lutz tube. 
 We take the double of $\lft(U(\sqrt{\pi}),\zeta\rgt)
\subset\lft(S^1\times\R^{2n},\zeta\rgt)$. 
 In other words, we glue two $\lft(U(\sqrt{\pi}),\zeta\rgt)$ 
along their boundaries $\rd U(\sqrt{\pi})$. 
 The boundary $\rd U(\sqrt{\pi})$ is divided as follows: 
\begin{align*}
  \rd U(\sqrt{\pi})
  =&\bigcup_{i=1}^n\lft\{(\phi,r_1,\theta_1,\dots,r_n,\theta_n)\in U(\sqrt{\pi})
  \;\lft\vert\; r_i=\sqrt{\pi}\rgt.\rgt\} \\
  =&\lft\{S^1\times\lft(\rd D^2(\sqrt{\pi})\times D^2(\sqrt{\pi})\times\dots 
  \times D^2(\sqrt{\pi})\rgt)\rgt\}\cup\cdots \\
  &\qquad\qquad\cdots\cup\lft\{S^1\times\lft(D^2(\sqrt{\pi})\times\dots
  \times D^2(\sqrt{\pi})\times\rd D^2(\sqrt{\pi})\rgt)\rgt\}. 
\end{align*}
 Let $(\rd U)_i$, $i=1,2,\dots,n$, denote each part 
\begin{align*}
  &\lft\{(\phi,r_1,\theta_1,\dots,r_n,\theta_n)\in S^1\times\R^2\;\lft\vert\; 
  r_1\le\sqrt{\pi},\dots,r_i=\sqrt{\pi},\dots,r_n\le\sqrt{\pi}\rgt.\rgt\} \\
  \cong
  &S^1\times D^2\times\dots\times\overset{(i)}{\rd D^2}\times\dots\times D^2. 
\end{align*}
 At each part $(\rd U)_i$ of the boundary $\rd U(\sqrt{\pi})$, 
the confoliation $1$-form is 
\begin{equation*}
  \omtw|_{(\rd U)_i}=-\lft(\prod_{\substack{j=1\\ j\ne i}}^n(\cos r_j^2)\rgt)d\phi
  +\sum_{\substack{j=1\\ j\ne i}}^{n}(\sin r_j^2)d\theta_j. 
\end{equation*}
 Note that it is independent of $r_i$ and $dr_i$. 
 Therefore, confoliation hyperplane fields 
on two $\lft(U(\sqrt{\pi}),\zeta\rgt)$ 
agree at each $(\rd U)_i$ as boundaries of opposite sides. 
 Near the part $(\rd U)_i$ of the boundary $\rd U(\sqrt{\pi})$, 
there exists the standard tubular (collar) neighborhood 
$(\rd U)_i\times(-\epsilon,\epsilon)\cong
(S^1\times D^2\times\dots\times\rd D^2\times\dots\times D^2)
\times(-\epsilon,\epsilon)$ with the confoliation $1$-form 
\begin{equation*}
  \bar\omega_{\mathtt{tw}}
  =(\cos t)\lft(\prod_{\substack{j=1\\ j\ne i}}^n\cos r_j^2\rgt)d\phi 
  +(\sin t)d\theta_i+\sum_{\substack{j=1\\ j\ne i}}^n(\sin r_j^2)d\theta_j. 
\end{equation*}
 Then, by gluing two $\lft(U(\sqrt{\pi}),\zeta\rgt)$ 
to the tubular neighborhood from the both sides, 
the confoliation hyperplane fields are glued 
along the boundary $\rd U(\sqrt{\pi})=\bigcup_{i=1}^n(\rd U)_i$ 
(see Figure~\ref{fig:cnstgltube}). 
 Let $\lft(\du,\zeta'\rgt)$ denote the double. 
 The underlying manifold is diffeomorphic to $S^1\times S^{2n}$ 
after smoothing. 
 Note that there exists an $S^1$-family of the bordered Legendrian open books 
in $\lft(\du,\zeta'\rgt)$ from Proposition~\ref{prop:exblob} and its proof. 
 In fact, $\lft(U(\sqrt{\pi}),\zeta\rgt)$ contains it 
in the boundary $\rd U(\sqrt{\pi})$. 

%
%
\begin{figure}[htb]
  \centering
  {\small \includegraphics[height=3.63cm]{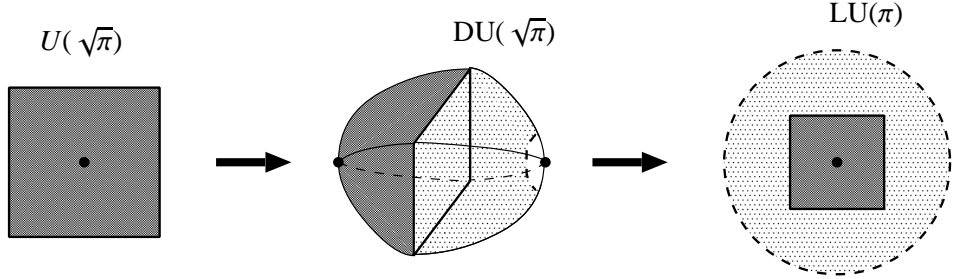}}
  \caption{Construction of the model Lutz tube $\pmod{\times S^1}$.}
  \label{fig:cnstgltube}
\end{figure} 

  Next, we perturb the confoliation $\zeta'$ on $\du\cong S^1\times S^{2n}$ 
to a contact structure 
so that it still has an $S^1$-family of bordered Legendrian open books. 
 In order to do that, 
we first carefully observe what the double $\lft(\du,\zeta'\rgt)$ is like. 
 Recall that $\lft(\du,\zeta'\rgt)$ is constructed 
by gluing two 
$\lft(U(\sqrt{\pi}),\zeta\rgt)\subset\lft(S^1\times\R^{2n},\zeta\rgt)$ 
in the generalized open Lutz tube. 
 The hyperplane field $\zeta=\ker\omtw$ is a conductive confoliation 
whose non-contact locus, in $U(\sqrt{\pi})$, is 
\begin{equation*}
  \Sigma(\zeta)
  =\bigcup_{\substack{i,j=1\\i\ne j}}^n
  \lft\{(\phi,r_1,\theta_1,\dots,r_n,\theta_n)\in S^1\times\R^{2n}\;\lft|\;
  r_i=\sqrt{\frac{\pi}{2}},\ r_j=\sqrt{\frac{\pi}{2}}\rgt.\rgt\} 
\end{equation*}
(see Equation~\eqref{eq:nonctloc}). 
 In dimension higher than $5$ (i.e.\ $n>2$), 
the non-contact locus $\Sigma(\zeta)$ intersects 
with the boundary $\rd U(\sqrt{\pi})$. 
 However, since the vector field 
$X=r_1(\uvv{r_1})+r_2(\uvv{r_2})+\dots+r_n(\uvv{r_n})$ 
is tangent to $\nul{\tau(\omtw)}^\perp$ 
(see the proof of Proposition~\ref{prop:conductive}), 
the confoliation $\zeta$ is conductive even on $U(\sqrt{\pi})$. 
 In the contact part $U(\sqrt{\pi})\setminus\Sigma(\zeta)$, 
there exists an $S^1$-family of bordered Legendrian open books 
\begin{equation*}
  P:=\lft\{(\phi,r_1,\theta_1,\dots,r_n,\theta_n)\in S^1\times\R^{2n}\;\lft|\;
  r_1\le\sqrt{\pi},\ r_2=\dots=r_n=\sqrt{\pi}\rgt.\rgt\} 
  \cong S^1\times(D^2\times T^{n-1})
\end{equation*}
whose binding is $T^{n-1}$ (see Equation~\eqref{eq:blobfam}). 
 This family $P$ lies in the boundary $\rd U(\sqrt{\pi})$. 
 We should remark that the family $P$ of bordered Legendrian open books 
never intersects with the non-contact locus $\Sigma(\zeta)$. 
 Then, by gluing two $\lft(U(\sqrt{\pi}),\zeta\rgt)$ as above, 
we obtain a conductive confoliation $\zeta'$ 
on a compact manifold $\du\cong S^1\times S^{2n}$. 
 There exists an $S^1$-family $P\cong S^1\times(D^2\times T^{n-1})$ 
of bordered Legendrian open books in $\lft(\du,\zeta'\rgt)$ 
which never intersects the non-contact locus $\Sigma(\zeta')$. 

  Now we perturb the conductive confoliation $\zeta'$ to a contact structure. 
 Since $\zeta'$ is a conductive confoliation on a compact orientable manifold 
$\du\cong S^1\times S^{2n}$, 
we can apply Theorem~\ref{thm:approx}. 
 Let $\omega$ be a confoliation $1$-form on $\du$ 
defining the confoliation $\zeta'$. 
 Then we have a contact form $\alpha$ on $\du$ 
which is $C^\infty$-close to $\omega$. 
 Although $\alpha$ is contact, it is not clear if it still has 
an $S^1$-family of bordered Legendrian open books. 
 In order to make it sure, we adjust $\alpha$ as follows. 
 Let $U\subset\du$ be a small open tubular neighborhood 
of $P\sqcup\Gamma_0\sqcup\Gamma_1$, 
that is, an $S^1$-family $P$ of bordered Legendrian open books 
and transverse curves $\Gamma_i\subset\lft(\du,\zeta'\rgt)$ 
corresponding to the core transverse curve 
of $\lft(U(\sqrt{\pi}),\zeta\rgt)\subset\lft(S^1\times\R^{2n},\zeta\rgt)$. 
 Let $V\subset\du$ be a small open tubular neighborhood 
of the non-contact locus $\Sigma(\zeta')\subset\lft(\du,\zeta'\rgt)$. 
 We may assume $\bar U\cap\bar V=\emptyset$. 
 There exists a function $f\colon \du\to[0,1]$ 
that satisfies the following conditions: 
(i) $f|_{\bar U}\equiv 1$, 
(ii) $\operatorname{supp}(f)\subset\du\setminus\bar V$. 
 With this function $f$, set
\begin{equation*}
  \tilde\alpha:=(1-f)\alpha+f\omega=\alpha+f(\omega-\alpha). 
\end{equation*}
 Then it is a contact form on $\du$ 
since the contact form $\alpha$ is $C^\infty$-close 
to the confoliation $1$-form $\omega$. 
In fact, 
\begin{equation*}
  d\tilde\alpha=d\alpha+df\wedge(\omega-\alpha)+f(d\omega-d\alpha)
  =df\wedge(\omega-\alpha)+\{d\alpha+f(d\omega-d\alpha)\}, 
\end{equation*}
then 
\begin{align*}
  \tilde\alpha\wedge(d\tilde\alpha)^n=&\alpha\wedge(d\alpha)^n
  +(d\omega-d\alpha)\wedge
  \sum_{i=1}^n\lft\{{}_{n}C_{i}f^i(d\omega-d\alpha)^{i-1}\wedge(d\alpha)^{n-i}\rgt\}
  \\
  &+(\omega-\alpha)\wedge\lft\{d\alpha+f(d\omega-d\alpha)\rgt\}^{n-1}\wedge
  \lft\{n\alpha\wedge df+f(d\alpha+f(d\omega-d\alpha))\rgt\}. 
\end{align*}
 Taking the contact form $\alpha$ 
sufficiently $C^\infty$-close to the confoliation $1$-form $\omega$, we have 
$\tilde\alpha\wedge(d\tilde\alpha)^n>0$. 
 Let $\tilde\zeta$ denote the contact structure $\ker\tilde\alpha$ on $\du$. 
 Further, we have $\tilde\alpha|_{\bar U}=\omega$, 
and $\tilde\alpha|_{\bar V}=\alpha$. 
 In other words, $P$ is an $S^1$-family of bordered Legendrian open books, 
and $\Gamma_i$ are transverse circles for the contact structure $\tilde\zeta$. 
 Then this $\tilde\zeta=\ker\tilde\alpha$ is the required contact structure 
on $\du$. 

  In order to change the double $\du\cong S^1\times S^{2n}$ 
to a tubular neighborhood of a circle, namely $S^1\times D^{2n}$, 
we remove a tubular neighborhood of a certain circle. 
 The embedded circles $\Gamma_i\subset\lft(\du,\tilde\zeta\rgt)$, $i=1,2$, 
corresponding to 
$\Gamma_0=S^1\times\{0\}\subset\lft(U(\sqrt{\pi}),\zeta\rgt)$ 
are transverse to the contact structure $\tilde\zeta$. 
 We remove the interior of the standard tubular neighborhood 
$U_1\subset\du$ of $\Gamma_1$, 
which is contactomorphic to 
$\lft(U_\delta, \ker\{d\phi+\sum_{i=1}^{n}r_i^2d\theta_i\}\rgt)$ 
for some radius $\delta>0$. 
 The obtained tubular neighborhood 
$\lft(\mlt,\tilde\zeta\rgt):=\lft(\du\setminus \inter U_1,\tilde\zeta\rgt)$ 
of the transverse circle $\Gamma_2$ 
is the required model to be replaced with the standard tubular neighborhood. 
 We call it the \emph{model $\pi$-Lutz tube}. 
 Note that $\lft(\mlt,\tilde\zeta\rgt)$ has an $S^1$-family 
of bordered Legendrian open books 
because the family $P\subset\lft(\du,\tilde\zeta\rgt)$ lies far 
from the core transverse curves $\Gamma_i\subset\lft(\du,\tilde\zeta\rgt)$ 
from the construction. 

\subsubsection{The generalized half Lutz twist}\label{sec:ghltw}
  Now, we define a generalization of the Lutz twist. 
 It is defined as a modification of a contact structure 
on a $(2n+1)$-dimensional manifold along an embedded transverse circle. 
 As a result, the operation makes an $S^1$-family 
of bordered Legendrian open books with a torus $T^{n-1}$ as binding. 
 In this subsubsection, 
we define a generalization of the so-called half Lutz twist. 

 This operation can also be defined by using contact round surgeries 
(see Subsection~\ref{sec:ltwtrvknt}, and Figure~\ref{fig:l1tw}). 

  The situation is as follows.  
 Let $(M,\xi)$ be a $(2n+1)$-dimensional contact manifold, 
and $\Gamma\subset(M,\xi)$ an embedded transverse circle. 
 From Corollary~\ref{cor:s1dar}, there exists the standard tubular neighborhood
$U\subset(M,\xi)$ of $\Gamma$, 
which is contactomorphic to 
$\lft(U_\epsilon,\ker\{d\phi+\sum_{i=1}^{n}r_i^2d\theta_i\}\rgt)$, 
where 
$U_\epsilon=\lft\{\sum_{i=1}^{n}r_i^2\le\epsilon^2\rgt\}\subset S^1\times\R^2$,
$(\phi,r_i,\theta_i)$ are the cylindrical coordinates of $S^1\times\R^{2n}$. 
 The generalized Lutz twist is to be defined 
by replacing the standard tubular neighborhood of $\Gamma$ 
with a certain tubular neighborhood.

  The model $\pi$-Lutz tube $\lft(\mlt,\tilde\zeta\rgt)$ obtained above 
is replaced with the standard tubular neighborhood as follows. 
 The theory of the $\xi$-round hypersurfaces 
(see Subsubsection~\ref{sec:xi-rdhyps}) is used. 
 Recall that we can take the standard tubular neighborhood 
$\lft(U_\epsilon,\xi_0=\ker\{d\phi+\sum_{i=1}^{n}r_i^2d\theta_i\}\rgt)$ 
for the transverse circle $\Gamma$ 
so that the boundary $\rd U_\epsilon$ is a $\xi_0$-round hypersurface 
modeled on the standard contact sphere $(S^{2n-1},\eta_0)$ 
(see Example~\ref{expl:xirdmostctsph}). 
 In the same way, we may assume that the boundary $\rd(\mlt)$ 
of the model $\pi$-Lutz tube $\lft(\mlt,\tilde\zeta\rgt)$ 
is a $\tilde\zeta$-round hyper surface. 
 From the construction, it is regarded as a $\xi_0$-round hypersurface 
modeled on the standard contact sphere $(S^{2n-1},\eta_0)$. 
 By Proposition~\ref{prop:xi-rdhyps}, $(M\setminus \inter U_\epsilon,\xi)$ 
and $\lft(\mlt,\tilde\zeta\rgt)$ are glued 
so that their longitudes and meridian spheres agree. 
 As a result, we obtain a manifold which is diffeomorphic to the given $M$ 
with a contact structure which is modified from the given $\xi$ 
along the given transverse circle $\Gamma$. 

  We call the modification that we have defined above 
the \emph{generalized Lutz twist}\/ along a transverse circle $\Gamma$ 
in a contact manifold $(M,\xi)$. 
 In dimension $3$, it is the so called $\pi$-Lutz twist. 
 Then, if necessary, we call the generalization 
the \emph{generalized half (or $\pi$-) Lutz twist}. 

%
%
\begin{rem}
  In dimension $3$, the operation above 
is the original $3$-dimensional half (or $\pi$-) Lutz twist. 
 Some rough reason is in Subsubsection~\ref{sec:revrdsurgrep}. 
 More precisely, it is explained as follows. 
 In dimension $3$, the generalized Lutz tube 
is the ordinary $3$-dimensional Lutz tube $(S^1\times\R^2,\zeta)$, 
where $\zeta=\ker\{\cos r^2\;d\phi+\sin r^2\;d\theta\}$ 
(see Remark~\ref{rem:3dimomtw}). 
 And $U(\sqrt{\pi})\subset(S^1\times\R^2,\zeta)$ consists of 
a family of overtwisted disks. 
 The important thing is that $(U(\sqrt{\pi})\setminus U_\epsilon,\zeta)$ 
is contact-embedded into 
$\lft(\{(\phi,r,\theta)\in S^1\times\R^2\mid \pi<r^2<2\pi\},\zeta\rgt)$ 
so that $\rd U(\sqrt{\pi})$ agrees 
with $\{r^2=\pi\}\subset(S^1\times\R^2,\zeta)$. 
 This implies the model $\pi$-Lutz tube $(\mlt,\zeta')$ 
is embedded into $\{r^2<2\pi\}\subset(S^1\times\R^2,\zeta)$. 
 Therefore, it is the original $3$-dimensional $\pi$-Lutz twist. 
\end{rem} 

\subsubsection{The generalized full Lutz twist}\label{sec:gfltw}
  We can construct a generalization of the full Lutz twist, 
using $\lft(U(\sqrt{3\pi/2}),\zeta\rgt)
\subset\lft(S^1\times\R^{2n},\zeta\rgt)$ 
instead of $\lft(U(\sqrt{\pi}),\zeta\rgt)$ in the previous procedure. 

  It is constructed as a modification of a contact structure $\xi$ 
on a $(2n+1)$-dimensional manifold $M$ 
along an embedded transverse circle $\Gamma\subset(M,\xi)$. 
 Removing the standard tubular neighborhood of $\Gamma$ and 
gluing back a tubular neighborhood with a certain contact structure, 
we obtain a new contact structure on $M$. 
 A new contact structure, or the contact structure on a tubular neighborhood 
to be replaced, is supposed to have a bordered Legendrian open book. 
 Especially, it is also supposed to be homotopic to the original one 
as almost contact structures 
(see Subsubsection~\ref{sec:homotalmct}). 

  It is left to construct a contact structure 
on a tubular neighborhood of a circle to be replaced. 
 Take the double 
of $\lft(U(\sqrt{3\pi/2}),\zeta\rgt)\subset\lft(S^1\times\R^{2n},\zeta\rgt)$ 
and make it a contact manifold $\lft(\du[\sqrt{3\pi/2}],\tilde\zeta_2\rgt)$ 
by Theorem~\ref{thm:approx} and the discussion in the previous subsubsection. 
 Then remove the standard tubular neighborhood 
of the transverse core of one of two $\lft(U(\sqrt{3\pi/2}),\zeta\rgt)$. 
 The obtained contact manifold $\lft(\mlt[2\pi],\tilde\zeta_2\rgt)$, 
which is diffeomorphic to $S^1\times D^{2n}$, 
is the required one. 
 We call it the \emph{model $2\pi$-Lutz tube}. 
 We can apply the same argument as in Subsubsection~\ref{sec:ghltw} 
to this $\lft(\mlt[2\pi],\tilde\zeta_2\rgt)$ 
to obtain the modified contact structure. 
 We call this procedure 
the \emph{generalized full \textup{(}or $2\pi$-\textup{)} Lutz twist}. 
 Note that $\lft(\mlt[2\pi],\tilde\zeta_2\rgt)$ has an $S^1$-family 
of the bordered Legendrian open books because 
$\lft(U(\sqrt{3\pi/2}),\zeta\rgt)
\supset\lft(U(\sqrt{\pi}),\zeta\rgt)$ has it. 

  We can check easily that the contact hyperplane $\tilde\zeta_2$, 
or the contact form as a covector, rotates more than $2\pi$ 
along the each radius $r_i$, $i=1,2,\dots,n$,  
moving from the center to the boundary of $\mlt[2\pi]$. 
 The important property to call it a ``full'' twist is studied 
in Subsubsection~\ref{sec:homotalmct}. 

\subsubsection{The $k$-fold generalized Lutz twist}\label{sec:k-fold}
  In a similar way, we can define the $k$-fold generalized Lutz twist 
for any positive integer $k\in\mathbb{N}$. 
 In this case, we use $\lft(U(\sqrt{(k+1)\pi/2}),\zeta\rgt)
\subset\lft(S^1\times\R^{2n},\zeta\rgt)$ 
instead of $\lft(U(\sqrt{\pi}),\zeta\rgt)$. 
 Then we obtain the \emph{model $k\pi$-Lutz tube}\/ 
$\lft(\mlt[k\pi],\tilde\zeta_k\rgt)$ 
from the double 
\begin{equation*}
  \lft(\du[\sqrt{(k+1)\pi/2}],\zeta'_k\rgt)
  =\lft(U(\sqrt{(k+1)\pi/2}),\zeta\rgt)
  \cup\lft(U(\sqrt{(k+1)\pi/2}),\zeta\rgt).
\end{equation*}
 The \emph{$k$-fold generalized Lutz twist}\/ is defined 
as the replacement of the standard tubular neighborhood of a transverse circle 
with the model $k\pi$-Lutz tube $\lft(\mlt[k\pi],\tilde\zeta_k\rgt)$. 
 Since $\lft(\mlt[k\pi],\tilde\zeta_k\rgt)$ includes an $S^1$-family 
of bordered Legendrian open books, 
the resulting contact structure is PS-overtwisted. 
 In terms of $k$-fold twists, the generalized half and full Lutz twists 
correspond to $k=1,2$ respectively. 

\subsection{Important properties of Half and full Lutz twists}
\label{sec:handftwists}
  In this section, we discuss important properties 
of the half and full generalized Lutz twists. 
 First, we discuss the contribution of the generalized ``half'' Lutz twist 
to the Euler class of a contact structure 
in Subsubsection~\ref{sec:eulerclass}. 
 Then, in Subsubsection~\ref{sec:homotalmct}, 
we show that the generalized ``full'' Lutz twist does not change 
the homotopy class as almost contact structures of a contact structure. 
 At the end of this subsection, in Subsubsection~\ref{sec:findtrvs1}, 
we show that the generalized Lutz twist can be operated anywhere. 

\subsubsection{Contribution of the generalized half Lutz twist 
  to the Euler class}\label{sec:eulerclass}
  We discuss the important property of the generalized half Lutz twists. 
 They may change the Euler class of a contact structure. 
 It is well known that, in dimension~$3$, the claim is true 
(see \cite{geitext}). 
 The generalization of the half Lutz twist 
obtained in Subsection~\ref{sec:defGenLztw} also has the same property. 
 It contributes to the Euler class of a contact structure. 
 This is an answer to a question in \cite{eptw} 
asking the existence of such a modification. 
 The claim is as follows. 

%
%
\begin{prop}\label{prop:cntrbtoE}
  Let $(M,\xi)$ be a contact manifold of dimension $(2n+1)$, 
and $\Gamma\subset(M,\xi)$ an embedded positive transverse circle. 
 The generalized half Lutz twist along $\Gamma$ contributes 
to the Euler class of the contact structure $\xi$ as follows\textup{:} 
\begin{equation*}
  e(\xi^\Gamma)-e(\xi)=-2\operatorname{PD}([\Gamma])\in H^{2n}(M;\Z), 
\end{equation*}
where $\xi^\Gamma$ is the contact structure on $M$ 
obtained from $\xi$ by the generalized half Lutz twist along $\Gamma$, 
and $\operatorname{PD}([\Gamma])$ is the Poincar\'e dual 
to the homology class $[\Gamma]\in H_1(M;\Z)$. 
\end{prop} 

  We compare the Euler classes by using the following fact 
(see \cite{bredon} for example): 
%
%
\begin{lemma}\label{prop:echar}
  Let $\pi\colon E\to M$ be an oriented vector bundle of rank $r$ 
over an oriented manifold $M$ of dimension $n$. 
 Let $\sigma\colon M\to E$ be a section of $\pi$ transverse to the zero section,
and $Z\subset M$ the zero locus of $\sigma$. 
 Then the Euler class $e(E)\in H^r(M;\Z)$ of the bundle $\pi\colon E\to M$ 
is the Poincar\'e dual to $[Z]\in H_{n-r}(M;\Z)$. 
\end{lemma}

\begin{proof}[Proof of Proposition~\ref{prop:cntrbtoE}]
  Recall that the generalized Lutz twist is operated 
on a tubular neighborhood of a transverse circle. 
 Then we discuss by using local models. 

  The standard model of a tubular neighborhood of a transverse circle 
in a contact manifold of dimension $2n+1$ 
is given as $S^1\times D^{2n}\subset S^1\times\R^{2n}$ 
with the standard contact structure: 
\begin{equation*}
  \xi_0=\ker\lft\{d\phi+r_1^2d\theta_1+\dots+r_n^2d\theta_n\rgt\}, 
\end{equation*}
where $(\phi,r_1,\theta_1,\dots,r_n,\theta_n)$ are the cylindrical coordinates 
(see Corollary~\ref{cor:s1dar}). 
 The radial vector field $\sigma:=r_1(\uvv{r_1})+\dots+r_n(\uvv{r_n})$
is a generic section of $\xi_0$, 
which vanishes to the first order along the core positive transverse circle 
$\Gamma=S^1\times\{0\}$. 

  The local model for the generalized half Lutz twist 
is the model $\pi$-Lutz tube $\lft(\mlt,\tilde\zeta\rgt)$ 
defined in Subsection~\ref{sec:defGenLztw}. 
 Note that the core positive transverse circle with orientation is $-\Gamma$. 
 Recall that the  model $\pi$-Lutz tube is constructed 
using the generalized open Lutz tube $\lft(S^1\times\R^{2n},\zeta\rgt)$. 
 Furthermore, the contact structure $\tilde\zeta$ is obtained 
from the confoliation $\zeta=\ker\omtw$ on $S^1\times\R^{2n}$ 
with the defining $1$-form 
\begin{align*}
  \omtw&=\lft((\cos r_1^2)(\cos r_2^2)\cdots(\cos r_n^2)\rgt)d\phi
  +(\sin r_1^2)d\theta_1+(\sin r_2^2)d\theta_2+\dots+(\sin r_n^2)d\theta_n \\
  &=\lft(\prod_{i=1}^n\cos r_i^2\rgt)d\phi+\sum_{i=1}^n(\sin r_i^2)d\theta_i
\end{align*}
by some slight perturbation as hyperplane fields 
(see Subsection~\ref{sec:defGenLztw}). 
 Then, in order to discuss the Euler class 
of the contact structure $\tilde\zeta$ on the model $\pi$-Lutz tube, 
we use the $1$-form $\omtw$. 
 In other words, we deal with 
$U(\sqrt{\pi})=\{r_i\le\sqrt{\pi},i=1,2,\dots,n\}
\subset\lft(S^1\times\R^{2n},\zeta=\ker\omtw\rgt)$ 
and $U(\sqrt{\pi})\setminus\inter U(S^1\times\{0\})
\subset\lft(S^1\times\R^{2n},\zeta=\ker\omtw\rgt)$, 
where $U(S^1\times\{0\})$ is the standard tubular neighborhood 
of the transverse circle $S^1\times\{0\}\subset\lft(S^1\times\R^{2n},\zeta\rgt)$.

  A generic section for the model $\pi$-Lutz tube is constructed as follows. 
 As we observed above, we construct generic sections 
$\sigma_1$ of $\zeta$ on $U(\sqrt{\pi})$ 
and $\sigma_2$ of $\zeta$ on $U(\sqrt{\pi})\setminus\inter U(S^1\times\{0\})$ 
so that they can be glued together along $\rd U(\sqrt{\pi})$. 
 Let $g\colon[0,+\infty)\to\R$ be a non-decreasing function 
which is constantly $0$ on $[0,\epsilon]$ 
and $1$ on $[\sqrt{\pi}-\epsilon,+\infty)$ 
for some sufficiently small $\epsilon>0$. 
 Then the vector field 
\begin{equation*}
  \sigma_1:=g(r)\sum_{i=1}^{n}r_i\uv{r_i}
  +\lft(1-g(r)\rgt)r\lft\{\lft(\sum_{i=1}^n\sin r_i^2\rgt)\uv\phi
  -\lft(\prod_{i=1}^n\cos r_i^2\rgt)\sum_{i=1}^n\uv{\theta_i}\rgt\}, 
\end{equation*}
where $r=\sqrt{r_1^2+\dots+r_n^2}$, 
is a section of $\zeta=\ker\omtw$. 
 In fact, 
\begin{equation*}
  \omtw(\sigma_1)=g(r)r\lft\{
  \lft(\prod_{i=1}^n\cos r_i^2\rgt)\lft(\sum_{i=1}^n\sin r_i^2\rgt)
  -\lft(\prod_{i=1}^n\cos r_i^2\rgt)\lft(\sum_{i=1}^n\sin r_i^2\rgt)\rgt\}=0. 
\end{equation*}
 Similarly, the vector field $\sigma_2:=-\sum_{i=1}^n r_i(\uvv{r_i})$ 
is a section of $\zeta$ which is nonzero 
on $U(\sqrt{\pi})\setminus\inter U(S^1\times\{0\})$. 
 From $\sigma_1$ and $\sigma_2$, a section $\sigma'$ of the contact structure 
for the model $\pi$-Lutz tube $\lft(\mlt,\tilde\zeta\rgt)$ is constructed 
by gluing them together. 
 In fact, at $\rd U(\sqrt{\pi})$, 
$\sigma_1=-\sigma_2=\sum_{i=1}^{n}r_i(\uvv{r_i})$. 

  We need the zero locus of this section $\sigma'$. 
 Since $\sigma_2$ is non-zero, we have only to observe $\sigma_1$. 
 The vector field $\sigma_1$ vanishes to the first order 
along the core positive transverse circle $-\Gamma$. 

  From the observations above on the zero loci of sections $\sigma$ of $\xi_0$ 
and $\sigma'$ of $\zeta$, 
and Lemma~\ref{prop:echar}, we have 
\begin{equation*}
  e(\xi)-\operatorname{PD}([\Gamma])=e(\xi^\Gamma)-\operatorname{PD}(-[\Gamma]). 
\end{equation*}
 Thus we obtain the conclusion. 
\end{proof}

\subsubsection{Homotopy class as almost contact structures}
\label{sec:homotalmct}
  We discuss the important property of the generalized full Lutz twist. 
 We show the following proposition. 
%
%
\begin{prop}
  The generalized full Lutz twist does not change 
 the homotopy class of a contact structure as almost contact structures. 
\end{prop} 
\begin{proof}
  For this proof, we construct a one-parameter family of hyperplane field 
between the given contact structure 
and the contact structure obtained by the generalized full Lutz twist. 
 Furthermore, we construct a corresponding family of non-degenerate $2$-forms 
on the hyperplane fields. 
 The generalized Lutz twist is a modification 
along an embedded circle which is transverse to the given contact structure. 
 Therefore, in order to construct the above-mentioned things, 
we may discuss by using the local model along a transverse circle. 

  The local model for the generalized full Lutz twist 
is the model $2\pi$-Lutz tube $\lft(\mlt[2\pi],\tilde\zeta_2\rgt)$ 
defined in Subsubsection~\ref{sec:gfltw}. 
 Like the proof of Proposition~\ref{prop:cntrbtoE}, 
we use the generalized open Lutz tube $\lft(S^1\times\R^{2n},\zeta\rgt)$ 
and the confoliation $\zeta=\ker\omtw$, where 
$\omtw=\lft(\prod_{i=1}^n\cos r_i^2\rgt)d\phi
+\sum_{i=1}^n\lft(\sin r_i^2\rgt)d\theta_i$. 
 In other words, we deal with 
$U(\sqrt{3\pi/2})=\lft\{r_i\le\sqrt{3\pi/2},i=1,2,\dots,n\rgt\}
\subset\lft(S^1\times\R^{2n},\zeta=\ker\omtw\rgt)$ 
and $U(\sqrt{3\pi/2})\setminus U(S^1\times\{0\})
\subset\lft(S^1\times\R^{2n},\zeta\rgt)$. 

  Although the local model is explicitly given by the $1$-form $\omtw$,  
we slightly generalize the description as in the definition 
of the $3$-dimensional Lutz twist in Subsection~\ref{sec:lutztwist}. 
 Setting
\begin{equation*}
  \hat\omega
  :=\lft(\prod_{i=1}^{n}f_i(r_i)\rgt)d\phi+\sum_{i=1}^{n}g_i(r_i)d\theta_i
  =\{f_1(r_1)\cdots f_n(r_n)\}d\phi+g_1(r_1)d\theta_1+\dots+g_n(r_n)d\theta_n
\end{equation*}
we have a $1$-form on $S^1\times\R^{2n}$. 
 If $f_i(r_i)=\cos r_i^2$, $g_i(r_i)=\sin r_i^2$, then $\hat\omega=\omtw$. 
 In the $3$-dimensional case ($n=1$), $\hat\omega$ is contact 
if a point $\lft(g_1(r_1),f_1(r_1)\rgt)$ in the $(g_1,f_1)$-plane 
rotates around the origin with respect to $r_1$ 
(see Subsection~\ref{sec:lutztwist}). 
 In the higher-dimensional case ($n>1$), situation is more complicated. 
 From homotopies of curves on the $(g_i,f_i)$-planes, $i=1,2,\dots,n$, 
we obtain a homotopy of hyperplane field 
if the curves do not pass through the origins. 
 Then the contact structure 
on $U(\sqrt{3\pi/2})\subset\lft(S^1\times\R^{2n},\zeta\rgt)$ 
is represented as the curve in Figure~\ref{fig:hmtp2pi}~(I). 
 All functions $f_i(r_i)$, $g_i(r_i)$, $i=1,2,\dots,n$, are given by the curve.
 In addition, the model $2\pi$-Lutz tube $\lft(\mlt[2\pi],\tilde\zeta_2\rgt)$ 
is represented as the thin curve consists of the real and dotted curves 
in Figure~\ref{fig:hmtp2pi}~(II) 
(Compare this with Figure~\ref{fig:LzTw}~(I)). 
 The thick dotted curve in Figure~\ref{fig:hmtp2pi}~(II) represents 
the contact structure 
on the standard tubular neighborhood of a transverse curve. 
%
%
\begin{figure}[htb]
  \centering
  {\small \includegraphics[height=4.53cm]{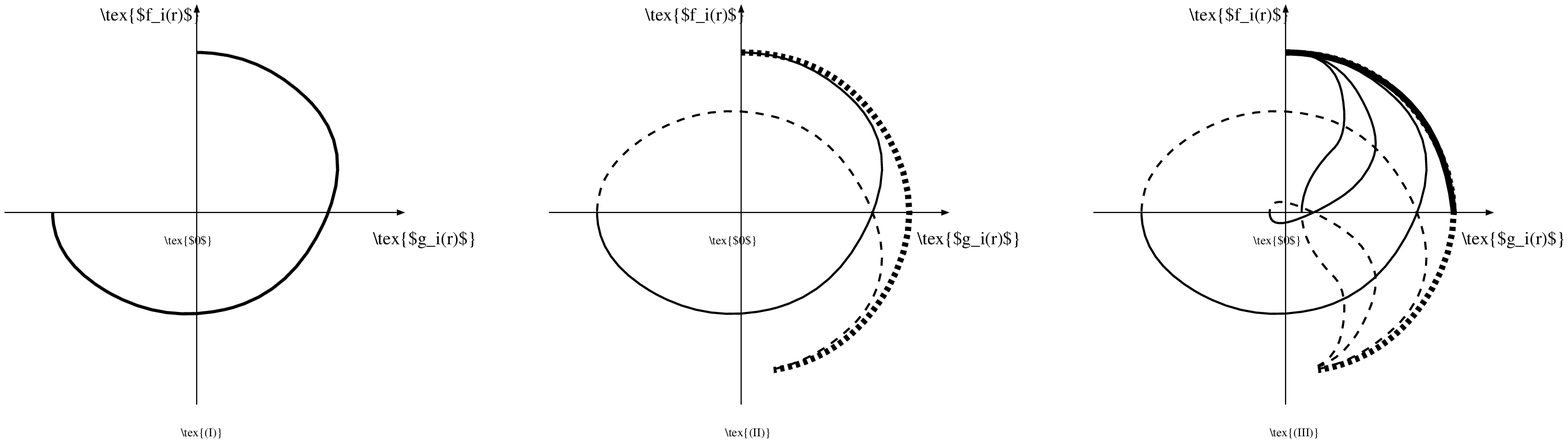}}
  \caption{Homotopy to $2\pi$-Lutz tube}
  \label{fig:hmtp2pi}
\end{figure} 
 Then, in order to construct a homotopy between the contact structure 
on the model $2\pi$-Lutz tube $\lft(\mlt[2\pi],\tilde\zeta_2\rgt)$ 
and the standard one, 
it is sufficient to observe the contact structure on $U(\sqrt{3\pi/2})$, 
that is, a half of the double $\du[\sqrt{3\pi/2}]\supset\mlt[2\pi]$. 

  A required path of hyperplane fields, 
or a path of nowhere-vanishing $1$-forms, is constructed as follows. 
 Let $\lft(g^t_i(r_i),f^t_i(r_i)\rgt)$ be a family, 
with respect to the parameter $t\in[0,1]$, 
of pairs of functions that represents a family of curves 
in Figure~\ref{fig:hmtp2pi}~(III) for each $i=1,2,\dots,n$. 
 Those are taken so that they satisfy the following: 
\begin{itemize}
\item $\displaystyle \lft(g^0_i(r_i),f^0_i(r_i)\rgt)
  =\lft(\sin\lft(\frac{\pi}{2}r_i^2\rgt),
  \cos\lft(\frac{\pi}{2}r_i^2\rgt)\rgt)$, 
\item $\displaystyle \lft(g^1_i(r_i),f^1_i(r_i)\rgt)
  =\lft(\sin\lft(\frac{3\pi}{2}r_i^2\rgt),
  \cos\lft(\frac{3\pi}{2}r_i^2\rgt)\rgt)$. 
\end{itemize}
 For such a family of pairs of functions, set 
\begin{equation*}
  \hat\omega_t
  :=\lft(\prod_{i=1}^{n}f_i^t(r_i)\rgt)d\phi+\sum_{i=1}^{n}g_i^t(r_i)\,d\theta_i
  =\lft\{f_1^t(r_1)\cdots f_n^t(r_n)\rgt\}d\phi
  +g_1^t(r_1)\,d\theta_1+\dots+g_n^t(r_n)\,d\theta_n. 
\end{equation*}
 Note that $\zeta_0:=\ker\hat\omega_0$ corresponds to the standard structure, 
and that $\zeta_1:=\ker\hat\omega_1$ corresponds to the $2\pi$-Lutz tube. 
 Then the family 
\begin{equation*}
  \omega_t:=\hat\omega_t+t(1-t)\sum_{i=1}^{n}r_i(1-r_i)dr_i
  =\lft(\prod_{i=1}^{n}f_i^t(r_i)\rgt)d\phi+\sum_{i=1}^{n}g_i^t(r_i)d\theta_i
  +t(1-t)\sum_{i=1}^{n}r_i(1-r_i)dr_i
\end{equation*}
of nowhere-vanishing $1$-forms, $t\in[0,1]$,  
determines a path of hyperplane fields from $\zeta_0$ to $\zeta_1$. 
 Let $\zeta_t$ denote the path $\ker\omega_t$, $t\in[0,1]$. 

  We verify that the family $\zeta_t$ is that of almost contact structures. 
 We give a family of $2$-forms each of which is nondegenerate 
on the hyperplane field $\zeta_t$ for each $t\in[0,1]$. 
 The following family of $2$-forms is one of the required one 
for some $A_i>0$, $i=1,2,\dots,n$, sufficiently large: 
\begin{align*}
  \gamma_t&:=d\hat\omega_t
  +\sum_{i=1}^n\lft(\prod_{j\ne i}\frac{dg_j^t}{dr_j}(r_j)\rgt)A_i\,
  d\theta_i\wedge d\phi \\
  &=\sum_{i=1}^n\lft(\prod_{j\ne i}f_j^t(r_j)\rgt)\frac{df_i^t}{dr_i}(r_i)\,
  dr_i\wedge d\phi
  +\sum_{i=1}^n\frac{dg_i^t}{dr_i}(r_i)\,dr_i\wedge d\theta_i
  +\sum_{i=1}^n\lft(\prod_{j\ne i}\frac{dg_j^t}{dr_j}(r_j)\rgt)A_i\,d\theta_i
  \wedge d\phi. 
\end{align*}
 In fact, we have
\begin{align*}
  \omega_t\wedge\gamma_t^n
  =&\left\{\lft(\prod_{i=1}^{n}f_i^t(r_i)\frac{dg_i^t}{dr_i}(r_i)\rgt)
    -\sum_{i=1}^n\lft(\prod_{j\ne i}f_j^t(r_j)\rgt)\frac{df_i^t}{dr_i}(r_i)
    \lft(\prod_{j\ne i}\frac{dg_j^t}{dr_j}(r_j)\rgt)g_i^t(r_i)\rgt. \\
    &\lft.+t(1-t)\sum_{i=1}^{n}A_i\lft(\prod_{j\ne i}\frac{dg_j^t}{dr_j}(r_j)\rgt)^2
    r_i(1-r_i)\rgt\}
    d\phi\wedge dr_1\wedge d\theta_1\wedge\dots\wedge dr_n\wedge d\theta_n \\
    =&\hat\omega_t\wedge(d\hat\omega_t)^n \\
    &+\lft\{t(1-t)\sum_{i=1}^{n}A_i
      \lft(\prod_{j\ne i}\frac{dg_j^t}{dr_j}(r_j)\rgt)^2
    r_i(1-r_i)\rgt\}
    d\phi\wedge dr_1\wedge d\theta_1\wedge\dots\wedge dr_n\wedge d\theta_n. 
\end{align*}
 The $(2n+1)$-form $\hat\omega_t\wedge(d\hat\omega_t)^n$ is non-negative 
if $t=0,1$. 
 Moreover, we may assume that $\lft(dg_i^t/dr_i\rgt)(r_i)$, $i=1,2,\dots,n$, 
vanish at mutually different points. 
 Then, for sufficiently large $A_i>0$, 
the $(2n+1)$-form $\omega_t\wedge \gamma_t^n$ is positive. 
 Then $\zeta_t=\ker\omega_t$ is a family of almost contact structures. 
\end{proof} 

\subsubsection{Finding a transverse circle}\label{sec:findtrvs1}
  We show that any embedded circle can be approximated by a transverse circle. 
 This implies that the generalized Lutz twist can be operated any where. 
  As a result, we obtain Corollary~\ref{cor:anywhere}. 

 First of all, recall that, by the well-known Chow lemma, 
any circle embedded in a contact manifold 
can be approximated to an isotropic circle. 

 We show that, for any isotropic circle, 
there exists a transverse circle close to it. 
 Let $L$ be an isotropic circle 
in a $(2n+1)$-dimensional contact manifold $(M,\xi)$. 
 From Proposition~\ref{prop:isotnbd}, 
there exists a tubular neighborhood $U\subset(M,\xi)$ of $L$ 
which is contactomorphic to some tubular neighborhood 
of the isotropic circle $S^1\times\{0\}$ in $S^1\times\R^{2n}$ 
with the contact structure 
\begin{equation*}
  \eta=\ker\lft\{(\cos\phi)dx_1-(\sin\phi)dy_1
  +\frac{1}{2}\sum_{i=2}^n(x_idy_i-y_idx_i)\rgt\},
\end{equation*}
where $(\phi,x_1,y_1,\dots,x_n,y_n)$ are coordinates of $S^1\times\R^{2n}$. 
 By the projection $\pi\colon S^1\times\R^{2n}\to S^1\times\R^2$, 
$(\phi,x_1,y_1,\dots,x_n,y_n)\mapsto(\phi,x_1,y_1)$, 
we have a Legendrian curve $\pi(L)\subset(S^1\times\R^2,\pi_\ast\eta)$ 
in the $3$-dimensional standard contact tubular neighborhood. 
 Then we can take a transverse push-off $T_+(L)$ of $L$. 
 In other words, we take one of the two parallel dividing curves 
on the boundary of the standard tubular neighborhood 
of $\pi(L)\subset(S^1\times\R^2,\pi_\ast\eta)$. 
 As a circle in $S^1\times\R^{2n}\supset S^1\times\R^2$, 
the obtained curve $T_+(L)$ 
is transverse, and is homotopic to the original curve $L$. 
 We call the obtained curve $T_+(L)\subset(M,\xi)$ 
a \emph{transverse push-off}\/ of $L$. 

\subsection{Non-compact case}\label{sec:non-cpt}
  In this subsection, we deal with generalized Lutz twists 
along non-compact transverse curves. 
 Corollary~\ref{thm:exoticR2n1} is observed here. 

  Along an embedded line, $\R$, in a contact $(2n+1)$-dimensional manifold, 
a generalization of the Lutz twist is defined as follows. 
 The generalized Lutz twist is defined in Subsection~\ref{sec:defGenLztw} 
as a modification along an embedded circle transverse to a contact structure. 
 Instead of the model Lutz tube 
$\lft(\mlt,\tilde\zeta\rgt)$ diffeomorphic to $S^1\times D^{2n}$, 
we need $\R\times D^{2n}$ with a certain contact structure. 
 Recall that the contact structure $\tilde\zeta$ is obtained 
from the confoliation $\zeta=\ker\omtw$ with 
$\omtw=\prod_{i=1}^n(\cos r_i^2)d\phi+\sum_{i=1}^n(\sin r_i^2)d\theta_i$. 
 Then the required contact structure $\bar\zeta$ on $\R\times D^{2n}$ 
is obtained from the confoliation $\ker\bar\omega_{\mathtt{tw}}$ 
with the $1$-form 
\begin{equation*} 
  \bar\omega_{\mathtt{tw}}
  :=\prod_{i=1}^n(\cos r_i^2)dz+\sum_{i=1}^n(\sin r_i^2)d\theta_i
  =\lft\{(\cos r_1^2)\cdots(\cos r_n^2)\rgt\}dz
  +(\sin r_1^2)d\theta_1+\dots+(\sin r_n^2)d\theta_n, 
\end{equation*}
where $(z,r_1,\theta_1,\dots,r_n,\theta_n)$ are the cylindrical coordinates 
of $\R^{2n+1}$. 
 In order to define the generalized Lutz twist along a transverse line, 
we can apply the same arguments as in Subsection~\ref{sec:defGenLztw} 
using the generalized Lutz tube $\lft(\R^{2n+1},\bar\zeta\rgt)$. 


\section{Generalized Lutz twist and Giroux domain}\label{sec:gtor}
  We discuss generalizations of the Lutz twist along certain submanifolds 
in this section. 
 In dimension $3$, the Lutz twist is defined along a pre-Lagrangian torus 
as well as along a transverse circle (see Subsection~\ref{sec:lutztwist}). 
 Recall that the important property of 
the $\pi$-Lutz twist along a pre-Lagrangian torus is 
to create the Giroux $\pi$-torsion domain. 
 As a natural generalization, we define a generalization of the Lutz twist 
along a $\xi$-round hypersurface $S^{2n-1}\times S^1$ in dimension $2n+1$ 
(see Subsection~\ref{sec:circsph}). 
 However, this modification creates not only a Giroux domain 
but also a bordered Legendrian open book. 
 This implies that this operation does not make a gap 
between weak and strong symplectic fillabilities. 
 Then we introduce another modification of a contact structure 
that creates a Giroux domain directly (see Subsection~\ref{sec:Gdom}). 
 It is operated along a pre-Lagrangian torus even in higher dimensions. 
 These claims amount to the proofs 
of Theorem~\ref{thm:gdom-xird} and Theorem~\ref{thm:gdom-pLag}. 

  Both the $\xi$-round hypersurface and the pre-Lagrangian torus 
are generalizations of a $2$-dimensional pre-Lagrangian torus 
in a contact $3$-manifold. 
 These generalizations depend on how we regard the $2$-dimensional torus. 
 It is foliated by Legendrian circles, as well as by transverse circles. 
\subsection{Generalization of the Lutz twist 
along a $\xi$-round hypersurface} 
\label{sec:circsph}
  We define the generalized Lutz twist along a $\xi$-round hypersurface 
$S^{2n-1}\times S^1$ 
modeled on the standard contact sphere of dimension $2n-1$ 
in a $(2n+1)$-dimensional contact manifold $(M,\xi)$. 
 It is a natural generalization of the $3$-dimensional Lutz twist 
along a pre-Lagrangian torus. 
 Because a pre-Lagrangian torus can be regarded 
as the boundary of the standard tubular neighborhood 
of a transverse knot in a contact $3$-manifold. 
 The generalized Lutz twist defined in Section~\ref{sec:genlztw} 
is also operated along a transverse circle 
in a $(2n+1)$-dimensional contact manifold $(M,\xi)$. 
 Then it is natural to consider some modification along
the boundary of the standard tubular neighborhood of a transverse circle. 
 Recall that, as a boundary of the standard tubular neighborhood 
of a transverse circle, 
we have a $\xi$-round hypersurface $S^{2n-1}\times S^1$ 
modeled on the standard contact sphere $(S^{2n-1},\eta_0)$ 
(See Example~\ref{expl:xirdmostctsph}). 

  This operation can also be defined by using contact round surgery 
(See Subsection~\ref{sec:ggtorbyrsrg}, and Figure~\ref{fig:l2ntw}). 

  In order to define a generalization of the Lutz twist along hypersurface, 
we use the generalized open Lutz tube $\lft(S^1\times\R^{2n},\zeta\rgt)$ 
defined in Subsection~\ref{sec:gLtube}. 
 Recall that $\zeta$ is the confoliation defined by the $1$-form 
$\omtw=\lft(\prod_{i=1}^n\cos r_i^2\rgt)d\phi
+\sum_{i=1}^n\sin r_i^2d\theta_i$ as $\zeta=\ker\omtw$. 
 An important notion is the double 
$\lft(\du,\zeta'\rgt)
=\lft(U(\sqrt{\pi}),\zeta\rgt)\cup\lft(U(\sqrt{\pi}),\zeta\rgt)$ of 
$U(\sqrt{\pi}):=\lft\{0\le r_i\le\sqrt{\pi}\rgt\}
\subset\lft(S^1\times\R^{2n},\zeta\rgt)$ 
and the contact structure $\tilde\zeta$ on $\du$ obtained from $\zeta'$ 
by Theorem~\ref{thm:approx}. 
 In order to define the generalized Lutz twist along a transverse circle, 
in Subsubsection~\ref{sec:ghltw}, 
we define the model $\pi$-Lutz tube $\mlt$ from the double $\du$ 
by removing the standard tubular neighborhood 
of one of the core transverse circle 
$S^1\times\{0\}\subset\lft(U(\sqrt{\pi}),\zeta\rgt)
\subset\lft(S^1\times\R^{2n},\zeta\rgt)$. 
 On the other hand, now we remove the standard tubular neighborhoods of 
the both of the core transverse circles from $\du$ (see Figure~\ref{fig:wgd}). 
%
%
\begin{figure}[htb]
  \centering
  {\small \includegraphics[height=3cm]{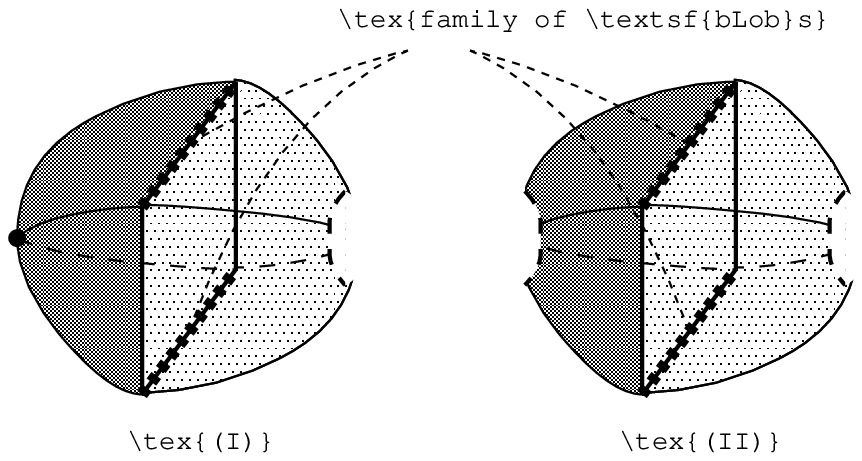}}
  \caption{(I) model Lutz tube $\mlt$, (II) wide Giroux domain $\wgd$. 
    $\pmod{\times S^1}$.}
  \label{fig:wgd}
\end{figure} 
 Then we have a contact manifold 
$\lft(S^1\times S^{2n-1}\times[0,1],\tilde\zeta\rgt)$ 
with a bordered Legendrian open book 
both of whose boundary components are $\tilde\zeta$-round hypersurfaces 
modeled on the standard contact sphere $(S^{2n-1},\eta_0)$. 
 Let $\lft(\wgd,\tilde\zeta\rgt)$ denote this contact manifold. 
 We call it the \emph{wide Giroux domain}. 
 We should remark that this $\lft(\wgd,\tilde\zeta\rgt)$ 
has an $S^1$-family of bordered Legendrian open books 
since the family in $\lft(\du,\tilde\zeta\rgt)$ is far 
from the both core transverse circles (see Figure~\ref{fig:wgd}). 
 Note that the meridional contact spheres $\{s\}\times S^{2n-1}\times\{i\}$, 
$s\in S^1$, $i=0,1$, of $\wgd\cong S^1\times S^{2n-1}\times[0,1]$ 
are homotopic in $\wgd$. 

  Now, we define a generalization of the Lutz twist along a hypersurface. 
 Let $H=S^{2n-1}\times S^1$ be a $\xi$-round hypersurface 
in a $(2n+1)$-dimensional contact manifold $(M,\xi)$ 
modeled on the standard contact sphere $(S^{2n-1},\eta_0)$. 
 Note that it has a meridional contact sphere $(S^{2n-1},\eta_0)$ 
and the longitudinal $S^1$-direction. 
 From Proposition~\ref{prop:xi-rdhyps}, 
there exists the standard tubular neighborhood along $H\subset(M,\xi)$. 
 Cut the contact manifold $(M,\xi)$ open along $H$. 
 Insert the wide Giroux domain $\wgd$ there. 
 Both components of $\rd\lft(\wgd\rgt)$ have the same tubular neighborhoods 
as $H$ from Proposition~\ref{prop:xi-rdhyps}, 
we can glue them so that the meridional contact spheres agree. 
 Then, without changing the manifold, we obtain a new contact structure 
which is the same as the given one except on the neighborhood of $H$. 
 The obtained contact structure has an $S^1$-family 
of bordered Legendrian open books since the inserted $\wgd$ has it. 
 We call this operation the \emph{generalized Lutz twist 
along a $\xi$-round hypersurface modeled on the standard contact sphere}. 

  Last of all, we mention the $k$-fold generalized Lutz twist 
along a hypersurface. 
 In a similar way, we can define the $k$-fold generalized Lutz twist 
for any positive integer $k\in\mathbb{N}$. 
 In this case, we use $\lft(U(\sqrt{(k+1)\pi/2}),\zeta\rgt)
\subset\lft(S^1\times\R^{2n},\zeta\rgt)$ 
instead of $\lft(U(\sqrt{\pi}),\zeta\rgt)$. 
 Then we obtain the \emph{wide Giroux $k\pi$-torsion domain}\/ 
$\lft(\wgd[k\pi],\tilde\zeta_k\rgt)$ 
from the double 
\begin{equation*}
  \lft(\du[\sqrt{(k+1)\pi/2}],\zeta'_k\rgt)
  =\lft(U(\sqrt{(k+1)\pi/2}),\zeta\rgt)
  \cup\lft(U(\sqrt{(k+1)\pi/2}),\zeta\rgt). 
\end{equation*}
 The \emph{$k$-fold generalized Lutz twist along a $\xi$-round hypersurface}\/ 
is defined 
as the operation inserting the wide Giroux $k\pi$-torsion domain 
$\lft(\wgd[k\pi],\tilde\zeta_k\rgt)$ 
after cutting a contact manifold $(M,\xi)$ open 
along a $\xi$-round hyper surface $H=S^{2n-1}\times S^1$ 
modeled on the standard contact sphere $(S^{2n-1},\eta_0)$. 
 Since $\lft(\wgd[k\pi],\tilde\zeta_k\rgt)$ includes an $S^1$-family 
of bordered Legendrian open books, 
the resulting contact structure is PS-overtwisted. 

%
%
\begin{rem}\label{rem:gltxird-diff}
  In dimension $3$, the operation above 
is the original $3$-dimensional Lutz twist along a pre-Lagrangian torus. 
 Although the $3$-dimensional one does not make overtwisted disks directly, 
the higher-dimensional generalized Lutz twist makes 
bordered Legendrian open books. 
 The difference is the positions of the bordered Legendrian open books 
in the generalized Lutz tube. 
 In dimension $3$, the core transverse circle which is to be removed 
intersects all overtwisted disks in the $S^1$-family of them. 
 However, in higher-dimensions, 
the family of bordered Legendrian open books in $U(\sqrt{\pi})$ 
does not intersect the core transverse circle (see Figure~\ref{fig:lth3diff}). 
%
%
\begin{figure}[htb]
  \centering
  {\small \includegraphics[height=3.9cm]{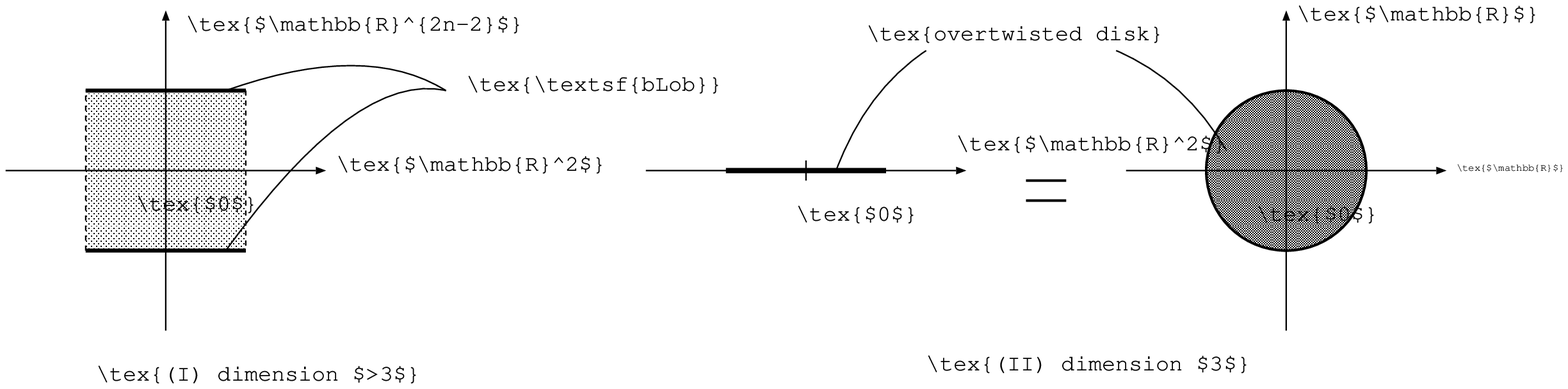}}
  \caption{\textsf{bLob} in Lutz tube $\pmod{\times S^1}$.}
  \label{fig:lth3diff}
\end{figure} 
 Concerning the higher-dimensional overtwisted disc 
in the sense of~\cite{newOT}, 
this operation does not make overtwisted discs directly 
(see Section~\ref{sec:newot}). 
\end{rem} 

\subsection{Giroux domain and generalization of Lutz twist 
along pre-Lagrangian torus} 
\label{sec:Gdom}
  Another generalization of the $3$-dimensional Lutz twist 
along a pre-Lagrangian torus is introduced in this subsection. 
 A generalization introduced in the preceding subsection 
creates a bordered Legendrian open book. 
 On the other hand, the operation introduced in this subsection 
creates a Giroux domain. 
 First, we find a natural Giroux domain appeared in the generalized Lutz tube 
(see Subsubsection~\ref{sec:gdilt}). 
 Then we define an operation that creates such Giroux domain 
(see Subsubsection~\ref{sec:alngtrs}). 

\subsubsection{Giroux domain in generalized Lutz tube}
\label{sec:gdilt}
 We find a Giroux domain in the generalized open Lutz tube 
$\lft(S^1\times\R^{2n},\zeta\rgt)$. 
 Then we define a certain domain as a subdomain 
in $\lft(S^1\times\R^{2n},\zeta\rgt)$ that contains the Giroux domain. 
 The domain can be considered as a generalization of 
the $3$-dimensional Giroux torsion domain. 
 It is needed to define a generalization of the Lutz twist 
in the following subsubsection. 

  First, we construct a simple model of the Giroux domain, 
which is to be found in the generalized open Lutz tube. 
 It is constructed as a generalization of the annulus-type example 
in Example~\ref{expl:annumodel}. 
 Set 
\begin{align*}
  &\Sigma^{2n}
  :=\overbrace{\lft(\lft[-\frac{\pi}{2},\frac{\pi}{2}\rgt]\times S^1\rgt)
  \times\dots\times\lft(\lft[-\frac{\pi}{2},\frac{\pi}{2}\rgt]\times S^1\rgt)}^n
  \cong T^n\times I^n, \\
  &\beta:=\frac{1}{\lft(\displaystyle \prod_{i=1}^n\cos s_i\rgt)}
  \sum_{j=1}^n\sin s_j\;d\theta_j
  =\frac{1}{\cos s_1\cdots\cos s_n}
  (\sin s_1\;d\theta_1+\dots+\sin s_n\;d\theta_n), 
\end{align*}
where $(s_1,\theta_1,s_2,\theta_2,\dots,s_n,\theta_n)\in\Sigma^{2n}$ 
are coordinates. 
 Then 
\begin{align*}
  \omega:=d\beta
  =\frac{1}{\displaystyle \prod_{i=1}^n\cos s_i}
  \lft\{\sum_{j=1}^n\frac{1}{\cos s_j}ds_j\wedge d\theta_j
  +\sum_{\substack{k,l=1\\k\ne l}}^n\frac{\sin s_k\sin s_l}{\cos s_l}\;ds_l\wedge d\theta_k\rgt\}
\end{align*}
is a symplectic structure on the interior $\inter\Sigma^{2n}$. 
 In fact, we have 
\begin{align*}
  \omega^n=\frac{n!}{\lft(\displaystyle \prod_{i=1}^n\cos s_i\rgt)^{n+1}}
  \begin{vmatrix}
    1&\sin s_1\sin s_2&\sin s_1\sin s_3&\dots&\sin s_1\sin s_n\\
    \sin s_2\sin s_1&1&\sin s_2\sin s_3&\dots&\sin s_2\sin s_n\\
    \sin s_3\sin s_1&\sin s_3\sin s_2&1&\dots&\sin s_2\sin s_n\\
    \vdots&\vdots& &\ddots&\vdots \\
    \sin s_n\sin s_1&\sin s_n\sin s_2&\sin s_n\sin s_3&\dots&1
  \end{vmatrix} \\ \hfill 
  ds_1\wedge ds_2\wedge d\theta_2\wedge\dots\wedge ds_n\wedge d\theta_n>0. 
\end{align*}
 We take a contact structure $\xi$ on $\rd\Sigma^{2n}$ given as follows. 
 The boundary $\rd\Sigma^{2n}$ is divided as 
$\rd\Sigma^{2n}=\cup U^i_\pm$, where
$U^i_\pm:=\lft\{(s_1,\theta_1,\dots,s_n,\theta_n)\in\Sigma^{2n}\mid 
s_i=\pm(\pi/2)\rgt\}$, $i=1,2,\dots,n$. 
 On each $U^i_\pm$, a contact structure $\xi$ is defined as 
\begin{equation*}
  \xi|_{U^i_\pm}
  :=\ker\lft(\pm d\theta_i+\sum_{\substack{j=1\\j\ne i}}^n\sin s_j\;d\theta_j\rgt). 
\end{equation*}
 Precisely, it is not contact at the end $\{s_j=\pm(\pi/2)\}\subset U^i_\pm$, 
or the corner of $\rd\Sigma^{2n}$. 
 However, if we smooth $\rd\Sigma^{2n}$ to the inside 
so that it is transverse to $\sum\tan s_i\;(\uvv{s_i})$, 
we may extend the contact structure on $\cup\inter U^i_\pm$ 
to the smoothed $\rd\Sigma^{2n}$ as the kernel 
of the restriction of $(\sum\sin s_i\;d\theta_i)$.
 We use the same notation for the smoothed things. 
 Then the triple $(\Sigma^{2n},\omega,\xi)$ is an ideal Liouville domain. 
 For a function $f\colon\Sigma^{2n}\to[0,\infty)$, 
$(s_1,\theta_1,\dots,s_n,\theta_n)\mapsto \prod_{i=1}^n\cos s_i$, 
which has $\rd\Sigma^{2n}$ as the regular level set $f^{-1}(0)$, 
the $1$-form $f\beta=\sum_{i=1}^n\sin s_i\;d\theta_i$ on $\Sigma^{2n}$ 
induces the contact structure $\xi$ on $\rd\Sigma^{2n}$. 
 Then the Giroux domain associated to this ideal Liouville domain 
$(\Sigma^{2n},\omega,\xi)$ is constructed as a contactization as follows: 
$\lft(\Sigma^{2n}\times S^1, \ker(fd\phi+f\beta)\rgt)$, 
where $\phi$ is the coordinate of the producted $S^1$. 
 The underlying manifold is 
$\Sigma^{2n}\times S^1\cong(T^n\times I^n)\times S^1\cong T^{n+1}\times I^n$. 
 And the induced contact structure is 
\begin{align*}
  \ker(fd\phi+f\beta)
  &=\ker\lft\{\lft(\prod_{i=1}^n\cos s_i\rgt)d\phi
  +\sum_{i=1}^n\sin s_i\;d\theta_i\rgt\} \\
  &=\ker\lft\{(\cos s_1\cdots\cos s_n)\;d\phi
  +\sin s_1\;d\theta_1+\dots+\sin s_n\;d\theta_n\rgt\}.
\end{align*}
 As you see, this is a contact structure 
that has already appeared in this paper.

%
%
\begin{rem}
  In dimension $3$, the Giroux domain constructed here is the same 
as that in Example~\ref{expl:annumodel}, that is, 
the $3$-dimensional Giroux $\pi$-torsion domain: 
$(T^2\times I,\ker(\cos s\;d\phi+\sin s\;d\theta))$. 
 In this sense, the Giroux domain constructed above is considered 
one of the simplest generalization 
of the $3$-dimensional Giroux $\pi$-torsion domain. 
\end{rem} 

  The Giroux domain $\lft(\Sigma^{2n}\times S^1,\ker(fd\phi+fd\beta)\rgt)$ 
that is constructed above 
exists in the generalized Lutz tube $\lft(S^1\times\R^{2n},\zeta\rgt)$. 
 Recall that $\zeta$ is the confoliation 
$\zeta=\ker\lft\{\lft(\prod_{i=1}^n\cos r_i^2\rgt)d\phi\rgt.
+\lft.\sum_{i=1}^n\sin r_i^2\;d\theta_i\rgt\}$ with the non-contact locus 
$\Sigma(\zeta)=\cup_{i\ne j}\lft\{\cos r_i^2=0,\ \cos r_j^2=0\rgt\}$ 
(see Subsection~\ref{sec:gLtube}). 
 Then the Giroux domain $\lft(\Sigma^{2n}\times S^1,\ker(f\phi+f\beta)\rgt)$ 
is contactomorphic to 
\begin{equation*}
  \lft\{(\phi,r_1,\theta_1,\dots,s_n,\theta_n)\in S^1\times\R^{2n}\ \lft|\ 
  \frac{\pi}{2}\le r_i^2\le\frac{3}{2}\pi,\ i=1,2,\dots,n\rgt.\rgt\} 
  \subset\lft(S^1\times\R^{2n},\zeta\rgt)
\end{equation*}
(see Figure~\ref{fig:GDinLT}). 
%
%
\begin{figure}[htb]
  \centering
  {\small \includegraphics[height=4.32cm]{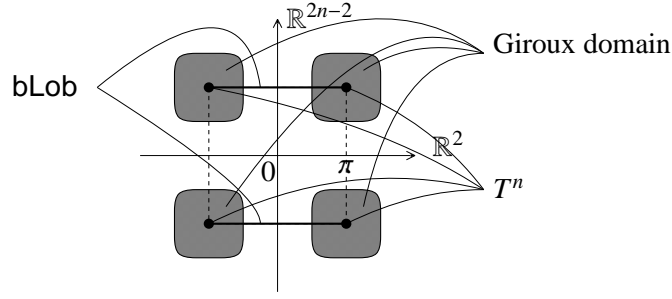}}
  \caption{Giroux domain in the generalized Lutz tube}
  \label{fig:GDinLT}
\end{figure} 
 Note that the Giroux domain does not contain the bordered Legendrian open books
$P_\phi=\lft\{0\le r_1^2\le\pi,\ r_2^2=\pi,\dots,\ r_n^2=\pi,\ \phi=\phi\rgt\}$ 
(compare Figures~\ref{fig:GDinLT} and~\ref{fig:lth3diff}). 
 We should remark that it is a tubular neighborhood 
of the $(n+1)$-dimensional torus 
$T^{n+1}:=\lft\{r_i^2=\pi,\ i=1,2,\dots,n\rgt\}
\subset\lft(S^1\times\R^{2n},\zeta\rgt)$. 
 The torus $T^{n+1}$ is pre-Lagrangian. 
 In fact, the level torus 
$T^n_\phi:=\lft\{\phi=\phi,\ r_i^2=\pi,\ i=1,2,\dots,n\rgt\}\subset T^{n+1}
\subset\lft(S^1\times\R^{2n},\zeta\rgt)$ is Legendrian 
for the contact structure part of 
$\zeta=\ker\lft\{\lft(\prod_{i=1}^n\cos r_i^2\rgt)d\phi
+\sum_{i=1}^n\sin r_i^2\;d\theta_i\rgt\}$. 
 Therefore, $T^{n+1}\subset\lft(S^1\times\R^{2n},\zeta\rgt)$ 
is a pre-Lagrangian torus with a product foliation with Legendrian leaves. 

\subsubsection{Generalization of the Lutz twist 
  along a pre-Lagrangian torus $T^{n+1}$}\label{sec:alngtrs}
  Now, we define a generalization of the Lutz twist 
along a pre-Lagrangian torus in a contact manifold. 
 Rough story is as follows. 
 First, we blow the contact manifold up along the pre-Lagrangian torus. 
 Then glue a tubular neighborhood of the pre-Lagrangian torus 
that contains the Giroux domain $\lft(\Sigma^{2n}\times S^1,\zeta\rgt)$ 
constructed above. 

  First, we discuss a blowing up procedure. 
 Let $T^{n+1}=T^n\times S^1$ be a pre-Lagrangian torus 
with Legendrian leaves $T^n\times\{\phi\}\subset T^{n+1}=T^n\times S^1$
in a contact $(2n+1)$-dimensional manifold $(M,\xi)$. 
 On account of Proposition~\ref{prop:nbdpL}, 
there exists the standard tubular neighborhood $U\subset(M,\xi)$ 
of $T^{n+1}\subset(M,\xi)$ which is contactomorphic to 
$\lft(S^1\times T^n\times D^n,\eta_0\rgt)$, 
where
\begin{equation*}
  \eta_0=\ker\lft\{d\phi+s_1d\theta_1+\dots+s_nd\theta_n\rgt\},
\end{equation*}
and $(\phi,\theta_1,\dots,\theta_n,s_1,\dots,s_n)\in S^1\times T^n\times D^n$ 
are coordinates. 
 By the spherical coordinate $(\rho,\psi_1,\dots,\psi_{n-1})$ 
of $D^n\subset\R^n$, that is, 
\begin{align*}
  &s_1=\rho\cos\psi_1,\ s_2=\rho\sin\psi_1\cos\psi_2,\dots \\
  &\quad\dots,\ s_{n-1}=\rho\sin\psi_1\cdots\sin\psi_{n-2}\cos\psi_{n-1},\ 
  s_n=\rho\sin\psi_1\cdots\sin\psi_{n-2}\sin\psi_{n-1}, 
\end{align*}
the contact structure $\eta_0$ is written as 
\begin{equation*}
  \eta_0=\ker\lft\{d\phi+(\rho\cos\psi_1)d\theta_1+\dots
  +(\rho\sin\psi_1\cdots\sin\psi_{n-1})d\theta_n\rgt\}. 
\end{equation*}
 On the other hand, we have a manifold $T^{n+1}\times\rd D^n\times[0,1]$ 
with a contact structure 
\begin{equation*}
  \eta_0'=\ker\lft\{d\phi'+(\rho'\cos\psi'_1)d\theta'_1\rgt.+\dots
  +\lft.(\rho'\sin\psi'_1\cdots\sin\psi'_{n-1})d\theta'_n\rgt\},
\end{equation*}
where $(\phi',\theta'_1,\dots,\theta'_n,\psi'_1,\dots,\psi'_{n-1},\rho')
\in T^{n-1}\times\rd D^n\times[0,1]$ are coordinates. 
 There 
 exists a contactomorphism 
\begin{align*}
  \varphi\colon&(T^{n+1}\times({D^n\setminus\{0\}}),\eta_0)
  \to(T^{n+1}\times(\rd D^n\times (0,1]),\eta_0'), \\
  &(\phi,\theta_1,\dots,\theta_n,\rho,\psi_1,\dots,\psi_{n-1})
  \mapsto(\phi,\theta_1,\dots,\theta_n,\psi_1,\dots,\psi_{n-1},\rho). 
\end{align*}
 Therefore, we can replace the tubular neighborhood $U$ 
of the pre-Lagrangian torus $T^{n+1}\subset(M,\xi)$ with 
$\lft(T^{n+1}\times\rd D^n\times(0,1],\eta_0'\rgt)$. 
 Let 
\begin{equation*}
  M_{\textup{bu}}(T^{n+1}):=\lft(M\setminus\inter(T^{n+1}\times D^n)\rgt)
  \cup\lft(T^{n+1}\times\rd D^n\times[0,1]\rgt)
\end{equation*}
denote the obtained manifold, 
and $\rd'M_{\textup{bu}}(T^{n+1})$ be the end 
that corresponds to $T^{n+1}\times S^1\times\{0\}$. 
 Thus, we obtain a contact structure $\eta$ on $M_{\textup{bu}}(T^{n+1})$ 
that satisfies: 
\begin{itemize}
\item $\xi\vert_{M\setminus T^{n+1}}$ is isomorphic to 
  $\eta\vert_{M_{\textup{bu}}(T^{n+1})\setminus\rd'M_{\textup{bu}}(T^{n+1})}$, 
\item $\eta$ near the end $\rd'M_{\textup{bu}}(T^{n+1})$ 
is isomorphic to $\eta_0'$ near the end $(T^{n+1}\times\rd D^n\times\{0\})$. 
\end{itemize}

  Next, we construct a tubular neighborhood 
of the pre-Lagrangian torus $T^{n+1}$ to be glued in. 
 Let $\lft(\Sigma^{2n}\times S^1,\zeta\rgt)$ be the Giroux domain 
constructed in Subsubsection~\ref{sec:gdilt}. 
 Recall that $\Sigma^{2n}\times S^1$ is diffeomorphic to $T^{n+1}\times D^n$ 
after some smoothing. 
 In addition, the core torus 
$T^{n+1}\subset\lft(\Sigma^{2n}\times S^1,\zeta\rgt)$ is pre-Lagrangian. 
 Then $\lft(\Sigma^{2n}\times S^1,\zeta\rgt)$ 
is regarded as a tubular neighborhood of the pre-Lagrangian torus $T^{n+1}$. 
 However, it can not be glued directly to $M_{\textup{bu}}(T^{n+1})$ 
along $\rd'M_{\textup{bu}}(T^{n+1})$. 
 We use a similar trick as the preceding generalized Lutz twists. 
 In other words, we take the double of the Giroux domain 
$\lft(\Sigma^{2n}\times S^1,\zeta\rgt)$. 
 Actually, the boundary of $\lft(\Sigma^{2n}\times S^1,\zeta\rgt)$ 
is a $\zeta$-round hypersurface. 
 Since the both boundaries of two $\lft(\Sigma^{2n}\times S^1,\zeta\rgt)$ 
are $\zeta$-round hypersurfaces modeled on the same contact manifold, 
we can glue two $\lft(\Sigma^{2n}\times S^1,\zeta\rgt)$ along the boundaries 
by Proposition~\ref{prop:xi-rdhyps}. 
 Let $(\Sigma_2,\zeta_2)$ denote the obtained double. 
 The manifold $\Sigma_2=(\Sigma^{2n}\times S^1)\cup_{\rd\Sigma^{2n}\times S^1}
(\Sigma^{2n}\times S^1)$ is diffeomorphic to $T^{n+1}\times S^n$. 
 We apply the blowing up procedure along the pre-Lagrangian torus 
$T^{n+1}\subset(\Sigma_2,\zeta_2)$ 
that corresponds to one of the core pre-Lagrangian torus 
of the Giroux domain $\lft(\Sigma^{2n}\times S^1,\zeta\rgt)$. 
 Thus, we obtain the contact manifold $(\Sigma_2',\zeta_2)$ 
whose underlying manifold $\Sigma_2'$ is diffeomorphic to $T^{n+1}\times D^n$. 
 We should remark here that, in $(\Sigma_2',\zeta_2)$, 
there remains one Giroux domain $\lft(\Sigma^{2n}\times S^1,\zeta\rgt)$. 
 In addition, $(\Sigma_2',\zeta_2)$ can be regarded 
as a tubular neighborhood of the core pre-Lagrangian torus $T^{n+1}$ 
in the Giroux domain 
$\lft(\Sigma^{2n}\times S^1,\zeta\rgt)\subset(\Sigma_2',\zeta_2)$. 

  Now, we glue the tubular neighborhood $(\Sigma_2',\zeta_2)$ of $T^{n+1}$ 
to $\lft(M_{\textup{bu}}(T^{n+1}),\eta\rgt)$ obtained by blowing up along $T^{n+1}$.
 The both ends are obtained by blowing up 
along the pre-Lagrangian torus $T^{n+1}=S^1\times T^n$ 
with Legendrian leaves $\{\phi\}\times T^n$, 
their neighborhoods are contactomorphic to the neighborhood of 
$T^{n+1}\times\rd D^n\times\{0\}\subset(T^{n+1}\times\rd D^n\times[0,1],\zeta)$. 
 There, the both contact hyperplanes are $\ker d\phi$. 
 Therefore, these contact manifolds are glued along the ends, 
so that their ``meridians'' agree. 
 As a result, we obtain the same manifold as the given $M$ 
with a new contact structure $\tilde\xi$. 
 This contact structure $\tilde\xi$ is the same as the original $\xi$ 
outside the modified tubular neighborhood of the pre-Lagrangian torus $T^{n+1}$,
since $\eta|_{M_{\textup{bu}}(T^{n+1})\setminus\rd'M_{\textup{bu}}(T^{n+1})}$ 
is isomorphic to $\xi|_{M\setminus T^{n+1}}$. 
 Moreover, the contact manifold $\lft(M,\tilde\xi\rgt)$ has one Giroux domain 
$\lft(\Sigma^{2n}\times S^1,\zeta\rgt)$ since it is included in the attached 
$(\Sigma_2',\zeta_2)$. 
 Then the claims in Theorem~\ref{thm:gdom-pLag} have been confirmed. 

  The procedures above amount to the definition of 
the \emph{generalized Lutz twist along a pre-Lagrangian torus}. 
 From the construction, there exists two Giroux domains glued together 
in the obtained contact manifold. 
 Then, from Theorem~\ref{thm:mnw_gdom}, the obtained contact manifold 
has no semi-positive strong symplectic filling. 
 The remark after Theorem~\ref{thm:gdom-pLag} is confirmed. 

%
%
\begin{rem}
  In dimension $3$, 
the generalized Lutz twist along a pre-Lagrangian torus $T^2$ defined above 
is the original $3$-dimensional $2\pi$-Lutz twist along a pre-Lagrangian torus 
with slope $1/n$, $n\in\Z$. 
 Such a pre-Lagrangian torus is foliated by Legendrian $S^1$ leaves. 
 The Giroux domain $\lft(\Sigma^2\times S^1,\zeta\rgt)$ 
is the Giroux $\pi$-torsion domain (see Example~\ref{expl:annumodel}). 
 Then the tubular neighborhood $(\Sigma'_2,\zeta_2)$ obtained from the double 
is the Giroux $2\pi$-torsion domain. 
 Therefore the operation above amounts to the $2\pi$-Lutz twist. 
\end{rem}

\section{Symplectic round handle and Contact round surgery}
\label{sec:ctrdsurg}
  Contact round surgery by using symplectic round handle 
is introduced in this section. 
 They were introduced first in \cite{art19}. 
 We review the definitions of the symplectic round handle 
in the following subsections. 
 Then we review the contact round surgery in Subsection~\ref{sec:hdimsurg}. 
 In this paper, we need contact round surgeries of index~$1$ and~$2n$ 
of $(2n+1)$-dimensional contact manifold. 
 Therefore, we observe the symplectic round handle of index~$1$ 
of any even dimension in Subsection~\ref{sec:ind12nsurg}. 
 Although the contact round surgery of index~$2n$ 
of a $(2n+1)$-dimensional contact manifold was not defined in \cite{art19}, 
it can be defined by using the convex hypersurface theory 
(see Subsubsection~\ref{sec:ind2nsurg}). 
 See also \cite{art20} for the contact round surgery of index~$2$ 
of a contact $3$-manifold. 

  We should remark that round handles with other kind of symplectic structures 
appear in \cite{gay}. 

\subsection{Symplectic round handle}\label{sec:symprdhdl}
  First, we define symplectic round handle. 

 An $n$-dimensional \emph{round handle}\/ of index $k$ 
attached to the boundary of an $n$-dimensional manifold $M$
is a pair
\begin{equation*}
  R_k=\lft(D^k\times D^{n-k-1}\times S^1, f\rgt)
\end{equation*}
consists of a product of an $(n-1)$-dimensional disk $D^k\times D^{n-k-1}$ 
with corner and a circle, 
and an attaching embedding 
$f\colon\rd_-\lft(D^k\times D^{n-k-1}\times S^1\rgt)\to\rd M$, 
where 
\begin{equation*}
  \rd_-\lft(D^k\times D^{n-k-1}\times S^1\rgt):=\rd D^k\times D^{n-k-1}\times S^1 
  \subset\rd\lft(D^k\times D^{n-k-1}\times S^1\rgt)
\end{equation*}
is the attaching region. 
 It was introduced by Asimov~\cite{asimov} 
to study non-singular Morse-Smale flows. 

  A \emph{symplectic round handle}\/ is defined as 
the model symplectic round handle 
and its attachment to the boundary of a symplectic manifold. 
 First, we define the model symplectic round handles, 
or symplectic structures on round handles. 
 The $2n$-dimensional model symplectic round handle of index $k$, 
($k=0,1,2,\dots,n-1$), is defined as a subset of $\R^{2n-1}\times S^1$ 
with the symplectic structure 
\begin{equation*}
  \omega_0:=\sum_{i=1}^{n-1}\lft(dp_i\wedge dq_i\rgt)+dz\wedge d\phi, 
\end{equation*}
where $(p_1,\dots,p_{n-1},q_1,\dots,q_{n-1},z,\phi)$ are coordinates 
of $\R^{2n-1}\times S^1=\R^{2(n-1)}\times\R\times S^1$. 
 Set a vector field $X_k$ on $\ostr$ as 
\begin{equation}\label{eq:liouville}
  X_k:=\sum_{i=1}^k\lft(-q_i\uv{q_i}+2p_i\uv{p_i}\rgt)
  +\frac12\sum_{i=k+1}^{n-1}\lft(q_i\uv{q_i}+p_i\uv{p_i}\rgt)+z\uv z, 
\end{equation}
for each $k=1,2,\dots, n-1$ 
(see the dotted curves and arrows in Figure~\ref{fig:symprdhdl}). 
 It is the so called Liouville vector field 
on $\lft(\R^{2n-1}\times S^1, \omega_0\rgt)$. 
 The \emph{Liouville vector field}\/ is defined 
as a vector filed $X$ on a symplectic manifold $(W,\omega)$ 
which satisfies $L_X\omega=\omega$. 
 It is well known that on a hypersurface $M\subset(W,\omega)$ 
transverse to the Liouville vector field $X$, 
a contact form is induced as $(X\intp\omega)\vert_{TM}$ (see \cite{weinstein}). 
 Then we take a round handle as a subset 
of $\lft(\R^{2n-1}\times S^1,\omega_0\rgt)$ 
so that its boundary is transverse to $X_k$. 
 Let $f_k$, $g_k$ be functions defined as follows: 
\begin{align}\label{eq:bdrhdl}
  f_k(p_i,q_i,z,\phi)&:=\sum_{i=1}^k\lft(-\frac12q_i^2+p_i^2\rgt)
  +\frac14\sum_{i=k+1}^{n-1}\lft(q_i^2+p_i^2\rgt)+\frac12z^2, \notag \\
  g_k(p_i,q_i,z,\phi)&:=-A\sum_{i=1}^kq_i^2
  +B\lft(\sum_{i=1}^kp_i^2+\sum_{i=k+1}^{n-1}\lft(q_i^2+p_i^2\rgt)+z^2\rgt). 
\end{align}
 By taking the positive constants $A$, $B$ suitably, 
we can cut off the subset $R^{(0)}_k\subset(\R^{2n-1}\times S^1,\omega_0)$ 
which is diffeomorphic to $D^k\times D^{2n-k-1}\times S^1$ as 
\begin{equation*}
  R^{(0)}_k
  :=\lft\{x=(p_i,q_i,z,\phi)\in\ostr\mid f_k(x)\ge-1,\ g_k(x)\le c\rgt\}, 
\end{equation*}
for some constant $c>0$ (see Figure~\ref{fig:symprdhdl}). 
%
%
\begin{figure}[htb]
  \centering
  {\footnotesize \includegraphics[height=5cm]{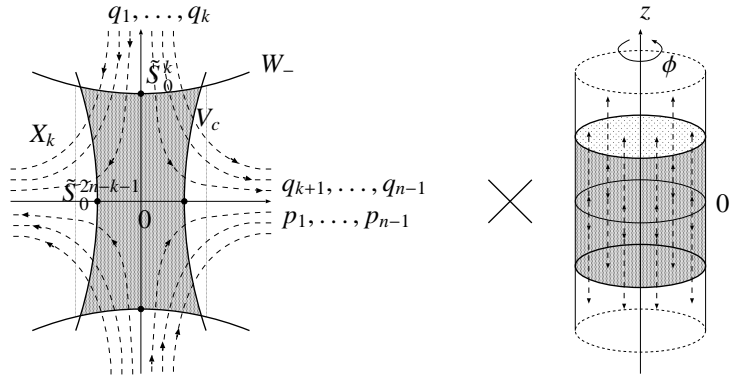}}
  \caption{Symplectic round handle}
  \label{fig:symprdhdl}
\end{figure} 
 We call $\lft(R^{(0)}_k,\omega_0\rgt)$ 
the $2n$-dimensional \emph{model symplectic round handle}\/ of index $k$, 
($k=1,2,\dots,n-1$). 

  The model symplectic round handle is attached 
to the boundary of a symplectic manifold as follows. 
 Let $(W,\omega)$ be a $2n$-dimensional symplectic manifold with boundary, 
and $M\subset\rd W$ a convex contact type subset of the boundary. 
 In other words, there exists a Liouville vector field $X$ 
defined near $M\subset (W,\omega)$ which is transverse to $M$ and 
looks outward at $M$. 
 Note that a contact form is induced on $M$ from $X$ and $\omega$. 
 Let $\tilde S^k=S^{k-1}\times S^1\subset M$, $k=1,2,\dots,n-1$, 
be an isotropic product of a sphere and a circle 
with a trivialization of the conformal symplectic normal bundle 
$\csn(\tilde S^k, M)$ 
with respect to the contact structure induced on $M$. 
 Then the following property is proved in \cite{art19}: 
%
%
\begin{prop}\label{prop:rhatt}
 The $2n$-dimensional model symplectic round handle 
$\lft(R^{(0)}_k,\omega_0\rgt)$ of index $k$ 
can be attached to $(W,\omega)$ along $\tilde S^k\subset M$. 
 In other words,  
the Liouville vector field $X$ and the symplectic structure $\omega$ 
extend to the manifold 
obtained from $W$ by attaching a round handle of index $k$ 
along $\tilde S^k$ 
so that the modified boundary is still convex. 
\end{prop} 

\subsection{Contact round surgeries}\label{sec:hdimsurg}
  We review the definition of contact round surgery 
by using the symplectic round handles. 
 Recall that a $(2n+2)$-dimensional round handle of index~$k$ is attached 
to the boundary of a $(2n+2)$-dimensional manifold 
along a product $S^{k-1}\times S^1$ of a $(k-1)$-sphere and a circle. 
 Let $(M,\xi=\ker\alpha)$ be a $(2n+1)$-dimensional contact manifold, 
and $L\cong S^{k-1}\times S^1$ an isotropic submanifold in $(M,\xi)$ 
with a trivialization of the conformal symplectic normal bundle $\csn(L,M)$. 
 We will define a contact round surgery of $(M,\xi)$ along $L$. 
 Topologically, round surgery is defined by attaching a round handle 
to the trivial cobordism $M\times[0,1]$. 
 We first construct a symplectic structure on $M\times[0,1]$. 
 Then we attach a symplectic round handle to it. 

  The trivial symplectic cobordism is taken as follows. 
 Take 
$\lft(M\times[0,1], d(e^t\alpha)|_{M\times[0,1]}\rgt)
\subset\lft(M\times\R,d(e^t\alpha)\rgt)$ 
as the trivial symplectic cobordism, in the symplectization of $(M,\xi)$. 
 Set $\omega:=d(e^t\alpha)|_{M\times[0.1]}$. 
 Both the boundary components $M\times\{0\},\ M\times\{1\}\subset M\times[0,1]$ 
are of concave and convex contact type respectively. 
 In fact, $X=\uvv t$ is the Liouville vector field. 
 The induced contact forms are $\alpha$ on $M\times\{0\}$ 
and $e\alpha$ on $M\times\{1\}$. 
 We identify both $(M\times\{0\},\ker\alpha)$ 
and $(M\times\{1\},\ker(e\alpha))$ with $(M,\xi=\ker\alpha)$. 

  We are attaching the standard symplectic round handle 
to the convex end $M\times\{1\}$ of the trivial cobordism 
$(M\times[0,1],\omega)$. 
 Attaching the $(2n+2)$-dimensional standard symplectic round handle 
of index~$k$ 
to the trivial cobordism $(M\times[0,1],\omega)$ 
along $L\subset M\times\{1\}$, 
by Proposition~\ref{prop:rhatt}, 
we obtain a new cobordism $W\cong(M\times[0,1])\cup R_k$. 
 The modified end, which $\tilde M$ denotes, is the manifold
obtained from $M$ by a round surgery of index~$k$. 
 From Proposition~\ref{prop:rhatt}, 
the Liouville vector field $X=\uvv t$ on $(M\times [0,1],\omega)$ extends 
to the Liouville vector field $\tilde X$ on $W$ 
which is also transverse to the new boundary $\tilde M\subset\rd W$ 
and looks outward there. 
 Therefore, the surgered manifold $\tilde M$ inherits 
the contact structure $\tilde\xi$ 
which is the same as the original $\xi=\ker\alpha$ 
except where the attachment takes place. 
 It is proved in \cite{art19} 
that the obtained contact manifold $(\tilde M,\tilde\xi)$ does not depend 
on the choice of the model symplectic round handle. 

  The operation above, 
constructing a new contact manifold $(\tilde M,\tilde\xi)$ 
from the given contact manifold $(M,\xi)$, 
is called the \emph{contact round surgery} of index~$k$ 
along the isotropic submanifold $L\cong S^{k-1}\times S^1$. 
 Note that, for $k=0,1,\dots,n-1$, the contact round surgeries of index~$k$ 
of $(2n+1)$-dimensional contact manifold is defined by this method. 

\subsection{Contact round surgeries of index $1$ and $2n$} 
\label{sec:ind12nsurg}
  In order to describe the generalized Lutz twists, 
we need the contact round surgery of index~$1$ and~$2n$ 
of a $(2n+1)$-dimensional contact manifold. 
 We observe these two kinds of surgery more carefully. 
 The contact round surgery of index~$2n$ 
of a $(2n+1)$-dimensional contact manifold has not been defined yet. 
 Following the observation of the contact round surgery of index~$1$ 
in Subsubsection~\ref{sec:ind1surg}, 
it is defined in Subsubsection~\ref{sec:ind2nsurg}. 
 Both such contact round surgeries of a $(2n+1)$-dimensional contact manifold 
are defined by the $(2n+2)$-dimensional symplectic round handle of index~$1$. 
 Then we observe this first in the following subsubsection. 

\subsubsection{Symplectic round handle of index $1$}
\label{sec:symprhi1}
  We observe a $(2n+2)$-dimensional symplectic round handle of index $1$. 
 Contact round surgeries of index $1$ and $2n$ are defined 
by removing and gluing the attaching and belt regions 
of the symplectic round handle of index $1$ 
(see Subsection~\ref{sec:hdimsurg}). 
 Therefore, we observe contact structures 
induced on the attaching and belt region of the symplectic round handle 
of index $1$. 

  From Subsection~\ref{sec:symprdhdl}, 
a $(2n+2)$-dimensional symplectic round handle of index $1$ 
is defined as follows. 
 It is a certain subset of $\R^{2n+1}\times S^1$ with the symplectic structure 
\begin{equation*}
  \omega_0=\sum_{i=1}^ndp_i\wedge dq_i+dz\wedge d\phi
\end{equation*}
where $(p_1,\dots,p_n,q_1,\dots,q_n,z,\phi)$ are coordinates 
of $\R^{2n+1}\times S^1$. 
 The model round handle of index $1$ is the region 
$R^{2n+2}_1=\lft\{f_1\ge-1,\ g_1\le c\rgt\}\subset\R^{2n+1}\times S^1$, 
where $f_1$, $g_1$ are functions defined as Equations~\eqref{eq:bdrhdl}, 
and $c$ is some constant (see Figure~\ref{fig:symprdhdl}). 
 The Liouville vector field for the symplectic round handle of index $1$ is 
\begin{equation*}
  X_1:=-q_1\uv{q_1}+2p_1\uv{p_1}
  +\frac12\sum_{i=2}^{n-1}\lft(q_i\uv{q_i}+p_i\uv{p_i}\rgt)+z\uv z. 
\end{equation*}

  We observe the contact structure induced on the attaching region 
$\rd_-R^{2n+2}_1$. 
 For the model symplectic round handle, 
the attaching region $\rd_-R^{2n+2}_1$ is in $\{f_1=-1\}=:W_-$ 
(see Figure~\ref{fig:symprdhdl}). 
 From the fact that the diffeomorphism constructed by the flow lines 
of the Liouville vector field $X_1$ preserves the induced contact structure, 
we calculate on $A_\pm:=\{q_1=\pm1\}$ instead. 
 The contact form induced on $A_\pm$ from $\omega_0$ and $X_1$ is 
\begin{align*}
  (X_1\intp\omega_0)|_{TA_\pm}
  &=\lft.\lft\{2p_1dq_1+q_1dp_1
  +\frac{1}{2}\sum_{i=2}^n(p_idq_i-q_idp_i)+zd\phi\rgt\}\rgt|_{TA_\pm} \\
  &=(\pm dp_1+zd\phi)+\frac{1}{2}\sum_{i=2}^n(p_idq_i-q_idp_i). 
\end{align*}
 The attaching core is 
$\lft\{q_1=\pm1,q_2=0,\dots,q_n=0,p_1=0,\dots,p_n=0\rgt\}\cong S^1\sqcup S^1$. 
 Then the attaching region $\rd_-R^{2n+2}_1$ consists of two connected components,
and each component is the standard tubular neighborhood 
of an isotropic circle with trivial $\csn(S^1,\rd_-R^{2n+2}_1)$. 

  The boundary of the attaching region $\rd_-R^{2n+2}_1$ 
is a convex hypersurface 
which is diffeomorphic to $S^1\times \rd D^{2n}$ 
with a dividing set $S^1\times S^{2n-2}\subset S^1\times \rd D^{2n}$. 
 In fact, we can verify its dividing set as follows. 
 We work on one component $A_+=\{q_1=1\}$, 
and let $\alpha$ denote the induced contact form $(X_1\intp\omega_0)|_{TA_+}$. 
 On each local charts $\{p_1=\pm1\}$, $\{p_i=\pm1\}$, $\{q_i=\pm1\}$, 
($i=2,3,\dots,n$), and $\{z=\pm1\}$, we calculate the characteristic foliation. 
 On $\{p_1=\pm1\}$, we have 
\begin{align*}
  (\alpha\wedge(d\alpha)^{n-1})|_{T\{p_i=\pm1\}}
  =&zd\phi\wedge\lft(\bigwedge_{i=2}^ndp_i\wedge dq_i\rgt) 
  +\frac{1}{2}\sum_{i=2}^np_idq_i
  \wedge\lft(\bigwedge_{j\ne i}dp_j\wedge dq_j\rgt)\wedge dz\wedge d\phi \\
  &-\frac{1}{2}\sum_{i=2}^nq_idp_i
  \wedge\lft(\bigwedge_{j\ne i}dp_j\wedge dq_j\rgt)\wedge dz\wedge d\phi
\end{align*}
 Then the vector field $V$ that generates the characteristic foliation 
for the volume form 
$\Omega=\pm\lft(\bigwedge_{i=2}^ndp_i\wedge dq_i\rgt)\wedge dz\wedge d\phi$
(i.e.~$V\intp\Omega=\alpha\wedge d\alpha|_{T\{p_1=\pm1\}}$) is 
\begin{equation*}
  V=\pm\lft\{z\uv z
  +\frac{1}{2}\sum_{i=2}^n\lft(p_i\uv{p_i}+q_i\uv{q_i}\rgt)\rgt\}. 
\end{equation*}
 Similarly, on $\{p_i=\pm1\}$, 
the vector field is $V=-\uvv{q_i}+(1/2)\uvv{p_1}$. 
 On $\{q_i=\pm1\}$, the vector field is $V=\uvv{p_i}+(1/2)\uvv{p_1}$. 
 On $\{z=\pm1\}$, the vector field is $V=-\uvv{\phi}+\uvv{p_1}$. 
 Then $\rd\lft(\rd_-R^{2n+2}_1\rgt)\cap\{p_1=0\}\cong S^1\times S^{2n-2}$ 
is a dividing set. 

  Next, we observe the contact structure induced on the belt region 
$\rd_+R^{2n+2}_1$. 
 For the model symplectic round handle, 
the belt region $\rd_+R^{2n+2}_1$ is in $\{g_1=c\}=:V_c$ 
(see Figure~\ref{fig:symprdhdl}). 
 From the same reason as above, 
it is sufficient to calculate on 
$\lft\{p_1^2+\sum_{i=2}^n(p_i^2+q_i^2)+z^2=1\rgt\}
\cong\R\times S^{2n-1}\times S^1$. 
 Further, for convenience, we calculate on 
$B^1_{\pm}:=\{p_1=\pm1\}$, $B^{p_i}_{\pm}:=\{p_i=\pm1\}$, 
$B^{q_i}_{\pm}:=\{q_i=\pm1\}$, ($i=2,\dots,n$), and $B^z_\pm:=\{z=\pm1\}$. 

  The contact form induced on $B^1_\pm$ is 
\begin{equation*}
  \alpha=(X_1\intp\omega_0)|_{TB^1_\pm}
  =\pm2dq_1+\frac{1}{2}\sum_{i=2}^n(p_idq_i-q_idp_i)+zd\phi. 
\end{equation*}
 On the belt core $BC:=\{q_1=0\}\subset B^1_\pm$ with the volume form 
$\Omega=\mp\bigwedge_{i=2}^n(dp_i\wedge dq_i)\wedge dz\wedge d\phi$, 
the vector field $V$ that generate the characteristic foliation $BC_\xi$, 
(i.e. $V\intp\Omega=\alpha\wedge(d\alpha)^{n-1}$), is 
\begin{equation*}
  V=\mp\frac{1}{2}\sum_{i=2}^n\lft(p_2\uv p_2+q_2\uv q_2\rgt)+z\uv z. 
\end{equation*}

  By similar arguments, we obtain the induced contact structures 
and the characteristic foliation on the belt core. 
 On $B^{p_i}_{\pm}=\{p_i=\pm1\}$, the induced contact form is 
\begin{equation*}
  \alpha=q_1dp_1+2p_1dq_1\pm\frac{1}{2}dq_i
  +\frac{1}{2}\sum_{\substack{j=2\\j\ne i}}^n(p_jdq_j-q_jdp_j)+zd\phi. 
\end{equation*}
 The vector field on $BC=\{q_1=0\}\subset B^{p_i}_\pm$ 
generating the characteristic foliation is $V=(1/2)\uvv{p_1}$. 
 On $B^{q_i}_{\pm}=\{q_i=\pm1\}$, the induced contact form is 
\begin{equation*}
  \alpha=q_1dp_1+2p_1dq_1\mp\frac{1}{2}dp_i
  +\frac{1}{2}\sum_{\substack{j=2\\j\ne i}}^n(p_jdq_j-q_jdp_j)+zd\phi. 
\end{equation*}
 The vector field on $BC=\{q_1=0\}\subset B^{q_i}_\pm$ 
generating the characteristic foliation is $V=(1/2)\uvv{p_1}$. 
 On $B^z_{\pm}=\{z=\pm1\}$, the induced contact form is 
\begin{equation*}
  \alpha=q_1dp_1+2p_1dq_1
  +\frac{1}{2}\sum_{j=2}^n(p_jdq_j-q_jdp_j)\pm d\phi. 
\end{equation*}
 The vector field on $BC=\{q_1=0\}\subset B^{z}_\pm$ 
generating the characteristic foliation is $V=\uvv{p_1}$. 

  As a consequence, from Theorem~\ref{thm:convFlex}, 
the belt region $\rd_+R^{2n+2}_1$ is the invariant tubular neighborhood 
of the convex hypersurface $BC\cong S^{2n-1}\times S^1$ 
with the dividing set $S^{2n-2}\times S^1\subset S^{2n-1}\times S^1\cong BC$. 

\subsubsection{Contact round surgery of index $1$}
\label{sec:ind1surg}
  A precise description of the contact round surgery of index $1$ 
of a $(2n+1)$-dimensional contact manifold is given here. 
 By definition, the surgery is an operation 
attaching a $(2n+2)$-dimensional symplectic round handle of index $1$ 
to the convex boundary of the trivial cobordism of the given contact manifold 
along an isotropic circle 
with a trivialization of the conformal symplectic normal bundle 
(see Subsection~\ref{sec:hdimsurg}). 
 In other words, it is the operation 
removing the attaching region of the symplectic round handle 
and regluing the belt region of the symplectic round handle. 
 We have observed what the regions are, in Subsubsection~\ref{sec:symprhi1}. 
 We describe the surgery from the view point 
of $(2n+1)$-dimensional contact manifolds. 

  The situation is the following. 
 Let $(M,\xi)$ be a $(2n+1)$-dimensional contact manifold, 
and $L_i\subset(M,\xi)$, $i=1,2$, two non-intersecting isotropic circles 
with trivializations of $\csn(L_i,M)$. 

  First, we remove the interior of some tubular neighborhoods 
of $L_1$ and $L_2$. 
 By Proposition~\ref{prop:isotnbd}, 
there exist the standard tubular neighborhoods $U_i$ of $L_i\subset(M,\xi)$. 
 The symplectic round handle is attached there. 
 Then we remove the interior of the attaching region 
of the $(2n+2)$-dimensional symplectic round handle of index $1$, 
which is the standard tubular neighborhood $\tilde U_i\subset (M,\xi)$ 
of the isotropic circles $L_1$, $L_2$ 
(See Subsubsection~\ref{sec:symprhi1}). 
 Recall also that the boundaries $\rd\tilde U_i\cong S^{2n-1}\times S^1$ 
are convex hypersurfaces with the dividing sets 
$S^{2n-2}\times S^1\subset S^{2n-1}\times S^1$. 

  Next, we reglue a tubular neighborhood of $S^{2n-1}\times S^1$ 
with some contact structure. 
 By definition, it is the belt region 
of the $(2n+2)$-dimensional symplectic round handle of index~$1$. 
 From the observation in Subsubsection~\ref{sec:symprhi1}, 
it is the invariant tubular neighborhood of the convex hypersurface 
$S^{2n-1}\times S^1$ with the dividing set 
$S^{2n-2}\times S^1\subset S^{2n-1}\times S^1$. 
 Then, since both the section by the removal above 
and the boundary of the invariant tubular neighborhood 
are convex hypersurfaces with the same dividing set, 
they are glued together. 
 However, it is clear that the boundaries of the attaching and belt regions 
of the symplectic round handle agree, though (see Figure~\ref{fig:symprdhdl}). 

  Thus, we obtain a new $(2n+1)$-dimensional contact manifold. 
 This procedure is the contact round surgery of index~$1$ 
of a $(2n+1)$-dimensional contact manifold $(M,\xi)$ 
along an isotropic link $L_1\sqcup L_2\subset(M,\xi)$ 
with the trivialization of $\csn(L_i,M)$ 
from the view point of contact manifold. 

\subsubsection{Contact round surgery of index $2n$}
\label{sec:ind2nsurg}
  Following the observation above, 
the contact round surgery of index~$2n$ 
of a contact $(2n+1)$-dimensional manifold is introduced here. 
 Recall that it is not defined in Subsection~\ref{sec:hdimsurg} 
because the $(2n+2)$-dimensional symplectic round handle of index~$k$ 
is defined for $k=0,1,\dots,n$. 
 In order to define the surgery, 
we need the convex hypersurface theory 
instead of the neighborhood theorem for isotropic submanifold. 

  First, we review the topological $(2n+1)$-dimensional round surgery 
of index~$2n$. 
 Let $M$ be a $(2n+1)$-dimensional manifold. 
 The round surgery of index~$2n$ of $M$ is defined 
by the attachment of $(2n+2)$-dimensional round handle 
$R^{2n+2}_{2n}=D^{2n}\times D^1\times S^1$ of index~$2n$ 
to the boundary of the trivial cobordism $M\times[0,1]$. 
 In other words, the surgery is the operation 
removing the attaching region $\rd_-R^{2n+2}_{2n}=\rd D^{2n}\times D^1\times S^1$ 
and regluing the belt region $\rd_+R^{2n+2}_{2n}=D^{2n}\times\rd D^1\times S^1$
trivially. 
 As manifolds, the $(2n+2)$-dimensional round handles 
$R^{2n+2}_{2n}$, $R^{2n+2}_1$ of index~$2n$ and~$1$ are diffeomorphic. 
 The attaching (resp.\ belt) region of $R^{2n+2}_{2n}$ 
is the belt (resp.\ attaching) region of $R^{2n+2}_1$. 
 Therefore, the round surgery of index~$2n$ can be regarded 
as the converse operation to that of index~$1$. 

  From the observation above, 
we define the contact round surgery of index~$2n$ 
of a $(2n+1)$-dimensional contact manifold 
as the converse operation to that of index~$1$. 
 Let $(M,\xi)$ be a $(2n+2)$-dimensional contact manifold, 
and $H=S^{2n-1}\times S^1\subset(M,\xi)$ a convex hypersurface 
with the dividing set $S^{2n-2}\times S^1\subset H$, that is, 
the product of the equator of $S^{2n-1}$ and $S^1$. 
 As $H\subset(M,\xi)$ is convex, 
there exists the invariant tubular neighborhood $U\cong H\times[0,1]$ of $H$. 
 We remove the interior $\inter U\subset(M,\xi)$ 
of the invariant tubular neighborhood. 
 Note that the section appeared by this removal 
is two convex hypersurfaces $H\times\{0,1\}$ 
with dividing sets $S^{2n-2}\times S^1\times\{0,1\}$. 
 Take the manifold $D^{2n}\times S^1$ 
with the contact structure 
\begin{equation*}
  \eta:=\ker\lft\{(\cos\phi)dp_1-(\sin\phi)dq_1
  +\frac{1}{2}\sum_{i=2}^n(p_idq_i-q_idp_i)\rgt\}, 
\end{equation*}
where $(p_1,q_1,\dots,p_n,q_n,\phi)\in D^{2n}\times S^1$ are coordinates. 
 Then the boundary $\rd D^{2n}\times S^1\subset(D^{2n}\times S^1,\xi)$ 
is convex with the dividing set $S^{2n-2}\times S^1$. 
 It is contactomorphic to one component of the attaching region 
of the symplectic round handle. 
 We can glue two $(D^{2n}\times S^1,\eta)$ to $(M,\xi)\setminus\inter U$ 
according to the dividing sets. 
 Thus, we obtain a new $(2n+1)$-dimensional contact manifold, 
which is diffeomorphic, as manifolds, 
to the manifold obtained from $M$ by round surgery of index~$2n$ 
along the hypersurface $H=S^{2n}\times S^1$. 
 We call this operation the \emph{contact round surgery of index~$2n$} 
of the $(2n+2)$-dimensional contact manifold $(M,\xi)$ 
along the convex hypersurface $H=S^{2n-1}\times S^1$ 
with dividing set $S^{2n-2}\times S^1\subset H$. 


\section{Generalized Lutz twists by contact round surgeries}
\label{sec:gentw}
  The generalized Lutz twists 
defined in Section~\ref{sec:genlztw} and Subsection~\ref{sec:circsph}
are described by using contact round surgeries in this section 
(in Subsection~\ref{sec:ltwtrvknt} and~\ref{sec:ggtorbyrsrg} respectively). 
 In other words, the model $\pi$-Lutz tube along any transverse knot, 
and the wide Giroux domain along any $\xi$-round hypersurface 
$S^{2n-1}\times S^1$ modeled on the standard contact sphere 
in a $(2n+1)$-dimensional contact manifold are realized 
by contact round surgeries of index $1$ and $2n$ 
with the model $\pi$-Lutz tube. 
 Then Theorem~\ref{thm:hdim} is proved here. 

  Further, from this point of view, 
we notice that the important object is the double $\lft(\du,\tilde\zeta\rgt)$. 
 The two modifications are regarded as procedures taking in this unit 
(see Figures~\ref{fig:l1tw} and~\ref{fig:l2ntw}). 

  Although the generalized Lutz twist is operated along a transverse circle, 
the contact round surgery of index~$1$ is operated along isotropic circles. 
 Similarly, although the generalized Lutz twist is operated 
along a $\xi$-round hypersurface, 
the contact round surgery of index~$2n$ is operated along a convex hypersurface.
 Such approximation procedures are introduced in Subsection~\ref{sec:approx}. 

\subsection{Approximation procedures}\label{sec:approx}
  Before giving the descriptions of the generalized Lutz twists 
by contact round surgeries, 
we should prepare the attaching cores of the symplectic round handles 
to apply surgeries. 
 We need isotropic circles and convex hypersurfaces 
instead of the given transverse circles and $\xi$-round hypersurfaces. 

\subsubsection{Isotropic push-off}\label{sec:istrpoff}
 First, we need a method to take an isotropic curve 
close to the given transverse curve.
 It is a generalization of the so-called Legendrian push-off 
of a transverse curve in a contact $3$-manifold. 
 The argument is similar to the transverse push-off 
in Subsubsection~\ref{sec:findtrvs1}. 

  Let $\gamma$ be a transverse circle 
in a $(2n+1)$-dimensional contact manifold $(M,\xi)$. 
 Then, from Corollary~\ref{cor:s1dar}, 
there exists the standard tubular neighborhood $U\subset(M,\xi)$ of $\gamma$ 
which is contactomorphic to some tubular neighborhood of a transverse circle 
$S^1\times\{0\}$ in the standard contact open solid torus 
$(S^1\times\R^{2n},\xi_0)$ with  
$\xi_0=\ker\lft\{d\phi+\sum_{i=1}^n(x_idy_i-y_idx_i)\rgt\}$, 
where $\phi$ is a coordinate of $S^1$, 
$(x_i,y_i)$ are coordinates of $\R^2$, $i=1,\dots,n$. 
 By the projection $\pi\colon S^1\times\R^{2n}\to S^1\times\R^2$, 
$(\phi,x_1,y_1,\dots,x_n,y_n)\mapsto(\phi,x_1,y_1)$, 
we have a transverse curve $\pi(\gamma)\subset(S^1\times\R^2,\pi_\ast\xi_0)$ 
in the $3$-dimensional standard contact tubular neighborhood. 
 Then we can take a Legendrian push-off $L(\pi(\gamma))$ of $\pi(\gamma)$. 
 In other words, we take a closed leaf of slope $1/n$, $n\in\Z$, 
of the characteristic foliation on the pre-Lagrangian boundary 
of the standard tubular neighborhood 
of $\pi(\gamma)\subset(S^1\times\R^2,\pi_\ast\xi_0)$. 
 As a circle in $S^1\times\R^{2n}\supset S^1\times\R^2$, 
the obtained curve $L(\pi(\gamma))$ is isotropic for $\xi_0$, 
and is homotopic to the original curve $\gamma$. 
 We call the obtained curve $L(\pi(\gamma))\subset(M,\xi)$ 
an \emph{isotropic push-off}\/ of $\gamma$. 
 In addition, we have a tubular neighborhood 
$\tilde U\subset(M,\xi)$ of $L(\pi(\gamma))$ 
which is contactomorphic to some tubular neighborhood of an isotropic circle 
$S^1\times\{0\}$ in the standard contact open solid torus 
$(S^1\times\R^{2n},\tilde\xi_0)$ with 
\begin{equation*}
  \tilde\xi_0=\ker\lft\{(\cos\phi)d\tilde x_1-(\sin\phi)d\tilde y_1
  +\sum_{i=2}^n(x_idy_i-y_idx_i)\rgt\}, 
\end{equation*}
where $\tilde x_1,\tilde y_1$ are new coordinates replaced with $x_1,y_1$ 
that depend on the choice of the projection $\pi$. 
 With this coordinates, a trivialization 
of the conformal symplectic normal bundle $\csn(L(\pi(\gamma)),M)$ is given by 
\begin{equation*}
  \lft\{(\cos\phi)\uv{\tilde x_1}-(\sin\phi)\uv{\tilde y_1},\uv{x_2},\uv{y_2},
  \dots,\uv{x_n},\uv{y_n}\rgt\}. 
\end{equation*} 
 Like the $3$-dimensional case, the first element depends on the thickness 
of the tubular neighborhood the transverse circle. 
 Note that the boundary of the tubular neighborhood 
is a convex surface $S^{2n-1}\times S^1$ with a dividing set $S^{2n-2}\times S^1$.

\subsubsection{From a $\xi$-round hypersurface to a convex hypersurface}
\label{sec:xirdtocvx}

  Next, we perturb the given $\xi$-round hypersurface 
$H\cong S^{2n-1}\times S^1$ in a contact $(2n+1)$-dimensional manifold $(M,\xi)$ 
so that it becomes convex. 
 Although, for any surface in a contact $3$-manifold, 
there exists a convex surface that approximates it, 
the fact is not true in higher dimensions. 
 However, in this case, a $\xi$-round hypersurface $H\cong S^{2n-1}\times S^1$ 
modeled on the standard contact sphere $(S^{2n-1},\eta_0)$ is perturbed 
to a convex surface $\tilde H\cong S^{2n-1}\times S^1$  
with a dividing set $S^{2n-2}\times S^1\subset S^{2n-1}\times S^1$. 
 It is explained as follows. 

  Recall that, from Proposition~\ref{prop:xi-rdhyps}, 
the $\xi$-round hypersurface $H\subset(M,\xi)$ is identified with 
$V(\rho):=S^1\times\rd D^{2n}(\rho)=\lft\{\sum_{i=1}^nr_i^2=\rho^2\rgt\}
=\lft\{\sum_{i=1}^n(x_i^2+y_i^2)=\rho^2\rgt\}$ 
in $(S^1\times\R^{2n},\xi_0)$ with 
\begin{equation*}
  \xi_0=\ker\lft\{d\phi+\sum_{i=1}^nr_i^2d\theta_i\rgt\} 
  =\ker\lft\{d\phi+\sum_{i=1}^n(x_idy_i-y_idx_i)\rgt\} 
\end{equation*}
for some $\rho>0$,
where $(\phi,r_1,\theta_1,\dots,r_n,\theta_n)$ and 
$(\phi,x_1,y_1,\dots,x_n,y_n)$ 
are coordinates of $S^1\times\R^{2n}$ (see Subsection~\ref{sec:circsph}). 
 The characteristic foliation $H_\xi$ on $H\subset(M,\xi)$ 
consists of parallel closed leaves. 
 There exists a closed integral hypersurface 
$h\cong S^{2n-2}\times S^1\subset S^{2n-1}\times S^1\cong H$ of $H_\xi$ 
which cuts $H$ into two parts. 
 By pushing $H=V(\rho)\subset S^1\times\R^{2n}$ along $h\subset H$ slightly 
to the center, 
we can modify the characteristic foliation $H_{\xi_0}$ 
to $\tilde H_{\xi_0}$ on the modified hypersurface $\tilde H$ 
which is still non-singular. 
 From the normal form of $\xi_0$ above, the modified characteristic foliation 
$\tilde H_{\xi_0}$ is transverse 
to the hypersurface $\tilde h\subset\tilde H$ modified from $h\subset H$ 
(see Figure~\ref{fig:xirdtocvx}). 
 Then $\tilde H\cong S^{2n-1}\times S^1$ is a convex hypersurface 
with $\tilde h\cong S^{2n-2}\times S^1\subset\tilde H$ as a dividing set. 

%
%
\begin{figure}[htb]
  \centering
  {\small \includegraphics[height=3.3cm]{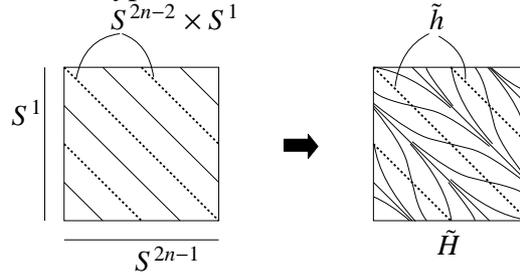}}
  \caption{$\xi$-round hypersurface to convex hypersurface}
  \label{fig:xirdtocvx}
\end{figure} 

%
%
\begin{rem}
  The procedure above is a generalization of 
the procedure that makes a pre-Lagrangian torus in a contact $3$-manifold 
convex with two parallel dividing curves. 
 In that case, the hypersurface is $H=S^1\times S^1$ 
and the dividing set is $\tilde h=S^0\times S^1$ 
(see Figure~\ref{fig:xirdtocvx}). 
\end{rem} 

\subsection{Generalized Lutz twist along a transverse circle} 
\label{sec:ltwtrvknt}
  The generalized Lutz twist along a transverse circle is realized 
by contact round surgeries of index~$1$ and~$2n$ in this subsection. 
 Let $\gamma$ be a transverse circle in a $(2n+1)$-dimensional contact manifold 
$(M,\xi)$. 
 We realize the generalized Lutz twist along $\gamma\subset(M,\xi)$ 
by contact round surgeries of index~$1$ and~$2n$ 
with the model $\pi$-Lutz tube $\lft(\mlt,\tilde\zeta\rgt)$. 
 Recall that the model $\pi$-Lutz tube is diffeomorphic, as a manifold, 
to $S^1\times D^{2n}$, which is constructed 
from $\lft(U(\sqrt{\pi}),\zeta\rgt)\subset\lft(S^1\times\R^{2n},\zeta\rgt)$, 
where $\zeta$ is the confoliation 
\begin{equation*}
  \zeta=\ker\lft\{\lft(\prod_{i=1}^{n}\cos r_i^2\rgt)d\phi
  +\sum_{i=1}^{n}\sin r_i^2d\theta_i\rgt\}. 
\end{equation*}
 The model $\pi$-Lutz tube $\lft(\mlt,\tilde\zeta\rgt)$ is constructed 
from the double of $\lft(U(\sqrt{\pi}),\zeta\rgt)$ 
by removing the standard tubular neighborhood 
of one of the transverse core $S^1\times\{0\}$ 
of $\lft(U(\sqrt{\pi}),\zeta\rgt)$ (see Subsubsection~\ref{sec:dbl-mlt}). 

  The rough story is as follows. 
 First, we operate the contact round surgery of index~$1$ 
along a certain isotropic circle 
close to the given transverse circle $\gamma\subset(M,\xi)$ 
and a certain isotropic circle 
close to the transverse core $\Gamma\subset\lft(\mlt,\tilde\zeta\rgt)$ 
of the model $\pi$-Lutz tube. 
 Then we operate the contact round surgery of index~$2n$ 
along a certain $\tilde\zeta$-round hypersurface $H\cong S^{2n-1}\times S^1$ 
isotopic as manifolds to $\rd\mlt$ but far from $\rd\mlt$ 
in the attached $\mlt$. 
 As a result, the given manifold $M$ is recovered to the original. 
 In addition, some part of the model $\pi$-Lutz tube is left 
in the manifold $M$. 
 In the following, we carefully describe these operations. 

  \underline{Operation~1.}\ 
 The first operation is the contact round surgery of index~$1$ 
between the given manifold $(M,\xi)$ 
and the model $\pi$-Lutz tube $\lft(\mlt,\tilde\zeta\rgt)$. 
 Recall that the contact round surgeries of index~$1$ is operated 
along a pair of isotropic circles. 
 We take isotropic push-offs 
of the given transverse curve $\gamma\subset(M,\xi)$ 
and the transverse core $\Gamma\subset\lft(\mlt,\tilde\zeta\rgt)$ 
(see Subsubsection~\ref{sec:istrpoff}). 
 Let $L(\gamma)$ and $L(\Gamma)$ denote them. 
 Now, we operate the contact round surgery of index~$1$ 
along isotropic circles $L(\gamma)\subset(M,\xi)$ 
and $L(\Gamma)\subset\lft(\mlt,\tilde\zeta\rgt)$. 
 Let $(M,\xi)\rsurg{L(\gamma),L(\Gamma)}\lft(\mlt,\tilde\zeta\rgt)$ 
denote the obtained contact manifold. 
 From the observation in Subsubsection~\ref{sec:ind1surg}, 
the operation corresponds to the removal of the standard tubular neighborhoods 
of isotropic circles $L(\gamma)\subset(M,\xi)$ 
and $L(\Gamma)\subset\lft(\mlt,\tilde\zeta\rgt)$,  
and the attachment of the invariant tubular neighborhood of the convex surface 
$S^{2n-1}\times S^1$ with a dividing set $S^{2n-2}\times S^1$. 

  \underline{Operation~2.}\ 
 The next operation is the contact round surgery of index~$2n$ 
of the obtained contact manifold 
$(M,\xi)\rsurg{L(\gamma),L(\Gamma)}\lft(\mlt,\tilde\zeta\rgt)$. 
 Recall that the contact round surgery of index~$2n$ is operated 
along a convex hypersurface $H=S^{2n-1}\times S^1$ 
with a dividing set $S^{2n-2}\times S^1$. 
 We take such a convex hypersurface close to the boundary 
$\rd\lft\{(M,\xi)\rsurg{L(\gamma),L(\Gamma)}
\lft(\mlt,\tilde\zeta\rgt)\rgt\}
=\rd\mlt$ 
so that they are homotopic. 
 In fact, the boundary $\rd\mlt$ of the model $\pi$-Lutz tube 
is the boundary of the removed standard tubular neighborhood 
of one of the transverse core of the double 
of $\lft(U(\sqrt{\pi}),\zeta\rgt)$, 
from the construction of the model $\pi$-Lutz tube 
(see Subsubsection~\ref{sec:dbl-mlt}). 
 In other words, the boundary $\rd\mlt$ 
is a $\tilde\zeta$-round hypersurface modeled on the standard contact sphere. 
 Then, by the argument in Subsubsection~\ref{sec:xirdtocvx}, 
we can perturb the boundary to have the required convex hypersurface 
$\tilde H=S^{2n-1}\times S^1$. 
 Now, we operate the contact round surgery of index~$2n$ 
along the convex hypersurface $\tilde H$. 
 From the observation in Subsubsection~\ref{sec:ind2nsurg}, 
the operation corresponds to the removal of the invariant tubular neighborhood 
of the convex hypersurface $\tilde H$ 
and the attachment of the standard tubular neighborhoods of isotropic circles. 

  We confirm that the operations amount to the generalized Lutz twist. 
 Throughout the operations, we can adjust the framings of surgeries, 
or the framings of the conformal symplectic normal bundles, 
by the choice of the projection $\pi$ 
and the thickness of the tubular neighborhoods 
(see Subsection~\ref{sec:approx}). 
 We can take the framing of the second contact round surgery of index~$2$ 
so that it recovers the model $\pi$-Lutz tube to the double 
$\lft(\du,\tilde\zeta\rgt)
=\lft(U(\sqrt{\pi}),\zeta\rgt)\cup\lft(U(\sqrt{\pi}),\zeta\rgt)$. 
And we can take the framing of the first contact round surgery of index~$1$ 
so that it does not change the manifold 
as the surgery between $(M,\xi)$ and the double. 
 Therefore, the operations can be regarded as the replacement 
of the tubular neighborhood of $\gamma\subset(M,\xi)$ 
with the double $\lft(\du,\tilde\zeta\rgt)$ 
without a tubular neighborhood of the transverse core. 
 The second object is nothing but the model $\pi$-Lutz tube. 
 Thus, we conclude that the operations amount to the generalized Lutz twist. 

%
%
\begin{rem}\label{rem:budsrg}
  The object obtained from the double 
$\lft(\du,\tilde\zeta\rgt)
=\lft(U(\sqrt{\pi}),\zeta\rgt)\cup\lft(U(\sqrt{\pi}),\zeta\rgt)$ 
by removing the tubular neighborhood of the both transverse cores 
is called the wide Giroux domain $\wgd$ (see Subsection~\ref{sec:circsph}). 
 Then the double $\du$, the model $\pi$-Lutz tube $\mlt$, 
and the wide Giroux domain $\wgd$ 
are related by the ``blowing up and down procedures''. 
 In that sense, the contact round surgeries of index~$1$ and~$2n$ 
have some relations with the blowing up and down, respectively. 
 In order to complete the Lutz twist, we should blow down 
the outer end of the Lutz tube. 
 That is the reason why we take the double 
$\lft(\du,\tilde\zeta\rgt)$ of $\lft(U(\sqrt{\pi}),\zeta\rgt)$ 
for the definition of the generalized Lutz twist 
(see also Subsubsection~\ref{sec:revrdsurgrep}). 

  From the observation above, 
instead of the contact round surgery of index~$2n$, 
we can use the double 
$\lft(\du,\tilde\zeta\rgt)$. 
 Because it is obtained from the model $\pi$-Lutz tube by blowing down 
the outer end. 
 In other words, the generalized Lutz twist along a transverse circle 
is the contact round surgery along the isotropic push-offs 
of the given transverse circle and one of the core transverse circles 
in the double (see Figure~\ref{fig:l1tw}). 
 Then compare Figure~\ref{fig:l1tw} with Figure~\ref{fig:cnstgltube}. 
%
%
\begin{figure}[htb]
  \centering
  {\small \includegraphics[height=3.3cm]{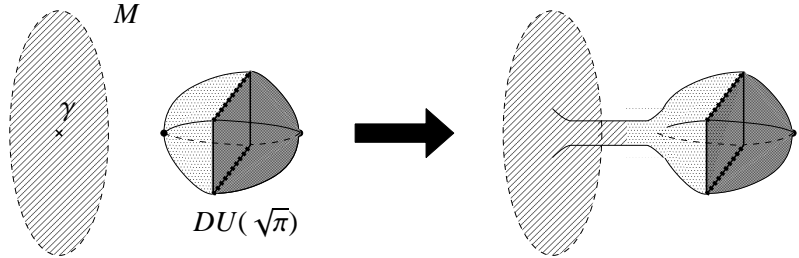}}
  \caption{Generalized Lutz twist along a circle $\pmod{\times S^1}$.}
  \label{fig:l1tw}
\end{figure}
\end{rem} 

\subsection{Generalized Lutz twist along a $\xi$-round hypersurface}
\label{sec:ggtorbyrsrg}
  The generalized Lutz twist along a $\xi$-round hypersurface is realized 
by contact round surgeries of index~$2n$ and~$1$ in this subsection. 
 Let $H\cong S^{2n-1}\times S^1$ be a $\xi$-round hypersurface 
modeled on the standard contact sphere $(S^{2n-1},\eta_0)$ 
in a $(2n+1)$-dimensional contact manifold $(M,\xi)$. 
 We realize the generalized Lutz twist along $H$ by contact round surgeries 
of index~$2n$ and~$1$. 

 The rough story is as follows. 
 First, we operate the contact round surgery of index~$2n$ 
along a certain convex hypersurface 
$\tilde H\cong S^{2n-1}\times S^1\subset(M,\xi)$ 
which is close to the given $\xi$-round hypersurface $H$. 
 By this operation, two tubular neighborhoods of isotropic circles 
are glued to the contact manifold (see Subsubsection~\ref{sec:ind2nsurg}). 
 We operate the generalized Lutz twist along a transverse push-off of 
one of the core isotropic circles. 
 In other words, by the arguments in Subsection~\ref{sec:ltwtrvknt}, 
this operation is realized by the contact round surgeries 
of index~$1$ and~$2n$. 
 Note that the model $\pi$-Lutz tube is left in the modified contact manifold. 
 Then, furthermore,  we operate the contact round surgery of index~$1$ 
along isotropic circles in the glued tubular neighborhoods. 
 As a result, the given manifold $M$ is recovered to the original one. 
 In addition, the wide $\pi$-Giroux domain is left in the manifold $M$. 
 In the following, we carefully describe these operations. 

  \underline{Operation~1.}\ 
 The first operation is the contact round surgery of index~$2n$. 
 Recall that the contact round surgery of index~$2n$ is operated 
along a convex hypersurface $S^{2n-1}\times S^1\subset(M,\xi)$ 
with a dividing set $S^{2n-2}\times S^1\subset S^{2n-1}\times S^1$ 
(see Subsubsection~\ref{sec:ind2nsurg}). 
 We perturb the given $\xi$-round hypersurface $H\subset(M,\xi)$ 
to a convex hypersurface. 
 From the argument in Subsubsection~\ref{sec:xirdtocvx}, 
we obtain a convex hypersurface $\tilde H\cong S^{2n-1}\times S^1$ 
with a dividing set $S^{2n-2}\times S^1$. 
 Then we operate the contact round surgery of index~$2n$ along $\tilde H$. 
 Recall that, from the description of the surgery obtained 
in Subsubsection~\ref{sec:ind2nsurg}, 
we glue the two standard tubular neighborhoods $(D^{2n}\times S^1,\eta)_i$ 
of isotropic circles $L_i=\{0\}\times S^1$, $i=1,2$, in this operation. 
 Let $(M_1,\xi_1)$ denote the obtained contact manifold. 

  \underline{Operation~2.}\ 
 Next, we operate the generalized Lutz twist 
along the isotropic circle $L_1\subset(M_1\xi_1)$. 
 Although the generalized Lutz twist is operated along a transverse circle, 
from the observation in Subsection~\ref{sec:ltwtrvknt}, 
it can be regarded as round surgeries along an isotropic push-off 
of the transverse circle. 
 Then we obtain a contact structure $\xi_1'$ on $M_1$. 
 Note that, as we observed in Remark~\ref{rem:budsrg}, 
the generalized Lutz twist 
can be regarded as the contact round surgery of index~$1$ 
along the isotropic circle $L_1\subset(M,\xi)$ 
and an isotropic push-off $\tilde l_1$ of one of the transverse cores 
$l_1,\ l_2$ of the double $\lft(\du,\tilde\zeta\rgt)$. 

  \underline{Operation~3.}\ 
 The final operation is the contact round surgery of index~$1$ 
along the isotropic circles $\tilde l_2, L_2\subset(M_1,\xi'_1)$, 
where $\tilde l_2$ is an isotropic push-off of the transverse core $l_2$. 
 We can take the surgery framing, or trivializations 
of the conformal symplectic normal bundles 
$\csn(\tilde l_2,(M_1,\xi'_1))$ and $\csn(L_2,(M_1,\xi'_1))$ 
so that, by the final contact round surgery of index~$1$, 
the manifold gets recovered to the original $M$. 
 In fact, since $\lft(\du,\tilde\zeta\rgt)$ is obtained from the double 
of $(U(\sqrt{\pi}),\zeta)\subset\lft(S^1\times\R^{2n},\zeta\rgt)$, 
we can take an isotropic push-off $\tilde l_2$ of the transverse core 
$l_2\subset\lft(U(\sqrt{\pi}),\zeta\rgt)\subset(M_1,\xi'_1)$ 
so that the trivialization of $\csn(\tilde l_2,\lft(U(\sqrt{\pi}),\zeta\rgt))
=\csn(\tilde l_2,(M_1,\xi'_1))$ 
is the same as that of 
$\csn(\tilde l_1,\lft(U(\sqrt{\pi}),\zeta\rgt))=\csn(L_1,\lft(M_1,\xi_1\rgt))$. 
 As the trivialization of $\csn(L_2,(M_1,\xi'_1))$ 
is the same as that of $\csn(L_2,(M,\xi_1))$, 
we take the one that corresponds to the trivialization 
of $\csn(L_1,(M_1,\xi_1))$ 
because $L_1$ and $L_2$ are separated by the contact round surgery of index~$2$ 
in Operation~1. 
 Then, by the contact round surgery of index~$1$ 
along isotropic circles $\tilde l_2,\ L_2\subset(M_1,\xi'_1)$ 
with the trivialization of $\csn(\tilde l_2,(M_1,\xi'_1))$ 
and $\csn(L_2,(M_1,\xi'_1))$ above, 
we obtain a manifold diffeomorphic to the original $M$ 
with a certain contact structure $\xi_2$ on $M$. 

  We confirm that the operations amount to the generalized Lutz twist 
along the $\xi$-round hypersurface $H\subset(M,\xi)$. 
 Although the manifold $M$ is not changed by Operations~$1$, $2$, and $3$, 
the obtained contact structure $\xi_2$ is modified from the original $\xi$. 
 In fact, in Operation~$3$, a generalized Lutz tube blown up 
along $\tilde l_2$, that is,  a wide Giroux domain is left in $(M,\xi_2)$. 
 Other part of $M$ is not changed. 
 These imply that these operations of contact round surgeries amount 
to the generalized Lutz twist along a $\xi$-round hypersurface $H$. 

%
%
\begin{rem}
  As we observed in Remark~\ref{rem:budsrg}, 
we can describe the generalized Lutz twist 
form the view point of the double 
$\lft(\du,\tilde\zeta\rgt)=(U(\sqrt{\pi}),\zeta)\cup(U(\sqrt{\pi}),\zeta)$. 
 In the operations above, we used the generalized Lutz twist along a circle. 
 Instead of the contact round surgery of index~$2n$ 
in the generalized Lutz twist, or a blowing down procedure, 
we use the double $\lft(\du,\tilde\zeta\rgt)$. 
 In other words, we can describe the generalized Lutz twist 
along a $\xi$-round hypersurface in the following two steps. 
 The first operation is the same as above. 
 The second operation consists of two contact round surgeries of index~$1$ 
along two pairs 
of isotropic push-offs of the core transverse curves of the double 
and isotropic cores attached in the first operation 
(see Figure~\ref{fig:l2ntw}). 
 Then compare Figure~\ref{fig:l2ntw} with Figure~\ref{fig:wgd}. 
%
%
\begin{figure}[htb]
  \centering
  {\small \includegraphics[height=4.8cm]{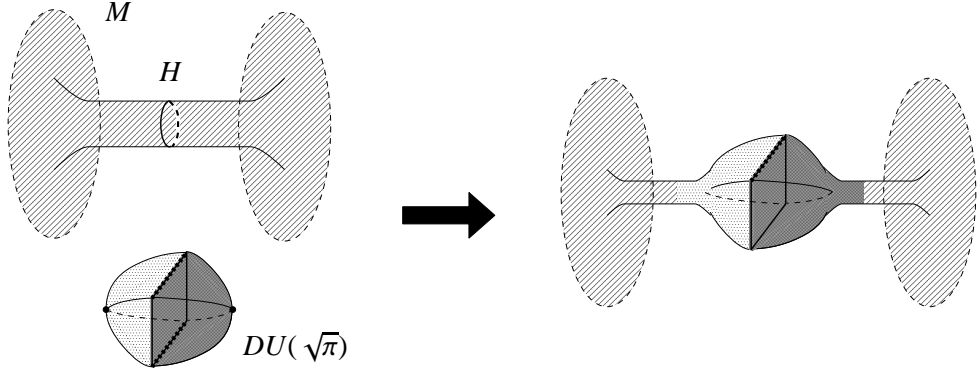}}
  \caption{Generalized Lutz twist along a hypersurface $\pmod{\times S^1}$.}
  \label{fig:l2ntw}
\end{figure}
\end{rem} 

\section{Overtwisted disc}\label{sec:newot}
  When this paper was being prepared, 
a new definition of overtwisted disc for all odd dimensions 
was announced by Borman, Eliashberg, and Murphy~\cite{newOT}. 
 The overtwistedness includes other candidates, 
such as bordered Legendrian open book, 
of overtwistedness in higher dimensions. 
 We introduce the definition, 
and discuss the relation 
between the generalized Lutz twist introduced in this paper 
and the new overtwisted disc in this section. 

\subsection{Overtwisted disc in all dimensions}\label{sec:newdefot}
  First we introduce the definition of overtwisted disc following~\cite{newOT}. 
 An overtwisted disc in a $(2n+1)$-dimensional contact manifold 
is a certain piecewise smooth $2n$-dimensional disc 
with a germ of contact structures along it. 
 In order to define it precisely, we need some notions. 
 Let $\Delta_{\textup{cyl}}\subset\R^{2n-1}$ be a domain defined as 
\begin{equation*}
  \Delta_{\textup{cyl}}:=\lft\{(z,r_1,\phi_1,\dots,r_{n-1},\phi_{n-1})\in\R^{2n-1}
  \;\lft|\;|z|\le1,\ \sum_{i=1}^{n-1}r_i^2\le1\rgt.\rgt\}. 
\end{equation*}
 Let $K_\epsilon\colon\Delta_{\textup{cyl}}\to\R$ be a function which satisfies: 
\begin{equation*}
  K_\epsilon(z,r_i,\phi_i)=
  \begin{cases}
    k_\epsilon(\sum r_i^2)+k_\epsilon(|z|) & 
    (\text{on } \Delta_{\textup{cyl}}\setminus\Delta_\epsilon), \\
    <0 & (\text{on } \inter\Delta_\epsilon), 
  \end{cases}
\end{equation*}
where 
$\Delta_\epsilon:=\lft\{(z,r_i,\phi_i)\;\lft|\;
|z|\le1-\epsilon,\ \sum r_i^2\le1-\epsilon\rgt.\rgt\}$, 
and $k_\epsilon\colon\R_{\ge0}\to\R$ is a function defined as 
\begin{equation*}
  k_\epsilon(s):=
  \begin{cases}
    0 & (s\le1-\epsilon), \\
    s-(1-\epsilon) & (s\ge1-\epsilon). 
  \end{cases}
\end{equation*}
 For such a function $K_\epsilon$ and a constant $C>-\min(K_\epsilon)$, set 
\begin{equation*}
  B^{S^1}_{K_\epsilon,C}:=\lft\{(z,r_i,\phi_i,r,\phi)\in\dcyl\times\mathbb{C}
  \;\lft|\;r^2\le K_\epsilon(z,r_i,\phi_i)+C\rgt.\rgt\}
  \subset\R^{2n-1}\times\mathbb{C}. 
\end{equation*}
 Furthermore, for the function $K_\epsilon\colon\Delta_{\textup{cyl}}\to\R$, 
there exists a family of functions 
$\rho_{(z,r_i,\phi_i)}\colon\R_{\ge0}\to\R$ 
for parameter $(z,r_i,\phi_i)\in\Delta_{\textup{cyl}}$ 
that satisfies 
\begin{enumerate}
\item $\rho_{(z,r_i,\phi_i)}(s)=s$ if $s\in\mathcal{O}p\{0\}\subset\R_{\ge0}$, 
\item $\rho_{(z,r_i,\phi_i)}(s)=s-C$ if $(z,r_i,\phi_i,s,\phi)
  \in\mathcal{O}p\lft\{s\ge K_\epsilon(z,r_i,\phi_i)+C\rgt\}
  \subset\Delta_{\textup{cyl}}\times\mathbb{C}$, 
\item $\frac{\rd}{\rd s}\rho_{(z,r_i,\phi_i)}(s)>0$ 
  if $(z,r_i,\phi_i)\in\mathcal{O}p(\rd\Delta_{\textup{cyl}})
  \subset\Delta_{\textup{cyl}}$, 
\end{enumerate}
where $\mathcal{O}p$ implies open neighborhood. 
 Then, for this family $\rho_{(z,r_i,\phi_i)}(s)$, set 
\begin{equation}\label{eq:otwform}
  \alpha_\rho:=dz+\sum_{i=1}^{n-1}r_i^2d\phi_i+\rho_{(z,r_i,\phi_i)}(r^2)d\phi, 
  \qquad \eta^{S^1}_{K_\epsilon,\rho}:=\ker\alpha_\rho. 
\end{equation}
 Then $\alpha_\rho$ is a $1$-form on $\Delta_{\textup{cyl}}\times\mathbb{C}$, 
and $\eta^{S^1}_{K_\epsilon,\rho}$ is a hyperplane distribution 
on $\Delta_{\textup{cyl}}\times\mathbb{C}$. 
 It is proved in~\cite{newOT} 
that $\eta^{S^1}_{K_\epsilon,\rho}$ is an almost contact structure 
on $B^{S^1}_{K_\epsilon,C}$ which is genuine contact 
near the boundary $\rd B^{S^1}_{K_\epsilon,C}$. 

  Now, overtwisted disc is defined as follows. 
 Let $D_{K_\epsilon}$ be the $2n$-dimensional disc defined as 
\begin{equation*}
  D_{K_\epsilon}:=\lft\{\lft.(z,r_i,\phi_i,r,\phi)\in\rd B^{S^1}_{K_\epsilon,C}
\;\rgt|\;z\in[-1,1-\epsilon]\rgt\}\subset\R^{2n-1}\times\mathbb{C}=\R^{2n+1}. 
\end{equation*}
 There exits a germ of contact structure $\eta_{K_\epsilon}$ along $D_{K_\epsilon}$
by restricting $\eta^{S^1}_{K_\epsilon,\rho}$. 
 For a sufficiently small $\epsilon>0$, 
the pair $(D_{K_\epsilon},\eta_{K_\epsilon})$ 
of a disc and a germ of contact structures is called an \emph{overtwisted disc} 
(see~\cite{newOT} for the precise estimate of $\epsilon$). 

  The left hand side of Figure~\ref{fig:otwdisc} is an overtwisted disc 
in a contact $3$-manifold. 
%
%
\begin{figure}[htb]
  \centering
  {\small \includegraphics[height=3.6cm]{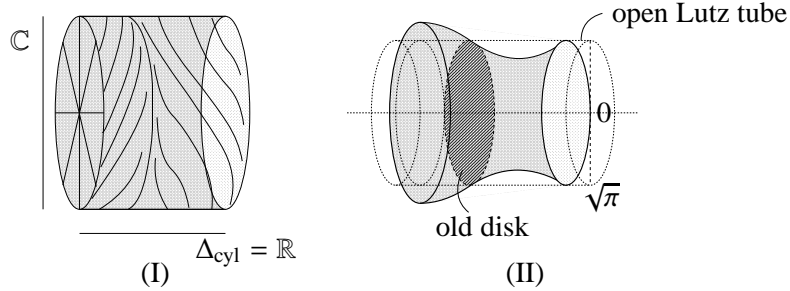}}
  \caption{Overtwisted disc in dimension~$3$}
  \label{fig:otwdisc}
\end{figure} 
 It is realized in a neighborhood of an ``overtwisted disk'' in an old sense 
defined in Subsection~\ref{sec:ovtw} 
(see the right hand side of Figure~\ref{fig:otwdisc}). 
 From this observation, we have an $S^1$-family of overtwisted discs 
in the open Lutz tube $\lft(S^1\times\R^2,\ker\{\cos r^2dx+\sin r^2dy\}\rgt)$. 
 We generalize this observation to higher dimensions 
in the following subsection. 

  Concerning the overtwisted contact structures, 
the following property is proved in~\cite{newOT}. 
%
%
\begin{thrm}[Borman, Eliashberg, Murphy]
  Overtwisted contact structures on a manifold 
which are homotopic as almost contact structures 
are isotopic to each other. 
\end{thrm} 
\subsection{Relation between the generalized Lutz twist 
  and overtwisted disc}
\label{sec:gltw-ot}
  In this subsection, 
we show that there exists an $S^1$-family of overtwisted discs 
in the model $\pi$-Lutz tube $\lft(\mlt,\tilde\zeta\rgt)$ 
defined in Subsection~\ref{sec:defGenLztw}. 
 The claim is as follows. 
%
%
\begin{prop}\label{prop:mlttootw}
  There exists an $S^1$-family of $2n$-dimensional overtwisted discs 
in the $(2n+1)$-dimensional model Lutz tube $\lft(\mlt,\tilde\zeta\rgt)$. 
\end{prop} 

  From Proposition~\ref{prop:mlttootw} 
and the construction of the generalized Lutz twist 
in Section~\ref{sec:genlztw},
we obtain the overtwistedness in Theorem~\ref{thm_main}. 
 In fact, the model $\pi$-Lutz tube is used 
to define the generalized Lutz twist along a transverse circle 
(see Section~\ref{sec:genlztw}). 
 Therefore, a generalized Lutz twist along a transverse circle 
creates an $S^1$-family of overtwisted discs. 

  Now, we show Proposition~\ref{prop:mlttootw}. 
\begin{proof}[Proof of Proposition~\ref{prop:mlttootw}.]
  Recall that the model $\pi$-Lutz tube $\lft(\mlt,\tilde\zeta\rgt)$ 
is constructed from the double of $\lft(U(\sqrt{\pi}),\zeta\rgt)$ 
in the generalized open Lutz tube 
$\lft(S^1\times\R^{2n}=S^1\times\C^n,\zeta\rgt)$. 
 Then we discus on $\lft(U(\sqrt{\pi}),\zeta\rgt)$. 
 Further, taking a point in $S^1\times\{0\}$, 
we discuss on 
$\lft(U(\sqrt{\pi}),\zeta\rgt)\subset\lft(\R\times\R^{2n},\zeta\rgt)$. 

  First, we review the realization of an overtwisted disc in dimension~$3$. 
 For an overtwisted disc $(D_{K_\epsilon},\eta_{K\epsilon})$, 
there exists a function $g\colon[-1,1-\epsilon]\to\R$ 
for which the disc 
\begin{equation*}
  D_{\textup{ot}}^2
  :=\lft\{(z,r,\phi)\in\R\times\R^2\mid z\in[-1,1-\epsilon], r=g(z)\rgt\}
  \cup\lft\{(z,r,\phi)\in\R\times\R^2\mid z=-1,\ r\le g(-1)\rgt\}. 
\end{equation*}
with the germ of contact structure from $\ker\{\cos r^2dz+\sin r^2d\phi\}$ 
is isomorphic to $(D_{K_\epsilon},\eta_{K\epsilon})$, 
like in Figure~\ref{fig:otwdisc}. 
 Note that $g(-1+\epsilon)=g(1-\epsilon)=\sqrt{\pi}$, and $g(-1)<\sqrt{\pi/2}$.

  We generalize the discussion to higher dimensions. 
 For the function $\tilde K_\epsilon\colon\dcyl\to\R$ defined as 
$\tilde K_\epsilon(z,r_i,\phi_i)=K_\epsilon(z)+K_\epsilon(\sum r_i^2)$, 
we have an overtwisted disc $\lft(D_{\tilde K_\epsilon},\eta_{\tilde K_\epsilon}\rgt)$.
 From the construction of the contact structure $\eta_{\tilde K_\epsilon}$, 
or Equations~\eqref{eq:otwform}, 
it corresponds to the disc 
\begin{align*}
  D_{\textup{ot}}^{2n}:=&\lft\{(z,r_i,\phi_i,r,\phi)
  \in\R\times\R^{2(n-1)}\times\R^2\;\lft|\;
  z\in[-1,1-\epsilon],\ r=g(z),\ \sum_{i=1}^{n-1}r_i^2\le g(-1)^2\rgt.\rgt\} \\
  &\cup\lft\{(z,r_i,\phi_i,r,\phi)
  \in\R\times\R^{2(n-1)}\times\R^2\;\lft|\;
  z\in[-1,1-\epsilon],\ r\le g(z),\ \sum_{i=1}^{n-1}r_i^2=g(-1)^2\rgt.\rgt\} \\
  &\cup\lft\{(z,r_i,\phi_i,r,\phi)
  \in\R\times\R^{2(n-1)}\times\R^2\;\lft|\;
  z=-1,\ r\le g(-1),\ \sum_{i=1}^{n-1}r_i^2\le g(-1)^2\rgt.\rgt\}
\end{align*}
with the germ of contact structures obtained from $\zeta$
(see Figure~\ref{fig:hdotwd}). 
%
%
\begin{figure}[htb]
  \centering
  {\small \includegraphics[height=3.6cm]{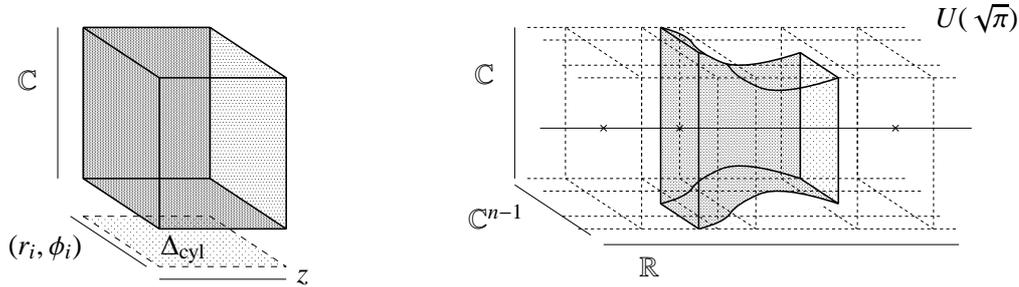}}
  \caption{Overtwisted disc in higher dimensions}
  \label{fig:hdotwd}
\end{figure} 
 Note that $D_{\textup{ot}}^{2n}$ is in the contact locus of 
$\lft(S^1\times\R^{2n},\zeta\rgt)$ because $g(-1)<\sqrt{\pi/2}$. 

  We confirm that the overtwisted disc is 
in the model $\pi$-Lutz tube $\lft(\mlt,\tilde\zeta\rgt)$. 
 The obtained disc $D_{\textup{ot}}^{2n}$ is not in $U(\sqrt{\pi})$ 
when $z\in(-1+\epsilon,\ 1-\epsilon)$. 
 However, we take the double $\lft(\du,\zeta'\rgt)$ 
of $\lft(U(\sqrt{\pi}),\zeta\rgt)$ 
in order to define the model $\pi$-Lutz tube. 
 Then $D_{\textup{ot}}^{2n}$ is in $\lft(\du,\zeta'\rgt)$. 
 Since the disc $D_{\textup{ot}}^{2n}$ lies in the contact locus 
of the confoliation $\zeta'$, 
it still exists in $\du$ after perturbing $\zeta'$ 
to a contact structure $\tilde\zeta$, 
from the discussion like in Subsubsection~\ref{sec:dbl-mlt}. 
 As a result, the disc $D_{\textup{ot}}^{2n}$ is in the model $\pi$-Lutz tube 
$\lft(\mlt,\tilde\zeta\rgt)$. 

  The discussion above is valid at any point 
in $S^1\times\{0\}\subset\lft(\mlt,\tilde\zeta\rgt)$. 
 Therefore, we have an $S^1$-family of overtwisted discs 
in the model $\pi$-Lutz tube $\lft(\mlt,\tilde\zeta\rgt)$. 
\end{proof}

  We should remark a relation between the generalized Lutz twist 
along a $\xi$-round hypersurface and the higher-dimensional overtwisted discs. 
 As we mentioned in Remark~\ref{rem:gltxird-diff}, 
this operation makes an $S^1$-family of bordered Legendrian open books 
except when the dimension of the manifold is $3$. 
 However, this operation does not make overtwisted discs directly 
in any dimension. 
 Because the operation is defined by inserting the wide Giroux domain 
$\lft(\wgd,\tilde\zeta\rgt)$. 
 The wide Giroux domain is constructed from the double of 
$(U(\sqrt{\pi}),\zeta)$ by removing neighborhoods of the both transverse cores, 
which intersect overtwisted discs (see Figure~\ref{fig:hdotwd}).

%
%

\end{document}